\documentclass[reqno,11pt]{amsart}

\usepackage[english]{babel}
\usepackage[utf8]{inputenc}
\usepackage[T1]{fontenc}

\usepackage[margin=1in, letterpaper]{geometry}
\usepackage{amssymb,amsmath,amsfonts,amsthm}
\usepackage{bm,bbm,mathrsfs}
\usepackage[dvipsnames]{xcolor}
\usepackage{url}
\usepackage{enumitem}
\usepackage{mathtools}
\usepackage{placeins}
\usepackage{extarrows}
\usepackage{subfig}
\usepackage[font=small,labelfont=bf]{caption}
\usepackage{floatrow}
\usepackage{verbatim}
\usepackage{wrapfig}
\usepackage{tikz}
\usetikzlibrary{arrows,arrows.meta,cd,decorations.pathmorphing,backgrounds,positioning,fit,matrix}
\usepackage{xcolor,colortbl,graphics,graphicx}
\usepackage[colorlinks=true, allcolors=Blue, pdfencoding=auto]{hyperref}
\usepackage[noabbrev,capitalize]{cleveref}
\crefname{equation}{}{}

\newtheorem{theorem}{Theorem}[section]
\newtheorem{proposition}[theorem]{Proposition}
\newtheorem{lemma}[theorem]{Lemma}
\newtheorem{corollary}[theorem]{Corollary}
\newtheorem{conjecture}[theorem]{Conjecture}

\Crefname{problem}{Problem}{Problems}
\newtheorem*{question*}{Question}
\Crefname{question}{Question}{Questions}
\newtheorem{assumption}[theorem]{Assumption}

\Crefname{knownlemma}{Lemma}{Lemmas}
\Crefname{knowntheorem}{Theorem}{Theorems}
\newtheorem{definition}[theorem]{Definition}

\Crefname{notation}{Notation}{Notations}
\newtheorem*{notation*}{Notation}

\newtheorem*{example*}{Example}

\theoremstyle{remark}
\newtheorem*{remark}{Remark}

\numberwithin{equation}{section}
\allowdisplaybreaks

\newcommand{\Z}{\mathbb{Z}}
\newcommand{\R}{\mathbb{R}}

\newcommand{\C}{\mathbb{C}}

\newcommand{\bD}{\mathbf{D}}
\newcommand{\mDw}{\mathcal{D}^{\textnormal{well}}}
\newcommand{\bDw}{\mathbf{D}^{\textnormal{well}}}
\renewcommand{\H}{\mathbb{H}}

\newcommand{\NN}{\mathcal{N}}

\newcommand{\RR}{\mathcal{R}}
\newcommand{\msE}{\mathscr{E}}
\newcommand{\msI}{\mathscr{I}}

\newcommand{\msW}{\mathscr{W}}
\newcommand{\msN}{\mathscr{N}}

\newcommand{\mB}{\mathcal{B}}
\newcommand{\mC}{\mathcal{C}}
\newcommand{\mD}{\mathcal{D}}

\newcommand{\mI}{\mathcal{I}}
\newcommand{\mK}{\mathcal{K}}

\newcommand{\mN}{\mathcal{N}}

\newcommand{\mR}{\mathcal{R}}
\newcommand{\mS}{\mathcal{S}}

\newcommand{\ma}{\mathfrak{a}}
\newcommand{\mb}{\mathfrak{b}}

\newcommand{\one}{\mathbbm{1}}

\newcommand{\GL}{\textnormal{GL}}

\newcommand{\PSL}{\textnormal{PSL}}

\newcommand{\eps}{\varepsilon}

\renewcommand{\bar}{\overline}
\renewcommand{\hat}{\widehat}
\renewcommand{\tilde}{\widetilde}
\newcommand{\cond}{\textnormal{cond}}

\renewcommand{\pmod}[1]{\ (\textnormal{mod } #1)}
\newcommand{\lcm}{\textnormal{lcm}}

\definecolor{myBlue}{rgb}{0, 0, 0.6}

\title{On the exponents of distribution of primes and smooth numbers}
\author[Alexandru Pascadi]{Alexandru Pascadi}
\address{Mathematical Institute, Radcliffe Observatory quarter, Woodstock Road, Oxford OX2 6GG, England}
\email{alexpascadi@gmail.com}

\begin{document}

\begin{abstract}
    We show that both primes and smooth numbers are equidistributed in arithmetic progressions to moduli up to $x^{5/8 - o(1)}$, using triply-well-factorable weights for the primes (we also get improvements for the well-factorable linear sieve weights). This completely eliminates the dependency on Selberg's eigenvalue conjecture in previous works of Lichtman and the author, which built in turn on results of Maynard and Drappeau. We rely on recent large sieve inequalities for exceptional Maass forms of the author for additively-structured sequences, and on a related result of Watt for multiplicatively-structured sequences. As applications, we prove refined upper bounds for the counts of twin primes and consecutive smooth numbers up to $x$.
\end{abstract}

\maketitle

\vspace{-0.6cm}
{
\setlength{\parskip}{0em}
\setcounter{tocdepth}{1}
\tableofcontents
}
\vspace{-0.7cm}

\section{Introduction} \label{sec:intro}
Let $q$ be a large positive integer, $a \in \Z$ have $(a, q) = 1$, and $A > 0$. The Siegel--Walfisz theorem gives a pointwise asymptotic for the number of primes up to $x$ which are congruent to $a$ modulo $q$,
\[
    \pi(x, q; a) \sim \frac{\pi(x)}{\varphi(q)}, \qquad \text{ as } x \to \infty, \text{ for } q \le (\log x)^A,
\] 
where $\pi(x) := \#\{\text{prime } p \le x\}$ and $\pi(x, q; a) := \#\{\text{prime } p \le x : p \equiv a \pmod{q}\}$. The small range of moduli $q \le (\log x)^A$ is an obstruction to many applications, and can be improved substantially to $q \le x^{1/2}(\log x)^{-B}$ assuming the Generalized Riemann Hypothesis (GRH).
Unconditionally, the celebrated Bombieri--Vinogradov theorem \cite{bombieri1965large,vinogradov1965density} achieves the same range of moduli in an on-average setting: for $B = B(A)$ large enough in terms of $A$, one has
\begin{equation} \label{eq:bv}
    \sum_{\substack{q \le x^{1/2}(\log x)^{-B} \\ (q, a) = 1}} 
    \left\vert \pi(x; q, a) - \frac{\pi(x)}{\varphi(q)}\right\vert \ll_A
    \frac{x}{(\log x)^A}.
\end{equation}
(In fact, a stronger statement holds true, with a maximum over $a \in (\Z/q\Z)^\times$ inside the sum.)
This result has been critical to sieve theory methods and their applications, for instance to results on small gaps between primes \cite{maynard2015small,polymath2014variants}. Overcoming the square-root barrier at $q < x^{1/2}$, i.e., going ``beyond GRH'' on average, remains a central open problem in analytic number theory. Elliot--Halberstam \cite{elliott1968conjecture} conjectured that the same estimate holds true in the optimal range of moduli $q \le x^{1-\eps}$, and Polymath8b \cite{polymath2014variants} showed that a generalization of this conjecture would imply the existence of infinitely many pairs of primes with distance at most $6$.

Since the pioneering work of Fouvry \cite{fouvry1984autour,fouvry1987autour,fouvry1985probleme,fouvry1982repartition}, Fouvry--Iwaniec \cite{fouvry1980theorem,fouvry1983primes}, and Bombieri--Friedlander--Iwaniec \cite{bombieri1986primes,bombieri1987primes2,bombieri1989primes3}, we have been able to overcome this square-root barrier in special settings \cite{zhang2014bounded,maynard2025primes,maynard2025primes2,maynard2025primes3,stadlmann2023primes} -- in particular, by replacing the absolute values in \cref{eq:bv} with special weights $(\lambda_q)$ that arise in sieve theory applications. If such an analogue of \cref{eq:bv} holds with a weighted sum over all $q \le x^\vartheta$, for any fixed residue $a$, then we say that the primes have \emph{exponent of distribution} $\vartheta \in (0, 1)$ (or \emph{level of distribution} $x^\vartheta$) with respect to the weights $(\lambda_q)$.

Motivated by a `well-factorable' variant of the linear sieve weights \cite{iwaniec1980new,fouvry1983primes}, Bombieri--Friedlander--Iwaniec \cite[Theorem 10]{bombieri1986primes} considered sequences $(\lambda_q)$ that can be expressed as a Dirichlet convolution of two sequences of any pre-specified lengths, and achieved an exponent of distribution of $\tfrac{4}{7}-\eps$ in this setting. More recently, Maynard \cite{maynard2025primes2} considered the refined setting of `triply-well-factorable' weights, which we recall from \cite[Definition 2]{maynard2025primes2}.

\begin{definition}[Triply-well-factorable weights \cite{maynard2025primes2}] \label{def:triply-well-fact}
A complex sequence $(\lambda_q)_{q \le Q}$ is said to be \emph{triply-well-factorable} of level $Q$ iff for any $Q_1, Q_2, Q_3 \ge 1$ with $Q_1 Q_2 Q_3 = Q$, there exist $1$-bounded complex sequences $(\alpha_{q_1})$, $(\beta_{q_2})$, $(\gamma_{q_3})$ supported on $q_i \le Q_i$, such that for all $q$,
\[
    \lambda_q
    =
    \sum_{q_1 q_2 q_3 = q} \alpha_{q_1} \beta_{q_2} \gamma_{q_3}.
\]
\end{definition}

For such weights, which arise in a slight variant of the $\beta$-sieve with $\beta \ge 2$, Maynard \cite[Theorem 1.1]{maynard2025primes2} achieved the exponent of distribution $\tfrac{3}{5} - \eps$ (i.e., with a level $Q = x^{3/5-\eps}$). His results also implied an improved exponent of $\tfrac{7}{12}-\eps$ for the well-factorable variant of the upper-bound linear sieve weights (the case $\beta = 1$), which are close to being triply-well-factorable.

Essentially all such results are based on equidistribution estimates for convolutions of sequences in arithmetic progressions, proven using Linnik's dispersion method \cite{linnik1963dispersion}. Ultimately, these rely on bounding sums of Kloosterman sums via the spectral theory of automorphic forms, following Deshouillers--Iwaniec \cite{deshouillers1982kloosterman}. In this context, Lichtman \cite{lichtman2023primes} used optimized Deshouillers--Iwaniec-style estimates, via Kim--Sarnak's bound \cite{kim2003functoriality}, to improve the exponent of distribution for triply-well-factorable weights to $\tfrac{66}{107} - \eps \approx 0.6168$ unconditionally, and up to $\tfrac{5}{8} - \eps = 0.625 - \eps$ assuming:

\begin{conjecture}[Selberg, 1965 \cite{selberg1965estimation}] \label{conj:selberg}
    All eigenvalues of the hyperbolic Laplacian on Maass forms for congruence subgroups $\Gamma_0(q)$ are at least equal to $1/4$.
\end{conjecture}

Our goal in this work is to completely eliminate the dependency on Selberg's conjecture in several exponent-of-distribution results. For the primes, parts $(i)$ and $(ii)$ of the result below improve on the previous exponents of $\tfrac{66}{107}$ due to Lichtman \cite{lichtman2023primes}, respectively $\tfrac{7}{12}$ due to Maynard \cite{maynard2015small}.

\begin{theorem}[Primes in APs to large moduli] \label{thm:primes}
    Let $a \in \Z \setminus \{0\}$, $A, \eps > 0$, $x \ge 2$. Assume either:
    \begin{itemize} 
        \item[$(i)$.] $Q \le x^{5/8-\eps}$, and $(\lambda_q)$ are triply-well-factorable weights of level $Q$, or
        \item[$(ii)$.] $Q \le x^{3/5-\eps}$, and $(\lambda_q)$ are the upper-bound well-factorable  linear sieve weights of level $Q$.
    \end{itemize}
    (See \cref{def:linear-sieve} for part $(ii)$.) Then one has
    \begin{equation} \label{eq:primes-in-ap}
        \sum_{\substack{q \le Q \\ (q, a) = 1}} 
        \lambda_q \left(\pi(x; q, a) - \frac{\pi(x)}{\varphi(q)}\right) \ll_{\eps,A,a}
        \frac{x}{(\log x)^A}.
    \end{equation}
\end{theorem}

Moreover, in \cref{thm:primes-ls-weights}, we obtain a similar result applicable to both the upper-bound and the lower-bound well-factorable linear sieve weights, with a variable exponent of distribution depending on the factorization of the modulus $q$; this refines \cite[Proposition 6.6]{lichtman2023primes}.
As a consequence, we deduce a sharper upper bound for the number of twin primes up to $x$.
\textit{}
\begin{corollary}[Count of twin primes] \label{cor:twin-primes}
As $x \to \infty$, one has
\[
    \#\{p \le x : p, p+2 \text{ are prime}\} \le (3.203 + o(1))\, \Pi_2(x),
\]
where $\Pi_2(x) := \tfrac{2x}{(\log x)^2} \prod_{p > 2} \tfrac{1 - 2/p}{(1 - 1/p)^2}$ is the asymptotic predicted by Hardy--Littlewood \cite{hardy1923problems}.
\end{corollary}

This improves the constant of $3.229$ from \cite[Theorem 1.1]{lichtman2023primes}; we point the reader to \cite[p.\,2]{lichtman2023primes} for a table of previous results. Once again, the key qualitative feature of \cref{cor:twin-primes} is that it cannot be improved directly by assuming Selberg's conjecture.

The analogous equidistribution problem for \emph{smooth} (friable) numbers \cite{fouvry1996repartition,drappeau2015theoremes} concerns the quantities
\begin{equation} \label{eq:smooth-notations}
\begin{gathered}
    \Psi(x, y) := \# \{n \le x : P^+(n) \le y\},
    \qquad\quad 
    \Psi_q(x, y) := \# \{n \le x : P^+(n) \le y,\ (n, q) = 1\},
    \\
    \Psi(x, y; a, q) := \# \{n \le x : P^+(n) \le y,\ n \equiv a \pmod{q}\},
\end{gathered}
\end{equation}
where $P^+(n)$ denotes the largest prime factor of $n$; here $y \le x^{1/C}$, with $C$ large. In this context, Granville \cite[Theorem 2]{granville1993integers} proved a suitable analogue of the Bombieri--Vinogradov theorem, achieving the exponent of distribution $\tfrac{1}{2}-\eps$ (see also \cite{wolke1973mittlereI,wolke1973mittlereII,fouvry1991entiers,granville1993integers2}). Relying on a triple convolution estimate of Bombieri--Friedlander--Iwaniec \cite[Theorem 4]{bombieri1986primes}, Fouvry--Tenenbaum \cite{fouvry1996repartition} raised the exponent to $\tfrac{3}{5}-\eps$, with an upper bound of $x (\log x)^{-A}$ as in \cref{eq:bv}. Drappeau later \cite{drappeau2015theoremes} strengthened the bound back to $\Psi(x, y)(\log x)^{-A}$, with the same exponent of $\tfrac{3}{5}-\eps$. We remark that all of these results use absolute values (i.e., arbitrary $1$-bounded weights $\lambda_q$), which is possible beyond the square-root barrier due to the flexible factorization properties of smooth numbers.

Using a different arrangement of exponential sums and optimized Deshouillers--Iwaniec-style estimates, the author \cite{pascadi2023smooth} recently showed that smooth numbers have exponent of distribution $\tfrac{66}{107} - \eps \approx 0.6168$, and up to $\tfrac{5}{8} - \eps = 0.625 - \eps$ assuming Selberg's eigenvalue conjecture. 
As in the case of primes, we can now fully close the gap between the conditional and unconditional results.

\begin{theorem}[Smooth numbers in APs to large moduli]
\label{thm:smooth}
Let $a \in \Z \setminus \{0\}$ and $A, \eps > 0$, $x \ge 2$. Then there exists a large enough $C = C(a, A, \eps) > 0$ such that for any $y \in [(\log x)^C, x^{1/C}]$ and $Q \le x^{5/8 - \eps}$, one has
\[
    \sum_{\substack{q \le Q \\ (q, a) = 1}}
    \left\vert 
    \Psi(x, y; a, q) - 
    \frac{\Psi_q(x, y)}{\varphi(q)} \right\vert 
    \ll_{\eps,A,a} 
    \frac{\Psi(x, y)}{(\log x)^A}.
\]
\end{theorem}

\begin{remark}
Following Drappeau--Granville--Shao \cite{drappeau2017smooth}, one can deduce a similar result for smooth-supported multiplicative functions in arithmetic progressions, using a slight extension of our triple convolution estimate (\cref{prop:triple-convo}).
\end{remark}

Moreover, in \cref{thm:smooth-mod-smooth}, we prove a similar result with a slightly-better saving when the sum over $q$ is supported on smooth moduli; this refines a result of de la Bret\`eche--Drappeau \cite[(2.1)]{de2020niveau}. As a consequence, we improve the exponent of $\tfrac{3}{5}$ in \cite[Th\'eor\`eme 4.1]{de2020niveau} to $\tfrac{5}{8}$ in \cref{cor:factorable-quadratic}, which includes the following upper bound for the number of consecutive smooth numbers up to $x$.

\begin{corollary}[Count of consecutive smooth numbers] \label{cor:consec-smooth}
    For any $\eps > 0$ there exists $C > 0$ such that for any $x \ge 2$ and $y \in [(\log x)^C, x^{1/C}]$, one has
    \[
        \#\{n \le x : P^+(n), P^+(n+1) \le y\} \ll_\eps x \varrho(u)^{1 + 5/8-\eps},
    \]
    where $u := (\log x)/\log y$ and $\varrho$ denotes the Dickman function \cite{hildebrand1986number}.
\end{corollary}

We note that $\tfrac{5}{8} - \eps$ is now the best exponent of distribution for both primes and smooth numbers, in essentially any setting relevant for sieve theory. In fact, there does not appear to be a slightly more flexible setting which allows for a better exponent with current methods (e.g., primes with quadruply-well-factorable weights, or smooth numbers with well-factorable weights).

Our improvements stem mainly from a recent large sieve inequality for exceptional Maass forms of the author \cite[Theorem 3]{pascadi2024large}, combined, in the case of primes, with a large sieve inequality of Watt \cite[Theorem 2]{watt1995kloosterman}. These results act as on-average substitutes for Selberg's eigenvalue conjecture, by improving the dependency on $X$ in bounds for sums of the shape
\begin{equation} \label{eq:sketch-large-sieve}
    \sum_f X^{\theta_f} \left\vert \sum_{n \sim N} a_n\, \rho_f(n) \right\vert^2,
\end{equation}
where $f$ ranges over certain families of automorphic forms, with Fourier coefficients $\rho_f(n)$ and spectral parameters $\theta_f \in [0, 7/32]$; here $\theta_f > 0$ only when $f$ fails Selberg's conjecture, and the uniform bound of $\tfrac{7}{32}$ is due to Kim--Sarnak \cite[Appendix 2]{kim2003functoriality}. Importantly, both \cite[Theorem 3]{pascadi2024large} and \cite[Theorem 2]{watt1995kloosterman} use special sequences $(a_n)$ that arise in applications, roughly of the form
\begin{equation} \label{eq:sequence-types}
    a_n := \sum_{h, h' \sim H} \one_{h\ell-h'\ell' = n},
    \qquad\quad 
    \text{respectively}\  
    \qquad\quad
    a_n := \sum_{\substack{h \sim H \\ k \sim K}} \one_{hk = n}.
\end{equation}
The first sequence in \cref{eq:sequence-types} comes from an additive convolution, and its additive structure is manifested through a sparse Fourier transform, which was crucial for the large sieve inequalities in \cite{pascadi2024large}. By contrast, the second sequence above comes from a multiplicative convolution, and Watt's argument \cite[Section 2]{watt1995kloosterman} crucially relied on its multiplicative structure. Both arguments use the smoothness of the variables $h, h', k$, which come from Fourier completion.
\begin{remark} 
For some applications, it would be interesting to obtain improved large sieve inequalities for sequences which display a mix of additive and multiplicative structure, such as 
\[
    a_n := \sum_{\substack{h, h' \sim H \\ k, k' \sim K}} \one_{hk\ell - h'k'\ell' = n}.
\]
In particular, this would be relevant for improving the total length in a mean-value estimate for the squared zeta function times a product of two Dirichlet polynomials, due to Deshouillers--Iwaniec \cite[Theorem 2]{deshouillers1984power} (and refined to an asymptotic by Bettin--Chandee--Radziwill \cite{bettin2017mean}). For an optimal choice of unbalanced ranges $M > N$, assuming Selberg's eigenvalue conjecture, one should reach the threshold $MN \le T^{5/8}$; the presence of the exponent $\tfrac{5}{8}$ here is not a coincidence, since these results rely on bounds for exponential sums of the shape in \cref{eq:sketch-kl-fractions}.
\end{remark}

\begin{remark}
Although we will focus on equidistribution results with fixed (or small) residues $a$, similar results are possible in the range $a \ll x^{1+\eps}$; this is relevant, e.g., to upper-bounding counts of Goldbach representations \cite{lichtman2023primes}. However, working with large values of $a$ ultimately has the effect of replacing some of the dependency on (progress towards) Selberg's eigenvalue conjecture with its non-Archimedean counterpart, the Ramanujan--Petersson conjecture at primes dividing $a$. In \cite{lichtman2023primes}, Lichtman incorporates technology of Assing--Blomer--Li \cite{assing2021uniform} to explicitate the dependency on the Ramanujan--Petersson conjecture; with this approach, the final exponent of distribution decreases with $a$. However, using appropriate non-Archimedean analogues of the large sieve inequalities \cite[Theorem 3]{pascadi2024large} and \cite[Theorem 2]{watt1995kloosterman}, one should be able to match the exponent of distribution $\tfrac{5}{8}-\eps$ in a larger range of $a$, and in the full range $a \ll x^{1+\eps}$ if $a$ is well-factorable.
\end{remark}

\subsection{Acknowledgements}
We thank James Maynard, Sary Drappeau, Jori Merikoski, Lasse Grimmelt, and Jared Duker Lichtman for many helpful comments. The author is sponsored by the EPSRC Scholarship at University of Oxford.

\section{Overview} \label{sec:informal-overview}

Our proofs of \cref{thm:primes,thm:smooth} build on the arguments of Maynard \cite{maynard2025primes2} and Lichtman \cite{lichtman2023primes}, respectively Drappeau \cite{drappeau2015theoremes} and the author \cite{pascadi2023smooth}, with new inputs in the exceptional automorphic spectrum. Here we give an informal outline of our arguments, ignoring various technical details such as smooth weights, common divisors, and some $x^{o(1)}$ factors.

\subsection{Reduction to sums of Kloosterman fractions.}


Let $Q \in (x^{1/2+o(1)}, x^{5/8-o(1)})$, and fix the residue $a = 1$ for simplicity.
In the critical ranges, \cref{thm:primes}.$(i)$, respectively \cref{thm:smooth,thm:smooth-mod-smooth}, rely on bounding sums of the form
\begin{equation} \label{eq:conv-estimate-primes}
    \sum_{q_1 \sim Q_1} \lambda_{q_1} \sum_{q_2 \sim Q_2} \mu_{q_2} \sum_{q_3 \sim Q_3} \nu_{q_3} \sum_{n \sim N} \alpha_{n} \sum_{m \sim x/N} \beta_{m} \left(\one_{mn \equiv 1 \pmod{q_1q_2q_3}} - \frac{\one_{(mn, q_1q_2q_3) = 1}}{\varphi(q_1 q_2 q_3)}\right),
\end{equation}
respectively
\begin{equation} \label{eq:conv-estimate-smooth}
    \sum_{q \sim Q} \lambda_{q} \sum_{n_1 \sim N_1} \alpha_{n_1} \sum_{n_2 \sim N_2} \beta_{n_2} \sum_{n_3 \sim N_3} \gamma_{n_3} \left(\one_{n_1 n_2 n_3 \equiv 1 \pmod{q}} - \frac{\one_{(n_1 n_2 n_3, q) = 1}}{\varphi(q)}\right),
\end{equation}
for certain ranges of $Q_i$, $Q$, $N_i$, $N$ with $\prod Q_i \asymp Q$ and $\prod N_i \asymp x$, and for arbitrary divisor-bounded coefficients $(\lambda_q), (\mu_q), (\alpha_n), (\beta_n), (\gamma_n)$. The goal in both cases is to beat the trivial bound of size about $x$, while making $Q$ as large as possible. 

In \cref{eq:conv-estimate-primes} (for primes, with triply-well-factorable weights in the modulus), we are essentially free to factorize $Q = \prod Q_i$ as we wish in terms of $x$ and $N$, and we will roughly choose 
\begin{equation} \label{eq:ranges-primes}
    Q_1 \approx N, \qquad\quad 
    Q_2 \approx \frac{Q^2}{x},
    \qquad\quad 
    Q_3 \approx \frac{x}{QN}.
\end{equation}
Similarly, in \cref{eq:conv-estimate-smooth} (for smooth numbers, with arbitrary weights in the modulus), we are free to factorize $N = \prod N_i$ as we wish in terms of $x$ and $Q$, and we will roughly choose
\begin{equation} \label{eq:ranges-smooth}
    N_1 \approx \frac{x}{Q},
    \qquad\qquad 
    N_2 \approx \frac{Q^2}{x},
    \qquad\qquad 
    N_3 \approx \frac{x}{Q}.
\end{equation}

There is a certain duality between the two problems, partly due to the correspondence of sizes $Q_2 \approx N_2$, $N Q_3 \approx N_3$; indeed, the two convolution estimates above reduce to bounding the same sum of Kloosterman fractions, given below in \cref{eq:sketch-kl-fractions}. This is why the final exponents of distribution are the same -- both in previous works \cite{maynard2025primes2,drappeau2015theoremes}, \cite{lichtman2023primes,pascadi2023smooth}, and in our \cref{thm:primes,thm:smooth}. Both proofs rely on Linnik's dispersion method \cite{linnik1963dispersion,bombieri1986primes,bombieri1987primes2,bombieri1989primes3}, which begins with an application of Cauchy--Schwarz in $q_1, m$, respectively $q, n_1$; expanding the square will duplicate the other variables. The main dispersion sums will contain smooth sums over $q_1, m$, respectively $q, n_1$, as well as congruences
\[
    \begin{cases}
    n \equiv n' \pmod{q_1}, \\
    m \equiv \bar{n} \pmod{q_1 q_2 q_3}, \\
    m \equiv \bar{n'} \pmod{q_2' q_3'},
    \end{cases}
    \qquad\quad 
    \text{respectively}
    \qquad\quad 
    \begin{cases}
    n_2n_3 \equiv n_2'n_3' \pmod{q}, \\
    n_1 \equiv \bar{n_2 n_3} \pmod{q}.
    \end{cases}
\]
One can Fourier-complete the sums in $m$ mod $q_1q_2q_3q_2'q_3'$, respectively $n_1$ mod $q$, which introduces a smooth variable $h$ of size
\[
    |h| \le \frac{Q_1 (Q_2 Q_3)^2}{x/N} \approx \frac{Q^2}{x},
    \qquad\quad
    \text{respectively}
    \qquad\quad
    |h| \le \frac{Q}{N_1} \approx \frac{Q^2}{x},
\]
and the contribution of the principal frequency at $h = 0$ simplifies with other main terms. Moreover, one can pass from $q_1$, respectively $q$, to the complementary divisors $\tfrac{n-n'}{q_1}$, respectively $\tfrac{n_2n_3 - n_2'n_3'}{q}$, which have size $\asymp 1$ (so we can ignore them for the moment). In the critical ranges, it essentially remains to bound
\begin{equation} \label{eq:sketch-kl-fractions}
    \sum_{\substack{c \sim Q \\ d \sim x/Q}} \left\vert \sum_{\ell \sim Q^2/x} v_\ell \sum_{h \sim Q^2/x}\ e\left(h\frac{\bar{\ell d}}{c}\right) \right\vert
    < Q^2,
\end{equation}
where the $(c, d, \ell)$ variables correspond to $(nq_2'q_3', n'q_3, q_2)$, respectively $(n_2'n_3', n_3, n_2)$, and $(v_\ell)$ are divisor-bounded coefficients. Note that we need to save a factor of roughly $Q^2/x$ over the trivial bound, corresponding to the loss from Fourier completion.

\subsection{Reaching the exceptional spectrum} \label{subsec:reaching-exceptional}
We may now forget about the original structure from \cref{eq:conv-estimate-primes,eq:conv-estimate-smooth}, and focus on the exponential sum in \cref{eq:sketch-kl-fractions}.
After applying Cauchy--Schwarz once again and swapping sums, one is left with proving that
\begin{equation} \label{eq:incomplete-kloosterman-overview}
    \sum_{\ell, \ell' \sim Q^2/x} v_{\ell'} \bar v_{\ell} \sum_{h, h' \sim Q^2/x}\ \sum_{\substack{c \sim Q \\ d \sim x/Q}} e\left((h\ell - h'\ell')\frac{\bar{\ell \ell' d}}{c}\right)
    < \frac{Q^4}{x}.
\end{equation}
The diagonal terms with $h\ell = h'\ell'$ are barely acceptable. In the off-diagonal terms, denoting $r := \ell \ell'$, $n := h\ell - h'\ell'$, and\footnote{Really, $a_{n,r}$ depends on $n, \ell, \ell'$ since there may be more factorizations $r = \ell\ell'$, but let us ignore this here.} 
\begin{equation} \label{eq:sketch-special-seq}
    a_{n,r} := \sum_{h, h' \sim Q^2/x} \one_{h\ell - h'\ell' = n},
\end{equation}
and completing Kloosterman sums (passing from $d$ to a variable $m$ of dual size $Q(x/Q)^{-1}$), it remains to show that
\begin{equation} \label{eq:sum-kloosterman-overview}
    \sum_{r \sim Q^4/x^2} \left\vert
    \sum_{m \sim Q^2/x}\,
    \sum_{n \sim Q^4/x^2} a_{n,r} 
    \sum_{c \sim Q}
    S(m \bar{r}, n; c) \right\vert < \frac{Q^6}{x^2}.
\end{equation}
At this point, applying the Kuznetsov trace formula \cite{kuznetsov1980petersson,deshouillers1982kloosterman} to the inner sum over $c$ brings in the Fourier coefficients $\rho_f(m)$, $\rho_f(n)$ of automorphic forms $f$ for the congruence subgroup $\Gamma_0(r)$. The contribution of Maass cusp forms is the most difficult to bound; after moving the sums over $m, n$ inside, we are left to bound
\begin{equation} \label{eq:before-di-cs}
    \sum_{r \sim Q^4/x^2}
    \sum_{f}^{\Gamma_0(r)}
    \sqrt{x}^{\theta_f} \left\vert \sum_{m \sim Q^2/x} \rho_f(m) \sum_{n \sim Q^4/x^2} a_{n,r} \bar\rho_f(n)\right\vert
    < \frac{Q^3}{x},
\end{equation}
where $\theta_f \in [0, 7/32]$ measures the failure of Selberg's eigenvalue conjecture as in \cref{eq:sketch-large-sieve}, and $\rho_f(m)$, $\rho_f(n)$ have size about $r^{-1/2}$ on average. Following Deshouillers--Iwaniec \cite{deshouillers1982kloosterman}, one can now apply Cauchy--Schwarz a third time, so that it remains to bound
\begin{equation} \label{eq:after-di-cs}
    \left(
    \sum_{r \sim Q^4/x^2}
    \sum_{f}^{\Gamma_0(r)}
    \left(\frac{Q^6}{x^3}\right)^{\theta_f} \left\vert \sum_{m \sim Q^2/x} \rho_f(m)\right\vert^2 \right) 
    \left(
    \sum_{r \sim Q^4/x^2}
    \sum_{f}^{\Gamma_0(r)} \left(\frac{x^4}{Q^6}\right)^{\theta_f}
    \left\vert \sum_{n \sim Q^4/x^2} a_{n,r} \bar\rho_f(n)\right\vert^2 \right)
    < 
    \frac{Q^6}{x^2}.
\end{equation}
The reason for this arrangement is that the large sieve inequalities from \cite{deshouillers1982kloosterman} obtain square-root cancellation in the sums over $m, n$. Moreover, in the exceptional spectrum where $\theta_f > 0$, \cite[Theorem 7]{deshouillers1982kloosterman} can incorporate the factor of $(Q^6/x^3)^{\theta_f}$ from the first sum with no losses. 

However, for the exceptional factor of $(x^4/Q^6)^{\theta_f}$ in the second sum, all previous works \cite{maynard2025primes2,lichtman2023primes,drappeau2015theoremes,pascadi2023smooth} essentially use an $L^\infty$ bound. Denoting $\theta := \max \theta_f$ and using the aforementioned spectral large sieve inequalities, this would leave us with
\[
    \frac{Q^4}{x^2} \frac{Q^2}{x} \cdot \left(\frac{x^4}{Q^6}\right)^\theta \frac{Q^4}{x^2} \frac{Q^4}{x^2} 
    < 
    \frac{Q^6}{x^2}
    \qquad \iff 
    \qquad 
    Q < x^{\frac{5-4\theta}{8-6\theta}},
\]
where we really need a power-saving in the final bound.
Plugging in Selberg's bound $\theta \le 1/2$ (which was the state of the art in \cite{deshouillers1982greatest}), one reaches the exponent of distribution $\tfrac{3}{5} - o(1)$ from the works of Maynard \cite{maynard2025primes2} and Drappeau \cite{drappeau2015theoremes}. Using the celebrated bound $\theta \le \tfrac{7}{32}$ of Kim--Sarnak \cite[Appendix 2]{kim2003functoriality}, one reaches the exponent $\tfrac{66}{107}-o(1)$ from the works of Lichtman \cite{lichtman2023primes} and the author \cite{pascadi2023smooth}. Conditionally on Selberg's conjecture that $\theta = 0$, the resulting exponent is $\tfrac{5}{8} - o(1)$.

\subsection{Our improvements} \label{subsec:improvements}

Naturally, one can hope to win more in the exceptional spectrum using a suitable large sieve inequality for the second sum in \cref{eq:after-di-cs}; but until very recently, it was impossible to obtain \emph{any} savings in the $\theta$-aspect for sequences like $(a_{n,r})$ which depend on the level $r$, when $n$ and $r$ have the same size. This is now possible using the author's work \cite{pascadi2024large}, provided that the sequence $(a_{n,r})$ has enough additive structure. Indeed, for the shape of $a_{n,r}$ from \cref{eq:sketch-special-seq}, which matches the left-hand side of \cref{eq:sequence-types}, \cite[Theorem 3]{pascadi2024large} saves an additional factor of $(Q^2/x)^{\theta_f}$ in \cref{eq:after-di-cs}. This leads to a final bound of
\[
    \frac{Q^4}{x^2} \frac{Q^2}{x} \cdot \left(\frac{x^5}{Q^8}\right)^\theta \frac{Q^4}{x^2} \frac{Q^4}{x^2} 
    < 
    \frac{Q^6}{x^2}
    \qquad \iff 
    \qquad 
    Q < x^{\frac{5-5\theta}{8-8\theta}} = x^{5/8},
\]
and thus to the unconditional exponent of distribution of $\tfrac{5}{8} - o(1)$.

This concludes the outline of our results on smooth numbers from \cref{thm:smooth,thm:smooth-mod-smooth}, up to various technical details. However, the case of primes from \cref{thm:primes}.$(i)$ presents a significant additional challenge: the triply-well-factorable condition from \cref{def:triply-well-fact} can only really guarantee that $Q_1 \le N$, $Q_2 \le \frac{Q^2}{x}$ and $Q_3 \le \frac{x}{QN}$, as opposed to the double-sided bounds implied in \cref{eq:ranges-primes}. The potential gap between $Q_1$ and $N$ creates a large complementary-divisor factor 
\[
    f := \frac{n-n'}{q_1} \ll F := \frac{N}{Q_1},
\] 
which ultimately alters the shape of the coefficients $(a_{n,r})$ from \cref{eq:sketch-special-seq} to
\[
    a_{n,r} \approx 
    \sum_{f \sim F}
    \sum_{h, h' \sim Q^2/(xF)} \one_{f(h\ell - h'\ell') = n}.
\]
This sequence displays a mix of additive and multiplicative structure, and we do not know how to prove a corresponding large sieve inequality in the exceptional spectrum, generalizing \cite[Theorem 3]{pascadi2024large} with a good dependency on $F$. This is a significant issue, since the previous argument could only barely reach the unconditional exponent of $\tfrac{5}{8} - o(1)$.

We overcome this issue by moving $f$-variable to the other entry of the Kloosterman sums, by a variant of the identity 
\[
    S(m \bar{r}, fn; c) = S(fm \bar{r}, n; c),
\]
which holds when $(f, c) = 1$; working around the latter coprimality constraint is a nontrivial argument in itself, within the proof of \cref{lem:expo-bound-well-fact}. In the $n$-aspect from \cref{eq:after-di-cs}, this leaves us with coefficients $(a_{n, r})$ as in the left-hand side of \cref{eq:sequence-types}, which can be handled by \cite[Theorem 3]{pascadi2024large}. In the $m$-aspect from \cref{eq:after-di-cs}, we are left with a multiplicative convolution of two smooth sequences, as in the right-hand side of \cref{eq:sequence-types}. For such sequences, Watt's large sieve inequality \cite[Theorem 2]{watt1995kloosterman}, incorporated into our \cref{prop:kloosterman-complete}, produces nearly-optimal savings when an average over the level $r$ is available. The final dependency of the resulting bounds on $F$ is acceptable, partly because the $m$-variable is much smaller than the level (so there is enough `room' for the $f$-variable). 


For \cref{thm:primes}.$(ii)$ and \cref{thm:primes-ls-weights}, we mention that Iwaniec's well-factorable linear sieve weights are not very far from being triply-well-factorable -- in fact, such results still depend on bounding the sum in \cref{eq:conv-estimate-primes}, but with less freedom in choosing the parameters $Q_1, Q_2, Q_3$. This lower degree of flexibility leads to fairly complicated (but purely elementary) combinatorial optimization problems, which we treat in \cref{sec:primes-linear-sieve}. Once again, the final levels of distribution match the best conditional results (that one would obtain by assuming Selberg's eigenvalue conjecture in our proofs).

\subsection{Structure of paper}
In \cref{sec:preliminaries}, we establish some notation and preliminaries; in particular, we reiterate the large sieve inequalities for exceptional Maass forms of the author \cite[Theorem 3]{pascadi2024large} and Watt \cite[Theorem 2]{watt1995kloosterman} in \cref{prop:large-sieve-additive,prop:large-sieve-multiplicative}, and give bounds for multilinear forms of Kloosterman sums in \cref{prop:kloosterman-incomplete,prop:kloosterman-complete}. We use these results to prove:
\begin{itemize} 
    \item \cref{thm:primes} (parts $(i)$ and $(ii)$, resp.) in \cref{sec:primes,subsec:upper-bound-linear-sieve}, building on Maynard \cite{maynard2025primes2};
    \item \cref{thm:primes-ls-weights} and \cref{cor:twin-primes} in \cref{subsec:linear-sieve-special-factors}, building on Lichtman \cite{lichtman2023primes}; 
    \item \cref{thm:smooth} in \cref{sec:smooth}, building on Drappeau \cite{drappeau2015theoremes};
    \item \cref{thm:smooth-mod-smooth} and \cref{cor:consec-smooth} in \cref{sec:smooth-weights-smooth}, building on de la Bret\`eche--Drappeau \cite{de2020niveau}.
\end{itemize}
The figure below summarizes the relationships between the results in this paper.
\begin{figure}[ht]
\centering
\begin{tikzpicture}[
basic/.style={rectangle, draw=black!60, fill=white, very thick, minimum size=5mm},
square/.style={rectangle, draw=black!60, fill=cyan!10, very thick, minimum size=5mm},
important/.style={rectangle, double, draw=black!60, fill=cyan!10, very thick, minimum size=5mm},
]
\node[basic] (large-sieve-multiplicative) {\cref{prop:large-sieve-multiplicative}};
\node[basic][below = 0.5cm of large-sieve-multiplicative] (kloosterman-complete) {\cref{prop:kloosterman-complete}};
\node[basic][right = 1cm of kloosterman-complete] (large-sieve-additive) {\cref{prop:large-sieve-additive}};
\node[basic][right = 1cm of large-sieve-additive] (kloosterman-incomplete) {\cref{prop:kloosterman-incomplete}};
\node[square][below = 0.5cm of kloosterman-complete] (conseq-kloosterman-complete) {\cref{lem:consequence-Kloosterman}};
\node[square][below = 0.5cm of kloosterman-incomplete] (conseq-kloosterman-incomplete) {\cref{lem:consequence-incomplete}};
\node[square][below = 0.5cm of conseq-kloosterman-complete] (expo-bound-well-fact) {\cref{lem:expo-bound-well-fact}};
\node[square][below = 0.5cm of conseq-kloosterman-incomplete] (expo-bound-convo) {\cref{lem:expo-bound-convo}};
\node[square][below = 0.5cm of expo-bound-well-fact] (well-fact-convo) {\cref{prop:well-fact-convolution}};
\node[square][below = 0.5cm of expo-bound-convo] (triple-convo) {\cref{prop:triple-convo}};
\node[important][xshift = -1.5cm][below = 0.5cm of well-fact-convo] (primes) {\cref{thm:primes}};
\node[square][xshift = 1.5cm][below = 0.5cm of well-fact-convo] (primes-ls) {\cref{thm:primes-ls-weights}};
\node[important][xshift = -1.5cm][below = 0.5cm of triple-convo] (smooth) {\cref{thm:smooth}};
\node[square][xshift = 1.5cm][below = 0.5cm of triple-convo] (smooth-mod) {\cref{thm:smooth-mod-smooth}};
\node[square][below = 0.5cm of primes-ls] (twin-primes) {\cref{cor:twin-primes}};
\node[square][below = 0.5cm of smooth-mod] (consec-smooth) {\cref{cor:consec-smooth}};
\draw[-triangle 45] (large-sieve-multiplicative) -- (kloosterman-complete);
\draw[-triangle 45] (kloosterman-complete) -- (conseq-kloosterman-complete);
\draw[-triangle 45] (kloosterman-incomplete) -- (conseq-kloosterman-incomplete);
\draw[-] ([xshift=-1cm] large-sieve-additive) |- ([yshift=0.35cm] conseq-kloosterman-complete.north);
\draw[-] ([xshift=1cm] large-sieve-additive) |- ([yshift=0.35cm] conseq-kloosterman-incomplete.north);
\draw[-triangle 45] (conseq-kloosterman-complete) -- (expo-bound-well-fact);
\draw[-triangle 45] (conseq-kloosterman-incomplete) -- (expo-bound-convo);
\draw[-triangle 45] (expo-bound-well-fact) -- (well-fact-convo);
\draw[-triangle 45] (expo-bound-convo) -- (triple-convo);
\draw[-triangle 45] (well-fact-convo) -- (primes);
\draw[-triangle 45] (well-fact-convo) -- (primes-ls);
\draw[-triangle 45] (triple-convo) -- (smooth);
\draw[-triangle 45] (triple-convo) -- (smooth-mod);
\draw[-triangle 45] (primes-ls) -- (twin-primes);
\draw[-triangle 45] (smooth-mod) -- (consec-smooth);
\end{tikzpicture}
\caption{Structure of proofs (\emph{arrows represent logical implications}).}
\label{fig:structure}
\end{figure}
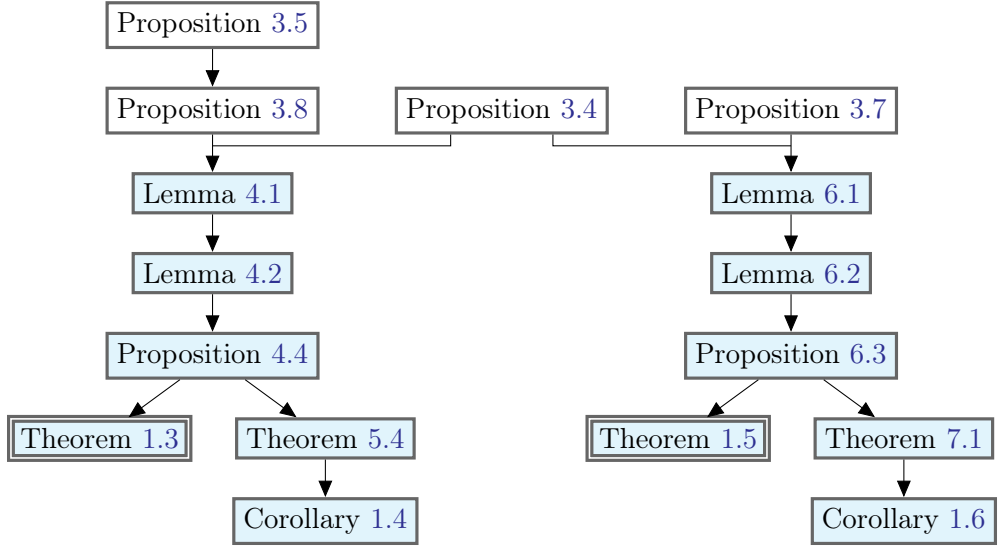
\FloatBarrier

\section{Preliminaries and the exceptional spectrum} \label{sec:preliminaries}

\subsection{Analytic and combinatorial notation} \label{subsec:analytic-combinatorial}

We use the standard asymptotic notation $f = o(g)$, $f = O(g)$, $f \ll g$, $f \asymp g$, indicating dependencies of implicit constants through subscripts (e.g, $f = O_\eps(g)$ means $|f| \le C_\eps |g|$ for some $C_\eps > 0$ depending only on $\eps$). Statements like $f(x) \ll x^{o(1)} g(x)$ should be read as $\forall \eps > 0$, $f(x) \ll_\eps x^\eps g(x)$. We use $n \sim N$ for the range of integers $n \in (N, 2N]$, $\one_S$ for the truth value ($0$ or $1$) of a statement $S$, and $\|a_n\|_q$ for the $\ell^q$ norm of a sequence $(a_n)$. We let $\tau_k(n) := \sum_{d_1\cdots d_k = n} 1$, $\tau := \tau_2$, say that a sequence $(a_n)$ is \emph{divisor-bounded} iff $a_n \ll \tau(n)^{O(1)}$.

We write $\Z_+$ for the set of positive integers, $P^+(n)$ for the greatest prime factor of $n$, $(a, b) = \gcd(a, b)$, $[a, b] = \lcm(a, b)$, and $\bar{x}$ for the inverse of $x$ modulo $c$, where $c$ depends on the context. For $\alpha \in \R$, we denote $\|\alpha\| := \min_{n \in \Z} |\alpha - n|$, which induces a norm on $\R/\Z$. Following \cite[Notation 5]{pascadi2024large} (up to a constant), for $\alpha \in \R/\Z$ and $N > 0$ we use the notation
\begin{equation} \label{eq:tn}
    T_N(\alpha) := \min_{t \in \Z_+} \left(t + N\|t\alpha\|\right),
\end{equation}
which is nondecreasing in $N$, and measures the quality of rational approximations to $\alpha$ with small denominators $t$. In particular, we have $T_N(\alpha) \le 1 + N\|\alpha\|$, $T_N(\alpha+\beta) \ll (1 + N\|\beta\|)T_N(\alpha)$, and $T_N(\alpha) \ll \sqrt{N}$ by a pigeonhole argument \cite[Lemma 7]{pascadi2024large}.

We denote $e(\alpha) := e^{2\pi i \alpha}$ for $\alpha \in \R/\Z$, and use the Fourier transform normalization
\[
    \hat{f}(\xi) := \int_\R f(t) e(-\xi t) dt,
\]
for absolutely integrable functions $f : \R \to \C$. We recall a truncated version of Poisson summation:

\begin{lemma}[Truncated Poisson with separation of variables] \label{lem:truncated-poisson}
Let $x \gg 1$ and $1 \ll N, Q \ll x^{O(1)}$, $a \in \Z$, $q \in \Z_+$ with $q \asymp Q$, and $\Phi : (0, \infty) \to \C$ be a smooth function, $\Phi(t)$ supported in $t \asymp 1$, with $\Phi^{(k)} \ll_k 1$ for $k \ge 0$. Then for any $A, \delta > 0$ and $H := x^\delta N^{-1} Q$, one has
\[
\begin{aligned}
    \sum_{n \equiv a \pmod{q}} \Phi\left(\frac{n}{N}\right)
    &=
    \frac{N}{q} \hat{\Phi}(0)
    +
    O_{A,\delta}\left(x^{-A}\right)
    \\
    &+
    \frac{N}{Q} \int \Phi\left(\frac{uq}{Q}\right) \sum_{\substack{H_j = 2^j \\ 1 \le H_j \le H}}
    \sum_{h \in \Z} e\left(- h\frac{uN}{Q}\right) 
     \Psi_j\left(\frac{|h|}{H_j}\right) e\left( \frac{ah}{q} \right) du,
\end{aligned}
\]
where $\Psi_j : (\frac{1}{2}, 2) \to \C$ are some compactly supported functions with $\Psi_j^{(k)} \ll_k 1$ for $k \ge 0$. Note that the integrand is supported in $u \asymp 1$, and that one can rewrite
\[
    \frac{N}{Q} \Phi\left(\frac{uq}{Q}\right) du = \frac{N}{q} \tilde\Phi\left(\frac{uq}{Q}\right) \frac{du}{u}, \qquad \text{where} \qquad \tilde\Phi(t) := t\Phi(t).
\]
\end{lemma}

\begin{proof}
This is \cite[Lemma 4]{pascadi2024large}, which follows quickly from the Poisson summation formula and a smooth dyadic partition of unity. Note that the variables $h$ and $q$ are separated at the cost of a slowly-varying exponential phase $e(h\omega)$ where $\omega = -uN/Q \ll x^\delta H^{-1}$.
\end{proof}

\subsection{Automorphic forms and large sieve inequalities} \label{subsec:automorphic-large-sieve}

For a \emph{level} $q \in \Z_+$, we recall that $\Gamma_0(q)$ is the congruence subgroup of $\PSL_2(\Z)$ consisting of all matrices with bottom-left entry divisible by $q$. The fundamental domain $\Gamma_0(q)\backslash \H$ has finitely many cusps, equivalent to fractions $u/w$ with $w \mid q$, $(u, w) = 1$, and $u \le (w, \tfrac{q}{w})$; in particular, the cusp at $\infty$ is equivalent to $1/q$. To a cusp $\ma$ we associate the parameter $\mu(\ma) := (w, \tfrac{q}{w})/q$, which will always equal $q^{-1}$ in this paper -- in other words, we will only use cusps equivalent to $1/s$ where $q = rs$, $(r, s) = 1$. For such a cusp $\ma = \tau(1/s)$, with $\tau \in \Gamma_0(q)$, we will use the choice of scaling matrix 
\begin{equation} \label{eq:scaling-choices}
    \sigma_{\ma} := \tau \cdot \begin{pmatrix} 
    \sqrt{r} & -\frac{\bar{s}}{\sqrt{r}} \\
    s\sqrt{r} & \bar{r}\sqrt{r}
    \end{pmatrix},
    \qquad\qquad 
    \text{for some } \bar{r}, \bar{s} \in \Z \text{ with }
    r\bar{r} + s\bar{s} = 1,
\end{equation}
as in \cite[(3.9)]{pascadi2024large}; this matters in the Fourier-expansion of automorphic forms around $\ma$. In particular, $\sigma_\infty$ is the identity matrix.

We use notation and normalization consistent with Deshouillers--Iwaniec \cite{deshouillers1982kloosterman} for automorphic forms on $\Gamma_0(q) \backslash \H$, and refer the reader to \cite[3.3]{pascadi2024large} for details. 
As outlined in \cref{subsec:improvements}, the improvements in our final results come from the treatment of the exceptional spectrum, i.e., of the Maass cusp forms $f : \Gamma_0(q) \backslash \H \to \C$ which are eigenforms of the hyperbolic Laplacian with eigenvalues $\lambda_f < 1/4$. \cref{conj:selberg} predicts that there are no such eigenvalues, but unconditionally, they arise in the Kuznetsov trace formula \cite{deshouillers1982greatest,kuznetsov1980petersson} and its consequences through factors of $X^{\theta_f}$, where $\theta_f = \sqrt{1 - 4\lambda_f}$ and $X$ is typically large in applications; recall \cref{eq:before-di-cs}, where $X \approx \sqrt{x}$. Thus letting
\[
    \theta_{\max} := \sup_{q} \sup_f^{\Gamma_0(q)} \sqrt{\max(0, 1 - 4\lambda_f)},
\]
Selberg's conjecture asserts that $\theta_{\max} = 0$, while the best (pointwise) unconditional result is due to Kim--Sarnak \cite[Appendix 2]{kim2003functoriality}:
\begin{equation} \label{eq:ks-bound}
    \theta_{\max} \le \frac{7}{32}.
\end{equation}
Following Deshouillers--Iwaniec \cite{deshouillers1982kloosterman}, the factors $X^{\theta_f}$ can be tempered, in expressions like \cref{eq:after-di-cs}, through large sieve inequalities for the Fourier coefficients of exceptional Maass forms, of the shape in \cref{eq:sketch-large-sieve}. The goal is to incorporate factors of $X^{\theta_f}$ with $X$ as large as possible, while matching the essentially-optimal upper bound for the regular-spectrum large sieve inequalities \cite[Theorem 2]{deshouillers1982kloosterman}. Below we recall \cite[Assumption 14]{pascadi2024large}, which is just a framework to state large sieve inequalities and their corollaries succinctly. On a first read, one should pretend that $\xi = 0$ and $A = \|a_n\|_2$.

\begin{assumption}[Exceptional large sieve] \label{ass:large-sieve}
We say that a tuple $(q, N, Z, (a_n)_{n \sim N}, A, Y)$, with $q \in \Z_+$, $N \ge 1/2$, $Z \gg 1$, $A \gg \|a_n\|_2$, $Y > 0$, satisfies this assumption iff the following holds. For any $\eps > 0$, $\xi \in \R$, any cusp $\ma$ of $\Gamma_0(q)$ with $\mu(\ma) = q^{-1}$, and any orthonormal basis of exceptional Maass cusp forms $f$ for $\Gamma_0(q)$, with Laplacian eigenvalues $\lambda_f$, $\theta_f := \sqrt{1-4\lambda_f}$, and Fourier coefficients $\rho_{f\ma}(n)$ (using the choice of scaling matrix in \cref{eq:scaling-choices}), one has
\[
    \sum_{\substack{f \\ \lambda_f < 1/4}}^{\Gamma_0(q)} X^{\theta_f} \left\vert \sum_{n \sim N} e\left(\frac{n}{N} \xi \right) a_n\, \rho_{f\ma}(n) \right\vert^2 
    \ll_\eps
    (qNZ)^\eps 
    \left(1 + \frac{N}{q}\right) 
     A^2,
\]
for all 
\begin{equation} \label{eq:x-saving-ass}
    X \ll \max\left(1, \frac{q}{N}\right) \frac{Y}{1+|\xi|^2}.
\end{equation}
\end{assumption}

The reason for this notation is that a result of Deshouillers--Iwaniec \cite{deshouillers1982kloosterman}, reiterated below, incorporates a factor of $X = \max(1, q/N)$ for an arbitrary sequence $(a_n)$; in fact, this is still the best-known result for general sequences and individual levels, when $N \gg \sqrt{q}$. Therefore, $Y$ in \cref{eq:x-saving-ass} represents the additional saving over this result, achieved using the special structure of the sequence $(a_n)$.

\begin{proposition}[Large sieve with general sequences \cite{deshouillers1982kloosterman}] \label{prop:large-sieve-general}
    Let $q \in \Z_+$, $N \ge 1/2$, and $(a_n)_{n \sim N}$ be an arbitrary complex sequence. Then the tuple $(q, N, 1, (a_n)_{n \sim N}, \|a_n\|_2, 1)$ satisfies \cref{ass:large-sieve}.
\end{proposition}
\begin{proof}
This follows immediately from \cite[Theorems 2 and 5]{deshouillers1982kloosterman}.
\end{proof}

We now recall a large sieve inequality of the author, concerning the first type of sequence from \cref{eq:sequence-types}. This is the main ingredient behind the improvements in \cref{thm:primes,thm:smooth}.

\begin{proposition}[Large sieve with additive convolutions \cite{pascadi2024large}] \label{prop:large-sieve-additive}
Let $N \ge 1/2$, $L, H \gg 1$, $\alpha_1, \alpha_2 \in \R/\Z$, and $q, \ell_1, \ell_2 \in \Z_+$, $a \in \Z$ be such that $q \gg L^2$, $\ell_1, \ell_2 \asymp L$, and $(\ell_1, \ell_2) = 1$. Let $\Phi_i(t) : (0, \infty) \to \C$ be smooth functions supported in $t \ll 1$, with $\Phi_i^{(j)} \ll_j 1$ for all $j \ge 0$, and
\[
    a_n := \sum_{\substack{h_1, h_2 \in \Z \\ a(h_1 \ell_1 - h_2 \ell_2) = n}} \Phi_1\left(\frac{h_1}{H}\right) \Phi_2\left(\frac{h_2}{H}\right) e(h_1 \alpha_1 + h_2 \alpha_2).
\]
Then the tuple $(q, N, H, (a_n)_{n \sim N}, A, Y)$ satisfies \cref{ass:large-sieve}, where
\[
    A := \|a_n\|_2 + \sqrt{N \left(\frac{H}{L} + \frac{H^2}{L^2}\right)}, \qquad\qquad Y := \max\left(1, \frac{NH}{|a|(H+L)L\min_i T_H(\alpha_i)}\right).
\]
\end{proposition}
\begin{proof}
This follows from \cite[Theorem 3]{pascadi2024large} with $a_n \gets a_{n/|a|}$, $N \gets N/|a|$, $a \gets |a|$, as in \cite[(5.18)]{pascadi2024large} (we note that the statement is trivial unless $N \ll |a|HL$, and recall our notation \cref{eq:tn}). One can in fact replace $T_H(\alpha_i)$ with the smaller quantity $T_{N/(|a|L)}(\alpha_i)$ (up to a constant), which is helpful in incorporating the phase $\xi$ from \cref{ass:large-sieve} via the bound
\[
    T_{N/(|a|L)}\left(\alpha_i \pm \frac{a\ell_i \xi}{N} \right) \ll \left(1 + \frac{N}{|a|L} \frac{|a \ell_i \xi|}{N}\right) T_{N/(|a|L)}(\alpha_i) \ll (1 + |\xi|) T_H(\alpha_i).
\]
\end{proof}

Finally, we recall a large sieve inequality of Watt \cite{watt1995kloosterman}, 
concerning the second type of sequence from \cref{eq:sequence-types}. We note that this result requires averaging over levels $q \sim Q$ with the same sequence $(a_n)$, while an important feature of \cref{prop:large-sieve-general,prop:large-sieve-additive} is that $(a_n)$ may depend on $q$.

\begin{proposition}[Large sieve with multiplicative convolutions \cite{watt1995kloosterman}] \label{prop:large-sieve-multiplicative}
Let $\eps$, $X > 0$, $Q \ge 1/2$, $N_1, N_2, Z \gg 1$, and $\Psi_1(t), \Psi_2(t)$ be smooth functions supported $t \asymp 1$, with $\Phi_i^{(j)} \ll_j Z^j$ for $j \ge 0$. Let $q \in \Z_+$ and $\infty_q$ be the cusp at $\infty$ of $\Gamma_0(q)$. Then using the same notation as in \cref{ass:large-sieve},
\[
    \sum_{q \sim Q} \sum_{\substack{f \\ \lambda_f < 1/4}}^{\Gamma_0(q)}
    X^{\theta_f} 
    \left\vert 
    \sum_{n_1, n_2} \Psi_1\left(\frac{n_1}{N_1}\right) \Psi_2\left(\frac{n_2}{N_2}\right) \rho_{f\infty_q}(n_1 n_2)
    \right\vert^2 
    \ll_\eps 
    Q^\eps Z^5
    \left(Q + N_1 N_2\right) N_1 N_2
\]
holds for any
\begin{equation} \label{eq:watt-saving}
    X \ll \frac{Q^2}{N_1^2 N_2}.
\end{equation}
\end{proposition}

\begin{proof}
This is \cite[Theorem 2]{watt1995kloosterman} with $H = N_1$ and $K = N_2$. In fact, \cite[Theorem 2]{watt1995kloosterman} is stated for functions $\Psi_1(t), \Psi_2(t)$ supported on $t \in [1, 2]$, but the same proof extends to any support $t \asymp 1$ (alternatively, one can use a smooth partition of unity to reduce to functions supported in $[1, 2]$).
\end{proof}

\subsection{Kloostermania} \label{subsec:sums-kloosterman}

We recall the classical Kloosterman sums, given for $m, n \in \Z$ and $c \in \Z_+$ by
\[
    S(m, n; c) := \sum_{x \in (\Z/c\Z)^\times} e\left(\frac{mx + n\bar{x}}{c}\right),
\]
where $x\bar{x} \equiv 1 \pmod{c}$. These can be bounded pointwise using:
\begin{lemma}[Weil and Ramanujan bounds] \label{lem:weil}
For any $m, n \in \Z$ and $c \in \Z_+$,
\[
\begin{aligned}
    S(m, n; c) &\ll \tau(c)\, (m, n, c)^{1/2} c^{1/2},
    \\
    |S(0, n; c)| &\le (n, c).
\end{aligned}
\]
\end{lemma}

\begin{proof}
The first bound is due to Weil, and uses algebraic geometry; see \cite[Corollary 11.12]{iwaniec2021analytic}. The second bound is classical and follows from M\"obius inversion.
\end{proof}

Following Deshouillers--Iwaniec \cite{deshouillers1982kloosterman}, multilinear forms of Kloosterman sums (crucially, with a smooth sum over the modulus $c$) can also be bounded using the spectral theory of automorphic forms. We state the necessary bounds using the framework of \cref{ass:large-sieve} for the exceptional spectrum; we will combine these with the large sieve inequalities from \cref{prop:large-sieve-additive,prop:large-sieve-general} to remove the dependency on Selberg's eigenvalue conjecture in our results.

\begin{proposition}[Sums of incomplete Kloosterman sums \cite{pascadi2024large}] \label{prop:kloosterman-incomplete}
Let $R, S, N \ge 1/2$, $C, D, Z \gg 1$, and $Y, \eps > 0$. For all $r \sim R$, $s \sim S$ with $(r, s) = 1$, let:
\begin{itemize} 
    \item $w_{r,s} \in \C$; 
    \item $\Phi_{r,s} : (0, \infty)^3 \to \C$ be smooth, with $\Phi_{r,s}(x, y, z)$ supported in $x, y, z \asymp 1$, and
    \[ 
        \partial_x^j \partial_y^k \partial_z^\ell \Phi_q(x, y, z) \ll_{j,k,\ell,\eps} Z^{j\eps},
        \qquad 
        \forall j, k, \ell \ge 0;
    \]
    \item $(rs, N, Z, (a_{n,r,s})_{n \sim N}, A_{r,s}, Y)$ be a tuple satisfying \cref{ass:large-sieve}.
\end{itemize} 
Then with a consistent choice of the sign $\pm$, it holds that
\begin{equation} \label{eq:kloosterman-incomplete}
    \sum_{\substack{r \sim R \\ s \sim S \\ (r, s) = 1}} w_{r,s}
    \sum_{n \sim N} a_{n,r,s}
    \sum_{\substack{c, d \\ (rd, sc) = 1}} \Phi_{r,s}\left(\frac{n}{N}, \frac{d}{D}, \frac{c}{C}\right) e\left(\pm n \frac{\bar{rd}}{sc}\right)
    \ll_\eps 
    (RSNCDZ)^{O(\eps)}\, \|w_{r,s} A_{r,s}\|_2 \msI,
\end{equation}
where
\[
    \msI^2 := D^2 NR
    +
    \left(1 + \frac{C^2}{R^2 S Y}\right)^{\theta_{\max}} 
    CS(C + DR)(RS + N).
\]
\end{proposition}
\begin{proof}
This is \cite[Corollary 18]{pascadi2024large}.
\end{proof}

Finally, we use \cref{prop:large-sieve-multiplicative} to deduce a close variant of \cite[Corollary 17]{pascadi2024large}, where the coefficient at $m$ is given by a multiplicative convolution of two smooth sequences.

\begin{proposition}[Sums of Kloosterman sums with multiplicative convolutions] \label{prop:kloosterman-complete}
Let $R, S, N \ge 1/2$, $M_1, M_2, C, Z \gg 1$, $M \asymp M_1 M_2$, and $Y, \eps > 0$. For all $r \sim R$, $s \sim S$ with $(r, s) = 1$, let:
\begin{itemize} 
    \item $w_{r,s} \in \C$; 
    \item $\Phi_{r,s} : (0, \infty)^4 \to \C$ be smooth, with $\Phi_{r,s}(x_1, x_2, y, z)$ supported in $x_1, x_2, y, z \asymp 1$, and
    \[ 
        \partial_{x_1}^{j_1} \partial_{x_2}^{j_2} \partial_y^k \partial_z^\ell \Phi_q(x_1, x_2, y, z) \ll_{j_1,j_2,k,\ell,\eps} Z^{(j_1+j_2+k)\eps},
        \qquad 
        \forall j_1, j_2, k, \ell \ge 0;
    \]
    \item $(rs, N, Z, (a_{n,r,s})_{n \sim N}, A_{r,s}, Y)$ be a tuple satisfying \cref{ass:large-sieve}.
\end{itemize} 
Then with a consistent choice of the sign $\pm$, it holds that
\begin{equation} \label{eq:kloosterman-rsm12nc}
\begin{aligned}
    \sum_{\substack{r \sim R \\ s \sim S \\ (r, s) = 1}} w_{r,s} \sum_{m_1, m_2 \in \Z} 
    \sum_{n \sim N} a_{n,r,s}
    \sum_{(c, r) = 1} \Phi_{r,s}\left(\frac{m_1}{M_1}, \frac{m_2}{M_2}, \frac{n}{N}, \frac{c}{C}\right) S(m_1m_2\bar{r}, \pm n; sc)
    \ll_\eps
    (RSMNCZ)^{O(\eps)}
    \\
    \times 
    \left(1 + \frac{C \sqrt{M_1}}{R \sqrt{SY}}\right)^{\theta_{\max}} 
    \|w_{r,s} A_{r,s}\|_2 \sqrt{RSM}
    \left(\frac{C^2}{R} (M + RS)(N + RS) + MN \right)^{1/2}.
\end{aligned}
\end{equation}
\end{proposition}

\begin{proof}
We closely follow the proof of \cite[Corollary 17]{pascadi2023smooth}, with minor changes.
We start by inserting coefficients $\Psi_i(m_i/M_i)$ in the sum $\mS$ from the left-hand side of \cref{eq:kloosterman-rsm12nc}; here $\Psi_i(t)$ are smooth functions with $\Psi_i^{(j)} \ll_j 1$, supported in $t \asymp 1$, and equal to $1$ on the supports of $x_1, x_2 \asymp 1$ in $\Phi_{r,s}(x_1, x_2, y, z)$. We then separate variables using the Fourier inversion formula 
\[
\begin{aligned}
    \Psi_{r,s}(x_1, x_2, y; z)
    &:= 
    \sqrt{x_1 x_2 y}\,\Phi_{r,s}\left(x_1, x_2, y, \frac{\sqrt{x_1 x_2 y}}{z} \right) \\
    &= 
    \iiint_{\R^3} \hat{\Psi}_{r,s}(\zeta_1, \zeta_2, \xi; z) \, e(x_1\zeta_1 + x_2\zeta_2 + y\xi)\, d\zeta_1\, d\zeta_2\, d\xi,
\end{aligned}
\]
where the Fourier transform is taken in the first three variables. Thus
\[
\begin{aligned}
    \Phi_{r,s}\left(\frac{m_1}{M_1}, \frac{m_2}{M_2}, \frac{n}{N}, \frac{c}{C}\right)
    &=
    \frac{\sqrt{M_1M_2N}}{\sqrt{m_1m_2n}}
    \\
    &\times 
    \iiint_{\R^2} \hat{\Psi}_{r,s}\left(\zeta_1, \zeta_2, \xi; \frac{C\sqrt{m_1m_2n}}{c\sqrt{M_1M_2N}}\right) \, e\left(\frac{m_1}{M_1}\zeta_1 + \frac{m_2}{M_2}\zeta_2 + \frac{n}{N}\xi\right)\, d\zeta_1\, d\zeta_2\, d\xi.
\end{aligned}
\]
Similarly as in \cite[(5.21)]{pascadi2024large}, this yields
\begin{equation} \label{eq:final-kloosterman-integral}
    \mS \ll_\eps Z^{O(\eps)} CS\sqrt{R} \iiint_{\R^2} \frac{S(\zeta_1, \zeta_2, \xi)\, d\zeta_1\, d\zeta_2\, d\xi}{(1 + \zeta_1^{100}) (1 + \zeta_2^{100}) (1 + \xi^{100})},
\end{equation}
where
\[
\begin{aligned}
    \mS(\zeta_1, \zeta_2, \xi) := \sum_{\substack{r \sim R \\ s \sim S \\ (r, s) = 1}} |w_{r,s}| \Bigg\vert \sum_{m_1, m_2 \in \Z} \Psi_1\left(\frac{m_1}{M_1}\right) \Psi_2\left(\frac{m_2}{M_2}\right) e\left(\frac{m_1}{M_1}\zeta_1 + \frac{m_2}{M_2}\zeta_2\right) \sum_{n \sim N} b_n\, e\left(\frac{n}{N}\xi\right) 
    \\
    \times \sum_{(c, r) = 1}\,
    \frac{S(m_1 m_2 \bar{r}, \pm n; sc)}{cs\sqrt{r}} 
    \varphi_{\zeta_1,\zeta_2,\xi,r,s}\left(\frac{4\pi \sqrt{m_1 m_2 n}}{c}\right) \Bigg\vert,
\end{aligned}
\]
and $\varphi_{\zeta_1, \zeta_2, \xi, r, s}(z)$ is supported in $z \asymp X^{-1}$, and satisfies $\varphi_{\zeta_1,\zeta_2,\xi}^{(\ell)} \ll_\ell X^\ell$ for 
\begin{equation} \label{eq:initial-X-factor}
    X := \frac{CS\sqrt{R}}{\sqrt{MN}}.
\end{equation}
We can incorporate the factors $e(t \zeta_i)$ into the functions $\Psi_i(t)$, incurring derivative bounds $\Psi_i^{(j)} \ll_j 1 + |\zeta_i|^j$. From here on, the proof is analogous to that of \cite[Corollary 17]{pascadi2024large} (starting with an application of the Kuznetsov formula for the cusps $\infty$ and $1/s$), except that we apply \cref{prop:large-sieve-multiplicative} instead of \cite[Theorem J]{pascadi2024large} in the exceptional spectrum; we use $Z = \max_i(1 + |\zeta_i|)$ in \cref{prop:large-sieve-multiplicative}, which disappears in the integral over $\zeta_1, \zeta_2$ from \cref{eq:final-kloosterman-integral}. Importantly, instead of \cite[(5.33)]{pascadi2023smooth} we use
\[
    X_1 := \frac{Q^2}{M_1^2 M_2},
\]
as in \cref{eq:watt-saving}. Combining this with the value of $X$ from \cref{eq:initial-X-factor} and the value of $X_2(\xi)$ from \cref{eq:x-saving-ass} (with $q = rs$) leads to a total exceptional factor of
\[
\begin{aligned}
    \left(1 + \frac{X}{\sqrt{X_1 X_2(0)}}\right)^{\theta_{\max}}
    &\ll
    \left(1 + \frac{CS\sqrt{R}}{\sqrt{M_1 M_2 N}} \frac{M_1 \sqrt{M_2}}{Q} \frac{\sqrt{N}}{\sqrt{RSY}} \right)^{\theta_{\max}} 
    \\ 
    &= 
    \left(1 + \frac{C\sqrt{M_1}}{R\sqrt{SY}} \right)^{\theta_{\max}} ,
\end{aligned}
\]
as in \cref{eq:kloosterman-rsm12nc}.
Other than this, the right-hand side of \cref{eq:kloosterman-rsm12nc} is identical to that of \cite[(5.31)]{pascadi2024large}, after inserting the follow-up bound from \cite[(5.32)]{pascadi2024large}.
\end{proof}

\section{Primes with triply-well-factorable weights} \label{sec:primes} 
Here we prove \cref{thm:primes}.$(i)$ rigorously, building on the arguments of Maynard \cite{maynard2025primes2}. Compared to the outline in \cref{sec:informal-overview}, we will essentially work in reverse, starting from bounds for multilinear forms of Kloosterman sums and building up to a convolution estimate in \cref{prop:well-fact-convolution}.

We begin with a bound for a sum like in \cref{eq:sum-kloosterman-overview}, which follows from \cref{prop:kloosterman-complete,prop:large-sieve-additive}.

\begin{lemma} \label{lem:consequence-Kloosterman}
Let $\eps > 0$, $a \in \Z \setminus \{0\}$, $1 \ll S, F, K, C, H \ll x^{O(1)}$, $\Phi_i(t)$ be smooth functions supported in $t \asymp 1$ with $\Phi_i^{(j)} \ll_j 1$, and  
\[
    \phi(h_1, h_2) := \Phi_1\left(\frac{h_1}{H}\right) \Phi_2\left(\frac{h_2}{H}\right) e(h_1\alpha_1 + h_2\alpha_2),
\]
where $\alpha_i \in \R/\Z$ have $\min_i T_H(\alpha_i) \ll_\eps x^\eps$ (recall \cref{eq:tn}). Then for any smooth function $\Phi(x_1, x_2, z)$ supported in $x_i, z \asymp 1$, satisfying $\partial_{x_1}^{j_1} \partial_{x_2}^{j_2} \partial_z^\ell \Phi(x_1, x_2, z) \ll_{j_1, j_2, \ell, \eps} x^{(j_1+j_2) \eps}$, one has
\[
\begin{aligned}
    &\sum_{s_1, s_2 \sim S} \left\vert
    \sum_{f,k}
    \sum_{\substack{h_1, h_2 \\ \ell = h_1 s_1 - h_2 s_2 \neq 0}} \phi(h_1, h_2)
    \sum_{(c, s_1s_2) = 1}
    \Phi\left(\frac{f}{F}, \frac{k}{K}, \frac{c}{C}\right)
    S(fk\bar{s_1 s_2}, a\ell; c)
    \right\vert
    \\
    &\ll_{\eps,a} x^{O(\eps)}
    \left(1 + \frac{C\sqrt{F}}{S^2\sqrt{\frac{H^2}{H+S}}}\right)^{\theta_{\max}} 
    \sqrt{H^2S^3 (H + S) F K}
    \left(\frac{C^2}{S}\left(FK + S^2\right)\left(H + S\right) + FKHS\right)^{1/2}.
\end{aligned}
\]
\end{lemma}

\begin{proof}
Let $\mK_0$ denote the sum in the left-hand side.
We first let $s_0 := (s_1, s_2)$, change variables $s_i \gets s_0 s_i$, $\ell \gets s_0 \ell$ for $i \in \{1, 2\}$, and put $s_0$ into dyadic ranges. This yields
\begin{equation} \label{eq:K-split}
    \mK_0 \ll x^{o(1)} \sup_{S_0 \ll S} \mK_1(S_0),
\end{equation}
where after simplifying $S(fk\bar{s_0^2 s_1 s_2}, a s_0 \ell; c) = S(fk\bar{s_0 s_1 s_2}, a\ell; c)$,
\[
    \mK_1 := \sum_{s_0 \sim S_0} \sum_{\substack{s_1, s_2 \sim S/s_0 \\ (s_1, s_2) = 1}} \Bigg\vert
    \sum_{f,k}
    \sum_{\substack{h_1, h_2 \\ \ell = h_1 s_1 - h_2 s_2 \neq 0}} \phi(h_1, h_2)
    \sum_{(c, s_0 s_1 s_2) = 1}
    \Phi\left(\frac{f}{F}, \frac{k}{K}, \frac{c}{C}\right)
    S(fk\bar{s_0 s_1 s_2}, a\ell; c)
    \Bigg\vert.
\]
We then put $n = |a\ell|$ and $r = s_0 s_1 s_2$ in dyadic ranges $n \sim \NN$, $r \sim \RR$, insert coefficients $\Psi(n/\NN)$ where $\Psi^{(j)} \ll_j 1$ and $\psi \equiv 1$ on $[1, 2]$, and use the divisor bound to write
\begin{equation} \label{eq:K1-split}
    \mK_1 \ll x^{o(1)} \sup_{\substack{\NN \ll_a HS/S_0 \\ \RR \asymp S^2 / S_0}}
    \mK_2(\NN, \RR),
\end{equation}
for
\begin{equation} \label{eq:K2}
\begin{aligned}
    \mK_2 := \sum_{r \sim \RR} 
    \max_{\substack{s_0 \sim S_0 \\ s_1, s_2 \sim S/s_0 \\ (s_1, s_2) = 1 \\ s_0 s_1 s_2 = r}}
    \Bigg\vert
    \sum_{f,k} 
    \sum_{n \sim \NN} 
    \sum_{\substack{h_1, h_2 \\ a(h_1 s_1 - h_2 s_2) = \pm n}} \phi(h_1, h_2)
    \sum_{(c, r) = 1}
    \Psi\left(\frac{n}{\NN}\right)
    \Phi\left(\frac{f}{F}, \frac{k}{K}, \frac{c}{C} \right)
    S(fk\bar{r}, \pm n; c)
    \Bigg\vert,
\end{aligned}
\end{equation}
where the supremum in \cref{eq:K1-split} includes the choice of the $\pm$ sign. If the maximum above is attained at some $s_1(r), s_2(r)$, we let
\[
    a_{n,r} := \sum_{\substack{h_1, h_2 \in \Z \\ a(h_1 s_1(r) - h_2 s_2(r)) = \pm n}}
    \phi(h_1, h_2)
    =
    \sum_{\substack{h_1, h_2 \in \Z \\ \pm a(h_1 s_1(r) - h_2 s_2(r)) = n}} \Phi_1\left(\frac{h_1}{H}\right) \Phi_2\left(\frac{h_2}{H}\right) e(h_1\alpha_1 + h_2\alpha_2).
\]
If the maximum is empty, we let $a_{n, r} := 0$. Then we can rewrite \cref{eq:K2} as
\[
    \mK_2 = \sum_{r \sim \RR}
    \Bigg\vert
    \sum_{f,k} 
    \sum_{n \sim \NN}
    a_{n,r}
    \sum_{(c, r) = 1}
    \Psi\left(\frac{n}{\NN}\right) \Phi\left(\frac{f}{F}, \frac{k}{K}, \frac{c}{C} \right)
    S(fk\bar{r}, \pm n; c)
    \Bigg\vert.
\]
By \cref{prop:large-sieve-additive}, the tuple $(r, \NN, x, (a_{n,r})_{n \sim \NN}, A_r, Y)$ satisfies \cref{ass:large-sieve}, where 
\[
    Y := \frac{\NN H}{|a|(H+S/S_0)(S/S_0) \min_i T_H(\alpha_i)},
    \qquad
    A_r := \left(\sum_{n \sim \NN} |a_{n,r}|^2\right)^{1/2} + \sqrt{\NN} \sqrt{\frac{H S_0}{S} + \frac{H^2 S_0^2}{S^2}}.
\]
Since $\min_i T_H(\alpha_i) \ll_\eps x^\eps$,
we further have
\[
    Y \gg_{\eps,a} x^{-\eps}\frac{\NN H S_0}{(H+S)S}.
\]
We can now apply \cref{prop:kloosterman-complete}; specifically, by \cref{eq:kloosterman-rsm12nc}, we obtain
\[
    \mK_2 \ll_{\eps,a} x^{O(\eps)}
    \left(1 + \frac{C\sqrt{F}}{\RR\sqrt{\frac{\NN H S_0}{(H+S)S}}}\right)^{\theta_{\max}} 
    \|A_r\|_2 \sqrt{\RR F K}
    \left(\frac{C^2}{\RR}(FK + \RR)(\NN + \RR) + FK\NN\right)^{1/2}.
\]
We claim that
\begin{equation} \label{eq:primes-l2-bound}
    \|A_r\|_2^2 \ll x^{o(1)} \NN(HS + H^2 S_0).
\end{equation}
Indeed, this follows from the definition of $A_r$ and the two bounds
\[
    \sum_{r \sim \RR} \NN \left(\frac{HS_0}{S} + \frac{H^2 S_0^2}{S^2}\right)
    \asymp 
    \frac{S^2}{S_0} \NN \left(\frac{HS_0}{S} + \frac{H^2 S_0^2}{S^2}\right)
    \asymp 
    \NN (HS + H^2 S_0),
\]
\[
\begin{aligned}
    \sum_{r \sim \RR} \sum_{n \sim \NN} |a_{n,r}|^2 
    &\ll
    \sum_{n \sim \NN}
    \sum_{s_0 \sim S_0} \sum_{\substack{s_1, s_2 \sim S/s_0 \\ (s_1, s_2) = 1}}
    \left( \sum_{\substack{h_1, h_2 \asymp H \\ h_1 s_1 - h_2 s_2 = \pm n/a}} 1 \right)^2
    \\
    &\ll 
    \sum_{n \sim \NN}
    \sum_{s_0 \sim S_0} 
    \sum_{\substack{s_1, s_2 \sim S/s_0 \\ (s_1, s_2) = 1}}
    \sum_{\substack{h_1, h_2 \asymp H \\ h_1 s_1 - h_2 s_2 = \pm n/a}}\, \sum_{\substack{h_1', h_2' \asymp H \\ s_1(h_1 - h_1') = s_2(h_2 - h_2')}} 1
    \\
    &\ll 
    \sum_{n \sim \NN}
    \sum_{s_0 \sim S_0} 
    \sum_{\substack{s_1 \sim S/s_0 \\ h_1 \asymp H}}
    \sum_{\substack{s_2 \sim S/s_0 \\ h_2 \asymp H \\ h_2 s_2 = h_1 s_1 \mp n/a}}
    \sum_{\substack{h_1' \asymp H \\ h_1' \equiv h_1 \pmod{s_2}}}\, 
    \sum_{\substack{h_2' \asymp H \\ s_1(h_1 - h_1') = s_2(h_2 - h_2')}} 1
    \\
    &\ll x^{o(1)}
    \NN S_0 \frac{S}{S_0} H \left(1 + \frac{H S_0}{S}\right)
    \quad 
    =
    x^{o(1)} \NN (HS + H^2 S_0).
\end{aligned}
\]
Using \cref{eq:primes-l2-bound}, we can further bound 
\[
\begin{aligned}
    \mK_2 &\ll_{\eps,a} x^{O(\eps)}
    \\
    &\times \left(1 + \frac{C\sqrt{F}}{\RR\sqrt{\frac{\NN H S_0}{(H+S)S}}}\right)^{\theta_{\max}} 
    \sqrt{\NN (HS + H^2 S_0) \RR F K}
    \left(\frac{C^2}{\RR}(FK + \RR)(\NN + \RR) + FK\NN\right)^{1/2},
\end{aligned}
\]
where the right-hand side is increasing in $\NN \ll_a HS/S_0$. Substituting this value of $\NN$ and $\RR \asymp S^2/S_0$, it follows from \cref{eq:K1-split} that
\[
\begin{aligned}
    \mK_1 
    &\ll_{\eps,a} 
    x^{O(\eps)}
    \left(1 + \frac{C\sqrt{F}}{\frac{S^2}{S_0}\sqrt{\frac{H^2}{H+S}}}\right)^{\theta_{\max}} 
    \\
    &\times
    \sqrt{\frac{HS}{S_0} (HS + H^2 S_0) \frac{S^2}{S_0} F K}
    \left(\frac{C^2 S_0}{S^2}\left(FK + \frac{S^2}{S_0}\right)\left(\frac{HS}{S_0} + \frac{S^2}{S_0}\right) + \frac{FKHS}{S_0}\right)^{1/2}.
\end{aligned}
\]
Since $\theta_{\max} < 1/2$, the right-hand side is seen to be decreasing in $S_0 \gg 1$; substituting $S_0$ with $1$ and plugging this into \cref{eq:K-split}, we obtain the desired bound for $\mK_0$.
\end{proof}

We use \cref{lem:consequence-Kloosterman} to deduce a power-saving bound for an exponential sum like in \cref{eq:sketch-kl-fractions} (before passing to the complementary divisor). This improves \cite[Lemma 7.1]{maynard2025primes2} by allowing larger ranges of $Q, R, S$ in \cref{eq:ranges-qrs}; as a technical difference, we require that $h$ lies in a smooth dyadic range. We note in passing that the case $h < 0$ follows immediately by changing $a \leftrightarrow -a$.

\begin{lemma}[Exponential sum bound for well-factorable weights] \label{lem:expo-bound-well-fact}
Let $a \in \Z \setminus \{0\}$, $d \in \Z_+$ with $(a, d) = 1$, $\eps, C > 0$, and $M, N, x, Q, R, S \gg 1$ satisfy $MN \asymp x$ and, with $\theta := 7/32$,
\begin{equation} \label{eq:ranges-qrs}
\begin{aligned}
    N^2 R^2 S &\le x^{1-8\eps}, 
    \\
    N^{\frac{2+\theta}{2-2\theta}} R S^{\frac{4-5\theta}{2-2\theta}}
    &\le 
    x^{1-16\eps},
    \\ 
    N^{\frac{1+\theta}{1-\theta}} Q^{\frac{1-3\theta}{1-\theta}} R^{2} S^{5} &\le 
    x^{2-32\eps}.
\end{aligned}
\end{equation}
Let $Q' \in [Q, 2Q]$, $1 \ll H \le x^{o(1)} Q R^2 S^2/M$, $B_i \gg 1$, and let $\msN \subset \Z_+^2$ be such that if $(u; v), (u'; v') \in \msN$, then $(u, v') = (u', v) = 1$. Finally, let $(\gamma_r)$, $(\lambda_s)$, $(\alpha_n)$ be $1$-bounded sequences, $\omega \in \R/\Z$ with $T_H(\omega) \ll x^{o(1)}$, and $\Phi(t)$ be a smooth function supported in $t \asymp 1$, with $\Phi^{(j)} \ll_j 1$ for $j \ge 0$. Then
\[
\begin{aligned}
    \sum_{\substack{Q \le q \le Q' \\ (q, a) = 1}}
    \sum_{r_1, r_2 \sim R}
    \sum_{\substack{s_1, s_2 \sim S \\ (r_1s_1, ar_2s_2) = 1 \\ (r_2 s_2, a q d r_1 s_1) = 1 \\ r_i s_i \le B_i}} 
    \frac{\gamma_{r_1} \lambda_{s_1} \bar{\gamma_{r_2}\lambda_{s_2}}}{r_1 r_2 s_1 s_2 q}
    \sum_{\substack{n_1, n_2 \sim N \\ n_1 \equiv n_2 \pmod{qd} \\ (n_1, n_2 q d r_1 s_1) = 1 \\ (n_2, n_1 q d r_2 s_2) = 1 \\ (n_1 r_2 s_2; n_2) \in \msN \\ |n_1 - n_2| \ge N/(\log x)^C}} 
    \alpha_{n_1} \bar{\alpha_{n_2}}&
    \\
    \times 
    \sum_{h \in \Z} e(h\omega)\, \Phi\left(\frac{h}{H}\right) 
    e\left(\frac{ah(n_1 - n_2) \bar{n_2 r_1 s_1 dq}}{n_1 r_2 s_2}\right)&
    \ll_{a,\eps,C}
    \frac{N^2}{Qx^\eps}.
\end{aligned}
\]
\end{lemma}

\begin{proof}
We closely follow the proof of \cite[Lemma 7.1]{maynard2025primes2}, taking $Q \ll N$ without loss of generality (otherwise the sum over $n_1, n_2$ vanishes). 
After the substitution $f d q = n_1 - n_2$, a separation of variables and an application of Cauchy--Schwarz in $f, n_1, n_2, r_1, r_2, s_2$, we reach the sum
\[
    \msW_4 :=
    \sum_{\substack{b, c, f \\ (b, c) = 1}} \Psi_0\left(\frac{b}{B}\right) \Psi_0\left(\frac{c}{C_0}\right)
    \Psi_0\left(\frac{f}{F_0}\right)
    \left\vert 
    \sum_{s \sim S} \lambda_s' \sum_{h \in \Z}
    e(h\omega)\, \Phi\left(\frac{h_1}{H}\right)
    e\left(\frac{ahf\bar{bs}}{c}\right)
    \right\vert^2,
\]
similar to \cite[p.\,23, third display]{maynard2025primes2}.
Here we also inserted a smooth majorant in the $f$ variable, where $\Psi_0$ is a compactly-supported nonnegative function satisfying $\Psi_0^{(j)} \ll_j 1$. As in \cite[p.\,23, second display]{maynard2025primes2}, the ranges $B, C_0, F_0$ satisfy
\begin{equation} \label{eq:bcf-bounds}
    B \ll NR, \qquad C_0 \ll NRS, \qquad F_0 \ll \frac{N}{Q},
\end{equation}
and as in \cite[(7.4)]{maynard2025primes2}, we need to show that $\msW_4 \ll_{\eps,a} x^{-6\eps} N^2 R^2 S^3$. Normally at this stage, we would expand the square in $\msW_4$, leading to a sum like in \cref{eq:incomplete-kloosterman-overview}, and then complete Kloosterman sums. But as outlined in \cref{subsec:improvements}, to achieve good savings in the complementary divisor ($f \sim F$) aspect, we will need to `move' $f$ to the other entry of the resulting Kloosterman sums. Towards this goal, we split the sum according to the value of $d = (f, c)$:
\begin{equation} \label{eq:W4-split}
    \msW_4 \le \sum_{1 \le d \ll x} \msW_5(d),
\end{equation}
where, after relaxing the constraint $(b, c) = 1$ to $(b, c/d) = 1$, substituting $(f, c) \gets (fd, cd)$, and letting
\begin{equation} \label{eq:cf-substitution}
    C := \frac{C_0}{d},
    \qquad\qquad 
    F := \frac{F_0}{d},
\end{equation}
we have
\[
    \msW_5 := \sum_{\substack{b, c, f \\ (bf, c) = 1}} \Psi_0\left(\frac{b}{B}\right) \Psi_0\left(\frac{c}{C}\right)
    \Psi_0\left(\frac{f}{F}\right)
    \left\vert 
    \sum_{\substack{s \sim S \\ (s, c) = 1}} \lambda_s' \sum_{h \in \Z}
    e(h\omega)\, \Phi\left(\frac{h_1}{H}\right)
    e\left(\frac{ahf\bar{bs}}{c}\right)
    \right\vert^2.
\]
Due to \cref{eq:W4-split} and $\sum_{1 \le d \ll x} \frac{1}{d} \ll \log x$, it is enough to show that
\begin{equation} \label{eq:W5-bound-to-show}
    \msW_5 \ll_{\eps,a} x^{-7\eps} \frac{N^2 R^2 S^3}{d}.
\end{equation}
Now let
\[
    \msW(x; c) := 
    \sum_{\substack{s \sim S \\ (s, c) = 1}} \lambda_s' \sum_{h \in \Z}
    e(h\omega)\, \Phi\left(\frac{h_1}{H}\right)
    e\left(\frac{ah\bar{xs}}{c}\right),
\] 
for $x \in (\Z/c\Z)^\times$, so we can write
\[
\begin{aligned}
    \msW_5 
    &= 
    \sum_{c} \Psi_0\left(\frac{c}{C}\right)
    \sum_{(f, c) = 1} \Psi_0\left(\frac{f}{F}\right)
    \sum_{(b, c) = 1} \Psi_0\left(\frac{b}{B}\right)
    |\msW(b\bar{f}; c)|^2
    \\
    &= 
    \sum_c \Psi_0\left(\frac{c}{C}\right)
    \sum_{(f, c) = 1} \Psi_0\left(\frac{f}{F}\right)
    \sum_{x \in (\Z/c\Z)^\times}
    |\msW(x; c)|^2
    \sum_{b \equiv xf \pmod{c}} \Psi_0\left(\frac{b}{B}\right)
    \\
    &\le
    \sum_c \Psi_0\left(\frac{c}{C}\right)
    \sum_{f} \Psi_0\left(\frac{f}{F}\right)
    \sum_{x \in (\Z/c\Z)^\times}
    |\msW(x; c)|^2
    \sum_{b \equiv xf \pmod{c}} \Psi_0\left(\frac{b}{B}\right),
\end{aligned}
\]
where we dropped the restriction $(f, c) = 1$ in the last line. Expanding the square and swapping sums, we get
\[
\begin{aligned}
    \msW_5 \le 
    \sum_{\substack{s_1, s_2 \sim S \\ (s_1 s_2, a) = 1}} \lambda'_{s_1} \bar{\lambda'_{s_2}}
    \sum_{h_1, h_2} \phi(h_1, h_2)
    \sum_{f} \Psi_0\left(\frac{f}{F}\right)
    \sum_{(c, s_1s_2) = 1} \Psi_0\left(\frac{c}{C}\right)
    \\
    \times
    \sum_{x \in (\Z/c\Z)^\times}
    e\left(\frac{a\ell \bar{s_1s_2 x}}{c}\right)
    \sum_{b \equiv xf \pmod{c}} \Psi_0\left(\frac{b}{B}\right),
\end{aligned}
\]
where
\begin{equation} \label{eq:smooth-in-h}
    \ell := h_1 s_1 - h_2 s_2,
    \qquad
    \phi(h_1, h_2) := e((h_1 - h_2)\omega)\, \Phi\left(\frac{h_1}{H}\right) \bar{\Phi\left(\frac{h_2}{H}\right)}.
\end{equation}
Splitting the sum above into the terms with $\ell = 0$ and $\ell \neq 0$, we have
\begin{equation} \label{eq:W5-split}
    \msW_5
    \le \msW_{\ell=0} + \msW_{\ell \neq 0}.
\end{equation}
In light of \cref{eq:bcf-bounds,eq:cf-substitution}, the diagonal terms contribute at most
\begin{equation} \label{eq:W-diagonal-bound}
\begin{aligned}
    \msW_{\ell = 0}
    &\ll 
    \sum_{s_1, s_2 \sim S} 
    \sum_{\substack{h_1, h_2 \asymp H \\ h_1 s_1 = h_2 s_2}}
    \sum_{f \asymp F}
    \sum_{c \asymp C}
    \sum_{b \asymp B} \sum_{x \in (\Z/c\Z)^\times} \one_{b \equiv xf \pmod{c}}
    \\
    &\ll 
    \sum_{s_1, s_2 \sim S} 
    \sum_{\substack{h_1, h_2 \asymp H \\ h_1 s_1 = h_2 s_2}}
    \sum_{f \asymp F}
    \sum_{c \asymp C}
    \sum_{b \asymp B} (b, c, f)
    \\
    &\ll 
    x^{o(1)}
    S H F C B
    \\
    &\ll
    x^{o(1)} S \frac{Q R^2 S^2}{x/N} \frac{N}{dQ} \frac{NRS}{d} NR
    \ll
    \frac{N^4 R^4 S^4}{d^2 x^{1-\eps}},
\end{aligned}
\end{equation}
and this is acceptable in \cref{eq:W5-bound-to-show} provided that
\[
    N^2 R^2 S \ll_\eps x^{1-8\eps},
\]
which we assumed in \cref{eq:ranges-qrs}. For the off-diagonal terms, which roughly correspond to \cref{eq:incomplete-kloosterman-overview}, we complete the inner sum over $b$ via \cref{lem:truncated-poisson} to obtain
\begin{equation} \label{eq:W-off-diagonal-split}
\begin{aligned}
    \msW_{\ell \neq 0} \ll |\msW_6|
    +
    O(x^{-90})
    +
    x^{o(1)} \sup_{\substack{K \ll x^{o(1)} B^{-1} C \\ \Psi^{(k)} \ll_k 1 \\ u \asymp 1}} |\msW_7(K, u)|,
\end{aligned}
\end{equation}
where $\Psi$ is a smooth function supported in $(\frac{1}{2}, 2)$,
\[
\begin{aligned}
    \msW_6 := \sum_{\substack{s_1, s_2 \sim S \\ (s_1 s_2, a) = 1}} \lambda'_{s_1} \bar{\lambda'_{s_2}}
    \sum_{h_1, h_2} \phi(h_1, h_2)
    \sum_{f} \Psi_0\left(\frac{f}{F}\right)
    \sum_{(c, s_1s_2) = 1} \Psi_0\left(\frac{c}{C}\right)
    \sum_{x \in (\Z/c\Z)^\times}
    e\left(\frac{a\ell \bar{s_1s_2 x}}{c}\right)
    \\
    \times \frac{B}{c} \hat{\Psi_0}(0)
\end{aligned}
\]
is the contribution of the principal frequency, and
\[
\begin{aligned}
    \msW_7 := \sum_{\substack{s_1, s_2 \sim S \\ (s_1 s_2, a) = 1}} \lambda'_{s_1} \bar{\lambda'_{s_2}}
    \sum_{\substack{h_1, h_2 \\ \ell = h_1 s_1 - h_2 s_2 \neq 0}} \phi(h_1, h_2)
    \sum_{f} \Psi_0\left(\frac{f}{F}\right)
    \sum_{(c, s_1s_2) = 1} \Psi_0\left(\frac{c}{C}\right)
    \sum_{x \in (\Z/c\Z)^\times}
    e\left(\frac{a\ell \bar{s_1s_2 x}}{c}\right)
    \\
    \times 
    \frac{B}{C} \Psi_0\left(\frac{uc}{C}\right) \sum_k \Psi\left(\frac{|k|}{K}\right) e\left(-k\frac{uB}{C}\right) e\left(\frac{xfk}{c}\right).
\end{aligned}
\]
We first bound $\msW_6$ using the Ramanujan sum bound (see \cref{lem:weil}), \cref{eq:bcf-bounds,eq:cf-substitution}:
\begin{equation} \label{eq:W6-bound}
\begin{aligned}
    \msW_6 
    &=
    \sum_{\substack{s_1, s_2 \sim S \\ (s_1 s_2, a) = 1}} \lambda'_{s_1} \bar{\lambda'_{s_2}}
    \sum_{h_1, h_2} \phi(h_1, h_2)
    \sum_{f} \Psi_0\left(\frac{f}{F}\right)
    \sum_{(c, s_1s_2) = 1} \Psi_0\left(\frac{c}{C}\right)
    S(0, a\ell; c)
    \frac{B}{c} \hat{\Psi_0}(0)
    \\
    &\ll
    \sum_{s_1, s_2 \sim S}
    \sum_{h_1, h_2 \asymp H}
    \sum_{f \asymp F}
    \sum_{c \asymp C}
    (a\ell, c)
    \frac{B}{c}
    \\
    &\ll_a 
    x^{o(1)} S^2 H^2 F B
    \\
    &\ll 
    x^{o(1)} S^2 \left(\frac{QR^2 S^2}{x/N}\right)^2 \frac{N}{dQ} NR 
    \ll_\eps 
    x^\eps \frac{Q R^5 S^6 N^4}{dx^2}.
\end{aligned}
\end{equation}
This is acceptable in \cref{eq:W5-bound-to-show} (i.e., $\ll_\eps x^{-7\eps} N^2 R^2 S^3/d$) provided that
\[
    N^2 Q R^3 S^3 \ll x^{2-8\eps},
\]
which follows from $Q \ll N$ and the first and third assumptions in \cref{eq:ranges-qrs}:
\[
\begin{aligned}
    N^2 Q R^3 S^3 
    &\ll N^3 R^3 S^3 
    \\
    &\le 
    \left(N^2 R^2 S\right)^{4/3} \cdot \left(N R^2 S^5\right)^{1/3}
    \le
    \left(x^{1-8\eps}\right)^{4/3} \left(x^{2-32\eps}\right)^{1/3}
    <
    x^{2-8\eps}.
\end{aligned}
\]
We are left to consider $\msW_7$, which roughly corresponds to the sum in \cref{eq:sum-kloosterman-overview}, and can be rewritten as
\[
\begin{aligned}
    \msW_7 = \frac{B}{C} \sum_{\substack{s_1, s_2 \sim S \\ (s_1 s_2, a) = 1}} \lambda'_{s_1} \bar{\lambda'_{s_2}}
    \sum_{f,k} \Psi_0\left(\frac{f}{F}\right)
    \Psi\left(\frac{|k|}{K}\right) e\left(-k\frac{uB}{C}\right)
    \sum_{\substack{h_1, h_2 \\ \ell = h_1 s_1 - h_2 s_2 \neq 0}} \phi(h_1, h_2)
    \\
    \times 
    \sum_{(c, s_1s_2) = 1}
    \Psi_0\left(\frac{c}{C}\right) \Psi_0\left(\frac{uc}{C}\right)
    S(fk\bar{s_1 s_2}, a\ell; c).
\end{aligned}
\]
We can now apply \cref{lem:consequence-Kloosterman} with the smooth weight
\[
    \Phi\left(x_1, x_2, z\right) := \Psi_0\left(x_1\right)
    \Psi\left(x_2\right) e\left(\mp x_2\frac{uKB}{C}\right)  \Psi_0\left(z\right) \Psi_0\left(uz\right),
\]
once for each choice of the $\pm$ sign (corresponding to the sign of $k$; note that $S(-fk\bar{s_1 s_2}, a\ell; c) = S(fk\bar{s_1 s_2}, -a\ell; c)$, so one can transfer the sign change to $a$ without loss of generality). This $\Phi$ is compactly supported and satisfies $\partial_{x_1}^{j_1} \partial_{x_2}^{j_2} \partial_z^\ell \Phi(x_1, x_2, z) \ll_{j_1,j_2,k,\ell} (KB/C)^{j_2} \ll x^{o(j_2)}$, where we used $K \ll x^{o(1)} B^{-1} C$ by \cref{eq:W-off-diagonal-split}. Since $\theta_{\max} \le \tfrac{7}{32} = \theta$, we can bound
\[
\begin{aligned}
    \msW_7 &\ll_a x^{o(1)} \frac{B}{C} 
    \left(1 + \frac{C\sqrt{F}}{S^2\sqrt{\frac{H^2}{H+S}}}\right)^{\theta} 
    \\
    &\times
    \sqrt{H^2S^3 (H + S) F K}
    \left(\frac{C^2}{S}\left(FK + S^2\right)\left(H + S\right) + FKHS\right)^{1/2}.
\end{aligned}
\]
At this point we note that by \cref{eq:ranges-qrs},
\[
    H \le x^{o(1)} \frac{NQ R^2 S^2}{x} 
    \ll 
    x^{o(1)}
    \frac{N^2 R^2 S}{x} S 
    \ll 
    x^{o(1)} S.
\]
Using this and the fact that $K \ll x^{o(1)} B^{-1} C$ from \cref{eq:W-off-diagonal-split}, our bound for $\msW_7$ simplifies to
\[
\begin{aligned}
    \msW_7 &\ll_a x^{o(1)} \frac{B}{C} 
    \left(1 + \frac{C\sqrt{F}}{S^{3/2} H}\right)^{\theta} 
    \sqrt{H^2S^4 F C / B}
    \left(C^2\left(FC/B + S^2\right) + FCHS/B\right)^{1/2}
    \\
    &\ll 
    x^{o(1)} 
    \left(1 + \frac{C\sqrt{F}}{S^{3/2} H}\right)^{\theta} 
    \sqrt{H^2S^4 F}
    \left(C\left(FC + BS^2\right) + FHS\right)^{1/2}.
\end{aligned}
\]
Plugging in the bounds for $B, C, F$ from \cref{eq:cf-substitution} and \cref{eq:bcf-bounds}, we are left with
\[
\begin{aligned}
    \msW_7
    &\ll_a \frac{x^{o(1)}}{d}
    \left(1 + \frac{C_0\sqrt{F_0}}{S^{3/2} H}\right)^{\theta} 
    \sqrt{H^2S^4 F}
    \left(C_0\left(FC + BS^2\right) + F_0HS\right)^{1/2}
    \\
    &\ll_a 
    \frac{x^{o(1)}}{d}
    \left(1 + \frac{NRS \sqrt{N/Q}}{S^{3/2} H}\right)^{\theta} 
    \sqrt{H^2S^4N/Q}
    \left(NRS \left((N/Q)NRS + NRS^2\right) + (N/Q)HS\right)^{1/2}
    \\
    &= 
    \frac{x^{o(1)}}{d}
    \left(1 + \frac{N^{3/2} R}{H \sqrt{SQ}}\right)^{\theta} 
    \frac{H S^{5/2} N}{Q}
    \left(NR^2S \left(N + QS\right) + H\right)^{1/2}.
\end{aligned}
\]
Finally, noting that the right-hand side is increasing in $H \le x^{o(1)} NQR^2 S^2/x$, we get
\[
\begin{aligned}
    \msW_7 
    &\ll_a \frac{x^{o(1)}}{d}
    \left(1 + \frac{x\sqrt{N}}{R S^{5/2} Q^{3/2}}\right)^{\theta} 
    \frac{N^2R^2S^{9/2}}{x}
    \left(N^2R^2S + NQR^2S^2 + \frac{NQR^2 S^2}{x}\right)^{1/2}
    \\
    &\ll_a
    \frac{x^{o(1)}}{d}
    \left(1 + \frac{x\sqrt{N}}{R S^{5/2} Q^{3/2}}\right)^{\theta} 
    \frac{N^{5/2} R^3 S^5}{x}
    \left(N + QS\right)^{1/2}
    \\
    &\ll_{\eps,a} 
    \frac{x^\eps}{d}
    \left(\frac{N^3 R^3 S^5}{x} \left(1 + \left(\frac{x\sqrt{N}}{R S^{5/2}} \right)^{\theta}\right) + \frac{N^{5/2} Q^{1/2} R^3 S^{11/2}}{x} \left(1 + \left(\frac{x\sqrt{N}}{R S^{5/2} Q^{3/2}} \right)^{\theta}\right)\right),
\end{aligned}
\]
where we omitted a factor of $Q^{-3\theta/2}$ in the first term of the expansion.
For this to be acceptable in \cref{eq:W5-bound-to-show} (i.e., $\ll_{\eps,a} x^{-7\eps} N^2 R^2 S^3 / d$), we need the following restrictions:
\[
\begin{gathered}
    N R S^2 \ll x^{1-8\eps},
    \qquad\qquad
    N^{2+\theta} R^{2-2\theta} S^{4-5\theta}
    \ll 
    x^{2-2\theta-16\eps},
    \\
    N Q R^2 S^5 \ll x^{2-16\eps},
    \qquad\qquad 
    N^{1+\theta} Q^{1-3\theta} R^{2-2\theta} S^{5-5\theta} \ll 
    x^{2-2\theta-16\eps}.
\end{gathered}
\]
All of these restrictions follow from \cref{eq:ranges-qrs}, $Q \ll N$, and $\theta \le 1/2$, as shown below:
\[
\begin{aligned}
    NRS^2 &\le \left(N^2 R^2 S \cdot N R^2 S^5 \right)^{1/3}
    \le 
    \left(x^{1-8\eps} \cdot x^{2-32\eps}\right)^{1/3}
    < 
    x^{1-8\eps},
    \\
    N Q R^2 S^5
    &\ll 
    \left(\frac{N}{Q}\right)^{\frac{2\theta}{1-\theta}} N Q R^2 S^5
    =
    N^{\frac{1+\theta}{1-\theta}} Q^{\frac{1-3\theta}{1-\theta}} R^2 S^5
    < 
    x^{2-16\eps},
    \\
    N^{2+\theta} R^{2-2\theta} S^{4-5\theta} 
    &= 
    \left(N^{\frac{2+\theta}{2-2\theta}} R S^{\frac{4-5\theta}{2-2\theta}}\right)^{2-2\theta}
    \le 
    (x^{1-16\eps})^{2-2\theta}
    \le 
    x^{2-2\theta-16\eps},
    \\
    N^{1+\theta} Q^{1-3\theta} R^{2-2\theta} S^{5-5\theta}
    &= 
    \left(N^{\frac{1+\theta}{1-\theta}} Q^{\frac{1-3\theta}{1-\theta}} R^{2} S^5\right)^{1-\theta}
    \le 
    (x^{2-32\eps})^{1-\theta}
    \le 
    x^{2-2\theta-16\eps}.
\end{aligned}
\]
In light of \cref{eq:W5-split,eq:W-diagonal-bound,eq:W-off-diagonal-split,eq:W6-bound}, this establishes \cref{eq:W5-bound-to-show} and completes our proof.
\end{proof}

Our next result is a convolution estimate corresponding to \cref{eq:conv-estimate-primes}, which improves \cite[Proposition 7.2]{maynard2025primes2}; to state it, we recall the Siegel--Walfisz condition from \cite[Definition 3]{maynard2025primes2}.

\begin{definition}[Siegel--Walfisz sequences] \label{def:siegel-walfisz}
A complex sequence $(a_n)_{n \sim N}$ is said to obey the \emph{Siegel--Walfisz condition} iff one has
\[
    \sum_{\substack{n \sim N \\ (n, d) = 1}} a_n \left(\one_{n \equiv a \pmod{q}} - \frac{\one_{(n, q) = 1}}{\varphi(q)} \right)
    \ll_A 
    \tau(d)^{O(1)} \frac{N}{(\log N)^A},
\]
for all $d, q \in \Z_+$, $a \in \Z$ with $(a, q) = 1$, and all $A > 1$.
\end{definition}

\begin{proposition}[Triply-well-factorable convolution estimate] \label{prop:well-fact-convolution}
Let $a \in \Z \setminus \{0\}$, $A, \eps > 0$, and $M, N, x, Q_1, Q_2, Q_3 \gg 1$ satisfy $MN \asymp x$ and
\begin{equation} \label{eq:ranges-q1q2q3-revisited}
\begin{aligned}
    Q_1 &\le \frac{N}{x^\eps}, 
    \\
    N^2 Q_2 Q_3^2 &\le x^{1-15\eps},
    \\ 
    N^2 Q_2^5 Q_3^{2} &\le 
    x^{2-40\eps}.
\end{aligned}
\end{equation}
Let $(\alpha_n), (\beta_m)$ be $1$-bounded complex sequences, such that $(\alpha_n)$ is supported on $P^-(n) \ge z_0 := x^{1/(\log\log x)^3}$ and satisfies the Siegel--Walfisz condition from \cref{def:siegel-walfisz}. Then for any $1$-bounded complex sequences $(\gamma_{q_1}), (\lambda_{q_2}), (\nu_{q_3})$ supported on $(q_i, a) = 1$, one has
\[
    \sum_{q_1 \sim Q_1} \gamma_{q_1} \sum_{q_2 \sim Q_2} \lambda_{q_2} \sum_{q_3 \sim Q_3} \nu_{q_3} \sum_{n \sim N} \alpha_n \sum_{m \sim M} \beta_m \left(\one_{mn \equiv a \pmod{q}} - \frac{\one_{(mn, q) = 1}}{\varphi(q)}\right)
    \ll_{\eps,A,a} 
    \frac{x}{(\log x)^A}.
\]
\end{proposition}

\begin{proof}
From \cref{eq:ranges-q1q2q3-revisited} we can deduce the slightly-weaker system of inequalities (with $\theta := 7/32$)
\begin{equation} \label{eq:ranges-q1q2q3}
\begin{aligned}
    Q_1 &\le \frac{N}{x^\eps}, 
    \\
    N^2 Q_2 Q_3^2 &\le x^{1-9\eps},
    \\
    N^{\frac{2+\theta}{2-2\theta}} Q_2^{\frac{4-5\theta}{2-2\theta}} Q_3
    &\le 
    x^{1-17\eps},
    \\ 
    N^{\frac{1+\theta}{1-\theta}} Q_1^{\frac{1-3\theta}{1-\theta}} Q_2^5 Q_3^{2} &\le 
    x^{2-33\eps},
\end{aligned}
\end{equation}
which will be enough for this proof.
Indeed, the third bound in \cref{eq:ranges-q1q2q3} follows from \cref{eq:ranges-q1q2q3-revisited} and $(2+\theta)/(2-2\theta) = 71/50$, $(4-5\theta)/(2-2\theta) = 93/50 < 187/100$, since
\[
    N^{\frac{2+\theta}{2-2\theta}} Q_2^{\frac{4-5\theta}{2-2\theta}} Q_3 
    \le (N^2 Q_2 Q_3^2)^{21/50} (N^2 Q_2^5 Q_3^2)^{29/100} \le x^{1-17.9\eps}.
\]
Similarly, the fourth bound in \cref{eq:ranges-q1q2q3} follows from \cref{eq:ranges-q1q2q3-revisited} since
\[
    N^{\frac{1+\theta}{1-\theta}} Q_1^{\frac{1-3\theta}{1-\theta}} Q_2^5 Q_3^2 \le 
    N^{\frac{1+\theta}{1-\theta}} N^{\frac{1-3\theta}{1-\theta}} Q_2^5 Q_3^2
    =
    N^2 Q_2^5 Q_3^2
    \le
    x^{2-40\eps}.
\]

We now closely follow the proof of \cite[Proposition 7.2]{maynard2025primes2}, which begins by factoring out the $z_0$-smooth parts of $q_2$ and $q_3$ and applying \cite[Proposition 5.8]{maynard2025primes2} (at that step we use $N \ge Q_1 x^\eps$). With $y_0 := x^{1/\log \log x}$, $D \le y_0^2$ and $DR \asymp Q_2 Q_3$, it remains to bound $|\msE_1| + |\msE_2| \ll \frac{N^2}{DQ_1 y_0}$,
where $\msE_i$ are the exponential sums from \cite[p.\,26]{maynard2025primes2}. As in \cite[p.\,27]{maynard2025primes2}, the contribution of $\msE_1$ is acceptable provided that
\[
    N^{3/2}Q_2 Q_3 \le x^{1-2\eps},
    \qquad\quad
    Q_1 Q_2 Q_3 \le x^{1-2\eps},
\]
both of which follow easily from \cref{eq:ranges-q1q2q3}. As in \cite[p.\,27]{maynard2025primes2}, to handle $\msE_2$ it suffices to bound another exponential sum $\msE_3$ by $\msE_3 \ll_\eps N^2/(Q_1 x^{\eps/10})$. We note that the range of the $h$ variable can be extended to all $h \in \Z \setminus \{0\}$, since the contribution of $|h| > H_2 := (\log x)^5 QDR^2/M$ is negligible.

We apply a close variant of \cite[Lemma 5.9]{maynard2025primes2} (which is \cite[Lemma 14.5]{maynard2025primes}) to $\msE_3$. Specifically, in the proof of \cite[Lemma 14.5]{maynard2025primes}, we omit the step of applying partial summation in the $h$ variable; instead, we follow the proof of \cref{lem:truncated-poisson} and put $|h|$ in \emph{smooth} dyadic ranges $\Psi(|h|/H')$, bound the contribution of $H' > H_2$ using the decay of $\psi_0$, separate the variables $h$ and $q d r_1 r_2$ variables via a Fourier integral, and fix the integration variable $u \asymp 1$. This produces a smooth factor $\tilde\psi_0(u q d r_1 r_2 / (QDR^2))$, which we eliminate by partial summation in $q, r_1, r_2$, leading to the exponential sum $\msE'$ below. But first, we need to verify the conditions \cite[(5.1), (5.2)]{maynard2025primes2} from \cite[Lemma 5.9]{maynard2025primes2}, 
\[
    Q_1 Q_2 Q_3 \le x^{2/3},
    \qquad \qquad
    Q_1 (Q_2 Q_3)^2 \ll M x^{1-2\eps},
\]
both of which follow from \cref{eq:ranges-q1q2q3}. Indeed, recalling that $MN \asymp x$, we have
\begin{equation} \label{eq:primes-easy-bounds}
\begin{aligned}
    Q_1 Q_2 Q_3 &\le \left(N^2 Q_2 Q_3^2\right)^{3/8} \left(N^{\frac{1+\theta}{1-\theta}} Q_1^{\frac{1-3\theta}{1-\theta}} Q_2^5 Q_3^2\right)^{1/8}
    \le 
    \left(x^{1-9\eps}\right)^{3/8} \left(x^{2-33\eps}\right)^{1/8}
    < x^{5/8},
    \\
    Q_1 (Q_2 Q_3)^2 &\le \frac{\left(N Q_2 Q_3\right)^2}{N}
    \le 
    \frac{x^{2-18\eps}}{N}
    \ll
    \frac{\left(N^2 Q_2 Q_3^2\right)^2}{N} \ll M x^{1-18\eps}.
\end{aligned}
\end{equation}
As in \cite[p.\,27]{maynard2025primes2}, it remains to bound an exponential sum $\msE'$ (roughly corresponding to \cref{eq:sketch-kl-fractions}, before passing to the complementary divisor of $q$) by 
\begin{equation} \label{eq:primes-expo-to-bound}
    \msE' \ll_\eps \frac{N^2}{Q_1 x^{\eps/2}},
\end{equation}
but we now have
\[
    \msE' := 
    \sum_{\substack{Q_1 \le q \le Q_1' \\ (q, a) = 1}}
    \sum_{\substack{R \le r_1 \le R_1 \\ R \le r_2 \le R_2 \\ (r_1, ar_2) = 1 \\ (r_2, aqdr_1) = 1}}
    \frac{\lambda_{r_1} \bar{\lambda_{r_2}}}{qdr_1 r_2}
    \sum_{\substack{n_1, n_2 \sim N \\ n_1 \equiv n_2 \pmod{qd} \\ (n_1, qdr_1n_2) = 1 \\ (n_2, qdr_2 n_1) = 1 \\ (n_1 r_2, n_2) \in \mN}}
    \alpha_{n_1} \bar{\alpha_{n_2}}
    \sum_{h \in \Z} e(h\omega)\,
    \Psi\left(\frac{|h|}{H'}\right) 
    e\left(\frac{ah\bar{n_2qdr_1}(n_1-n_2)}{n_1r_2}\right)
    ,
\]
where $\Psi : (\frac{1}{2}, 2) \to \C$ is compactly-supported with $\Psi^{(j)} \ll_j 1$, $Q_1' \le 2Q_1$, $R_1, R_2 \le 2R$, $H' \le H_2 = (\log x)^5 QDR^2/M$, and 
\[
    \omega := u \frac{M}{QDR^2} \ll x^{o(1)} H_2^{-1}
    \qquad\quad 
    \Rightarrow 
    \qquad\quad
    T_{H'}(\omega) \ll x^{o(1)}.
\]
All that changed from the sum $\msE'$ from \cite[p.\,27]{maynard2025primes2} is that $h$ now lies in a smooth dyadic range, which is ultimately required by our large sieve inequality in \cref{prop:large-sieve-additive}. After expanding $(\lambda_r)$ and fixing one of $x^{o(1)}$ choices of $q_2'', q_3''$, \cref{lem:expo-bound-well-fact} gives the desired bound in \cref{eq:primes-expo-to-bound} provided that
\[
\begin{aligned}
    N^2 {Q_3'}^2 Q_2' &\le x^{1-9\eps}, 
    \\
    N^{\frac{2+\theta}{2-2\theta}} Q_3 {Q_2'}^{\frac{4-5\theta}{2-2\theta}}
    &\le 
    x^{1-17\eps},
    \\ 
    N^{\frac{1+\theta}{1-\theta}} Q_1^{\frac{1-3\theta}{1-\theta}} {Q_3'}^{2} {Q_2'}^{5} &\le 
    x^{2-33\eps},
\end{aligned}
\]
for all $Q_2' \le Q_2$ and $Q_3' \le Q_3$. These bounds follow directly from \cref{eq:ranges-q1q2q3}, completing our proof.
\end{proof}

We are now ready to establish a Type II estimate for triply-well-factorable weights, improving \cite[Proposition 4.1]{maynard2025primes2}; recall the triply-well-factorable condition from \cref{def:triply-well-fact}.

\begin{remark}
The system of inequalities in \cref{eq:ranges-q1q2q3} is significantly more flexible than the triply-well-factorable condition from \cref{def:triply-well-fact}. In particular, in \cref{sec:primes-linear-sieve}, we will need to use \cref{prop:well-fact-convolution} directly rather than the result below.
\end{remark}

\begin{proposition}[Triply-well-factorable Type II estimate] \label{prop:well-fact-type-II}
    Let $a \in \Z \setminus \{0\}$, $A, \eps > 0$, and $(\lambda_q)$ be triply-well-factorable of level $Q \le x^{5/8-100\eps}$. Let $M, N, x \gg 1$ with $MN \asymp x$ and
    \[
        x^\eps \le N \le x^{3/8}.
    \]
    Let $(\alpha_n)$, $(\beta_m)$ be divisor-bounded complex sequences, such that $(\alpha_n)$ is supported on $P^-(n) \ge z_0 := z_0 := x^{1/(\log\log x)^3}$ and satisfies the Siegel--Walfisz condition from \cref{def:siegel-walfisz}. Then for any interval $\mI \subset [x, 2x]$, one has
    \[
        \sum_{q \le Q} \lambda_q 
        \sum_{\substack{m \sim M \\ n \sim N \\ mn \in \mI}} \alpha_n \beta_m \left(\one_{mn \equiv a \pmod{q}} - 
        \frac{\one_{(mn,q) = 1}}{\varphi(q)}
        \right) 
        \ll_{\eps,A,a} \frac{x}{(\log x)^A}.
    \]
\end{proposition}

\begin{proof}
We follow the proof on \cite[p.\,28]{maynard2025primes2}, reducing to the case of $1$-bounded coefficients via \cite[Lemma 5.1]{maynard2025primes2} and separating the variables $mn$ (from the condition $mn \in \mI$) via \cite[Lemma 5.2]{maynard2025primes2}. We may assume that $x^{1/2-\eps} \le Q \le x^{5/8-100\eps}$, since the Bombieri--Vinogradov theorem yields the result for $Q \le x^{1/2-\eps}$. The only difference is that we choose the ranges
\[
    Q_1 := \frac{N}{x^{\eps}},
    \qquad 
    Q_2 := \frac{Q^2}{x^{1-21\eps}},
    \qquad 
    Q_3 := \frac{x^{1-20\eps}}{NQ},
\] 
where we note that $Q_i \ge 1$ since $x^{1/2-\eps} \le Q \le x^{5/8-100\eps}$ and $x^\eps \le N \le x^{3/8}$.

We claim that any $Q_1' \le Q_1$, $Q_2' \le Q_2$, $Q_3' \le Q_3$ obey the constraints in \cref{eq:ranges-q1q2q3-revisited}; indeed, using $Q \le x^{5/8-100\eps}$, we have
\[
    N^2 Q_2 Q_3^2 = x^{1-19\eps} \le x^{1-15\eps},
\]    
\[
    N^2 Q_2^5 Q_3^{2}
    =
    \frac{Q^8}{x^{3-65\eps}}
    \le 
    x^{2-40\eps}.
\]
After decomposing the triply-well-factorable weights as in \cref{def:triply-well-fact} and putting $q_i$ in dyadic ranges, \cref{prop:well-fact-convolution} yields the result.
\end{proof}


\begin{proof}[Proof of \cref{thm:primes}.$(i)$]
This is completely analogous to the proof on \cite[p.\,10]{maynard2025primes2}, decomposing the von Mangoldt function via the Heath-Brown identity \cite[Lemma 4.3]{maynard2025primes2}, and noting that $x^{1/3} \le x^{3/8}$. After \cref{prop:well-fact-type-II} handles the critical ranges where some $M_i$ or $N_i$ lies in $[x^\eps, x^{3/8}]$, \cite[Lemma 4.4]{maynard2025primes2} handles the case of one large smooth factor $N_i > x^{3/8}$, while \cite[Proposition 4.2]{maynard2025primes2} handles the case of two large smooth factors $N_i > x^{3/8}$. 
\end{proof}

\section{Primes with linear sieve weights} \label{sec:primes-linear-sieve}

Here we work with the upper-bound and lower-bound linear sieve weights, using \cref{prop:well-fact-convolution}. We first recall some definitions from \cite{maynard2025primes2,lichtman2023primes}.

\begin{definition}[Linear sieve support] \label{def:linear-sieve}
For $D \ge 1$, consider the sets of positive integers
\begin{equation} \label{eq:linear-sieve-support}
\begin{aligned}
    \mD^+(D) &:= \left\{ p_1 \cdots p_r : p_1 \ge \cdots \ge p_r \text{ primes},\, p_1 \cdots p_{j-1} p_j^3 \le D \text{ for odd } j \le r \right\},
    \\
    \mD^-(D) &:= \left\{ p_1 \cdots p_r : p_1 \ge \cdots \ge p_r \text{ primes},\, p_1 \cdots p_{j-1} p_j^3 \le D \text{ for even } j \le r,\, p_1^2 \le D \right\},
    \\
    \mDw(D) &:= \left\{ p_1 \cdots p_r : p_1 \ge \cdots \ge p_r \text{ primes},\, p_1 \cdots p_{j-1} p_j^2 \le D \text{ for } j \le r \right\},
\end{aligned}
\end{equation}
which have
\[
    \mD^\pm(D) \subset \mDw(D).
\]
Similarly, for $r \in \Z_+$, we define the sets of vectors 
\begin{equation} \label{eq:linear-sieve-vectors-support}
\begin{aligned}
    \bD_r^+(D) &:= \left\{ (P_1, \ldots, P_r) : P_1 \ge \cdots \ge P_r \ge 1,\, P_1 \cdots P_{j-1} P_j^3 \le D \text{ for odd } j \le r \right\},
    \\
    \bD_r^-(D) &:= \left\{ (P_1, \ldots, P_r) : P_1 \ge \cdots \ge P_r \ge 1,\, P_1 \cdots P_{j-1} P_j^3 \le D \text{ for even } j \le r,\, P_1^2 \le D \right\},
    \\
    \bDw_r(D) &:= \left\{ (P_1, \ldots, P_r) : P_1 \ge \cdots \ge P_r \ge 1,\, P_1 \cdots P_{j-1} P_j^2 \le D \text{ for } j \le r \right\},
\end{aligned}
\end{equation}
which have
\[
    \bD_r^\pm(D) \subset \bDw_r(D).
\]
\end{definition}

The \emph{standard} upper-bound ($+$) and lower-bound $(-)$ linear sieve weights of level $D$ are given by
\[
    \lambda_d^\pm := \mu(d) \cdot \one_{d \in \mD^\pm(D)}.
\]
There is also a \emph{well-factorable} variant $\tilde{\lambda}_d^{\pm}$ of these weights due to Iwaniec (see \cite[Chapter 12.7]{friedlander2010opera}, \cite{iwaniec1980new}, \cite[Chapter 8]{maynard2025primes2}), which produces results of essentially the same strength in sieve problems. Given small parameters $\upsilon, \eta > 0$, $\tilde{\lambda}_d^\pm = \tilde{\lambda}_d^\pm(\upsilon,\eta)$ is a sum of $O_{\upsilon,\eta}(1)$ sequences of the form
\[
    \tilde{\lambda}_{d,\vec{P}}^\pm := 
    \begin{cases} 
    (-1)^r, & d = p_1 \cdots p_r, \, p_j \in (P_j, P_j^{1+\eta}] \text{ primes}, \\
    0, & \text{otherwise},
    \end{cases}
\]
where $\vec{P} = (P_1, \ldots, P_r) \in \bD_r^\pm(D^{1/(1+\eta)})$ and $P_i$ are part of the sequence $(D^{\upsilon(1+\eta)^j})_{j \ge 0}$. Each such sequence is supported on $d \in \mD^\pm(D)$, and \emph{well-factorable} in the sense that for \emph{any} choice of $D_1 D_2 = D$ with $D_i \ge 1$, one can write 
\[
    \tilde{\lambda}_{d,\vec{P}}^\pm = 
    \sum_{d_1 d_2 = d} \alpha_{d_1} \beta_{d_2},
\] 
for some $1$-bounded sequences $(\alpha_{d_1})$, $(\beta_{d_2})$ supported on $d_i \le D_i$. This is inherited from the fact that every $d \in \mDw(D)$ can be greedily factorized as $d = d_1 d_2$, for some positive integers $d_i \le D_i$.

However, the weights $\tilde\lambda^\pm_d$ are not \emph{triply-well-factorable} in the sense of \cref{def:triply-well-fact}, since not every $d \in \mD^\pm(D)$ can be factorized as $d = d_1 d_2 d_3$ with $d_i \le D_i$, given \emph{any} choice of $D_1 D_2 D_3 = D$. Fortunately, to apply our \cref{prop:well-fact-convolution}, it will suffice to use a particular choice of $D_1, D_2, D_3$ which obey the system in \cref{eq:ranges-q1q2q3-revisited}. Specifically, following \cite{maynard2025primes2,lichtman2023primes}, it will be enough to show that every modulus $d$ of interest has a factorization $d = d_1 d_2 d_3$ into positive integers obeying the system
\begin{equation} \label{eq:system-to-verify}
\begin{aligned}
    d_1 &\le \frac{N}{x^\delta}, \\ 
    N^2d_2d_3^2 &\le x^{1-\delta}, \\
    N^2 d_2^5 d_3^2 &\le x^{2-\delta},
\end{aligned}
\end{equation}
for some small $\delta > 0$ (compare this with \cref{eq:ranges-q1q2q3-revisited}).

\subsection{The upper-bound linear sieve weights} \label{subsec:upper-bound-linear-sieve}

Here we deduce \cref{thm:primes}.$(ii)$ for the upper-bound well-factorable linear sieve weights $\tilde{\lambda}_{d}^+(\upsilon,\eta)$, where $\upsilon$ is chosen to be sufficiently small in terms of $\eps$, and $\eta$ is sufficiently small in terms of $\eps, \upsilon$ (if one allows arbitrarily small values of $\upsilon, \eta$, the implicit constant in \cref{eq:primes-in-ap} should also depend on $\upsilon, \eta$). Our key factorization result is the following.

\begin{proposition}[Factorization in the upper-bound linear sieve support] \label{prop:fact-upper-linear-sieve}
Let $0 < \delta < 10^{-5}$, $D = x^{3/5 - 50\delta}$, $x^{2\delta} \le N \le x^{1/3 + \delta}$, and $d \in \mD^+(D)$. Then there exists a factorization $d = d_1 d_2 d_3$ into positive integers obeying \cref{eq:system-to-verify}.
\end{proposition}

\begin{remark}
The level $D = x^{3/5-o(1)}$ in
\cref{prop:fact-upper-linear-sieve} is optimal, as seen by taking $N = x^{1/5}$ and $p_1 \approx p_2 \approx x^{1/5}$. There are various other limiting cases, but essentially all situations where $D \le x^{3/5-o(1)}$ can be handled by an interpolation of the ranges 
\[
    d_1 \le N x^{-o(1)}, \qquad\quad 
    d_2 \le x^{1/5 - o(1)}, \qquad\quad 
    d_3 \le \frac{x^{2/5 + o(1)}}{N},
\]
and
\[
    d_1 \le N x^{-o(1)}, \qquad\quad 
    d_2 \le x^{4/15 - o(1)}, \qquad\quad 
    d_3 \le \frac{x^{1/3 + o(1)}}{N},
\]
both of which are acceptable in \cref{eq:system-to-verify} (up to a good choice of the $o(1)$ exponents in terms of $\delta$).
\end{remark}

\begin{proof}[Proof of \cref{thm:primes}.$(ii)$ assuming \cref{prop:fact-upper-linear-sieve}]
This follows analogously as in \cite[Chapter 8]{maynard2025primes2}, using \cref{prop:well-fact-convolution} instead of \cite[Proposition 7.2]{maynard2025primes2}, and \cref{prop:fact-upper-linear-sieve} instead of  \cite[Proposition 8.1]{maynard2025primes2}.
\end{proof}

Our proof of \cref{prop:fact-upper-linear-sieve} is structured differently from Maynard's computations in \cite[Chapter 8]{maynard2025primes2}, to accommodate a new range of limiting cases. We start with a preliminary result concerning the greatest $6$ prime factors of the elements of $\mD^+(D)$.

\begin{lemma}[Placing the first 6 prime factors] \label{lem:placing-first-6}
Let $0 < \delta < 10^{-5}$, $D = x^{3/5 - 50\delta}$, $x^{2\delta} \le N \le x^{1/3 + \delta}$, and $d \in \mD^+(D)$ have $6$ prime factors, counting multiplicities. Then there exists a factorization $d = d_1 d_2 d_3$ into positive integers such that
\begin{equation} \label{eq:linear-sieve-di-bounds}
    d_1 \le D_1 := N x^{-\delta}, \qquad\quad 
    d_2 \le D_2 := x^{4/15 - 5\delta}, \qquad\quad 
    d_3 \le D_3 := \frac{x^{2/5 + 2\delta}}{N}.
\end{equation}
\end{lemma}

\begin{remark}
The values of $D_1, D_2, D_3$ in \cref{eq:linear-sieve-di-bounds} do \emph{not} multiply up to $D$, and do not (yet) obey the conditions on $Q_1, Q_2, Q_3$ from \cref{prop:well-fact-convolution}. So \cref{lem:placing-first-6} does not directly imply \cref{prop:fact-upper-linear-sieve} (not even for integers with $6$ primes factors), but it will be a crucial step in its proof.
\end{remark}

\begin{proof}[Proof of \cref{lem:placing-first-6}]
We first observe that $D_1, D_2, D_3 \ge x^\delta$ by the conditions on $N$ and $\delta$. Now let $d \in \mD^+(D)$ have prime factorization $d = p_1 \cdots p_6$, with $p_1 \ge \cdots \ge p_6$; in particular, we have
\[
    p_1^3 \le D, \qquad\quad
    p_1 p_2 p_3^3 \le D, \qquad\quad 
    p_1 p_2 p_3 p_4 p_5^3 \le D.
\]
We claim that it is impossible for the following system of inequalities to hold true simultaneously:
\[
\begin{aligned}
    p_1 p_4 p_5 &> D_2, \\
    p_2 p_3 p_5 &> D_2, \\
    p_1 p_2 p_3 p_4 p_5^4 &> D_1 D_2 D_3.
\end{aligned}
\]
Indeed, by multiplying all three inequalities, one would obtain
\[
    x^{6/5 - 14\delta} = D_1 D_2^3 D_3 < (p_1 p_2 p_3 p_4 p_5^3)^2 
    \le 
    D^2,
\]
which is a contradiction. Thus at least one of these inequalities fails, which leads us to three cases.

\underline{\textbf{Case 1:}} $p_1 p_4 p_5 \le D_2$. Then, we fix $d_2 := p_1 p_4 p_5$. We will construct $d_1 \le D_1$ and $d_3 \le D_3$ such that $d_1 d_3 = p_2 p_3 p_6$; for now, we set $d_1 = d_3 := 1$. Since 
\[
    p_2^2 \le p_1^2 \le D^{2/3} \le x^{2/5} \le D_1 D_3,
\]
we must have $p_2 \le D_1$ or $p_2 \le D_3$; we set $d_1 \gets p_2$ if $p_2 \le D_1$, and $d_3 \gets p_2$ otherwise. We also have
\[
    p_2 p_3^2 = \left(p_2 p_3^4\right)^{1/2} p_2^{1/2} 
    \le 
    p_1^{1/2} \left(p_1 p_2 p_3^3\right)^{1/2} 
    \le x^{1/10 + 3/10} \le D_1 D_3,
\]
so $p_3 \le \sqrt{D_1 D_3/(d_1 d_3)}$, which forces $p_3 \le D_1/d_1$ or $p_3 \le D_3/d_3$; we set $d_1 \gets d_1 p_3$ if $p_3 \le D_1/d_1$, and $d_3 \gets d_3 p_3$ otherwise. Finally, we note that
\[
    p_2 p_3 p_6^2 = p_2^{1/2} \left(p_2 p_3^2 p_6^4\right)^{1/2}
    \le
    p_1^{1/2} \left(p_1 p_2 p_3 p_4 p_5^3\right)^{1/2}
    \le 
    x^{1/10 + 3/10} \le D_1 D_3,
\]
so $p_6 \le \sqrt{D_1 D_3/(d_1 d_3)}$, which forces $p_6 \le D_1/d_1$ or $p_6 \le D_3/d_3$; we set $d_1 \gets d_1 p_6$ if $p_6 \le D_1/d_1$, and $d_3 \gets d_3 p_6$ otherwise. At this point, we have $d_1 d_2 d_3 = p_1 \cdots p_6$ and $d_i \le D_i$ for $i \in \{1, 2, 3\}$, as we wanted.

\underline{\textbf{Case 2:}} $p_2 p_3 p_5 \le D_2$. Then, we fix $d_2 := p_2 p_3 p_5$, and construct $d_1 \le D_1, d_3 \le D_3$ such that $d_1 d_3 = p_1 p_4 p_6$. The process is completely analogous to the previous case (with $p_1, p_4$ taking the places of $p_2, p_3$), using the bounds
\begin{equation} \label{eq:linear-sieve-case-2-bounds}
\begin{aligned}
    p_1 &\le D^{2/3} \le x^{2/5} \le D_1 D_3,
    \\
    p_1 p_4^2 &= \left(p_1 p_4^4\right)^{1/2} p_1^{1/2} 
    \le 
    p_1^{1/2} \left(p_1 p_2 p_3^3\right)^{1/2} 
    \le x^{1/10 + 3/10} \le D_1 D_3,
    \\
    p_1 p_4 p_6^2 &= p_1^{1/2} \left(p_1 p_4^2 p_6^4\right)^{1/2}
    \le
    p_1^{1/2} \left(p_1 p_2 p_3 p_4 p_5^3\right)^{1/2}
    \le 
    x^{1/10 + 3/10} \le D_1 D_3.
\end{aligned}
\end{equation}

\underline{\textbf{Case 3:}} $p_1 p_2 p_3 p_4 p_5^4 \le D_1 D_2 D_3$. We can also assume without loss of generality that we are not in the previous case, so $p_2 p_3 p_5 > D_2$. Then, we fix $d_2 := p_2 p_3$, noting that
\[
    p_2 p_3 \le p_1^{1/3} (p_1 p_2 p_3^3)^{1/3} \le D^{4/9} \le x^{4/15 - 5\delta} = D_2.
\]
We will construct $d_1 \le D_1$ and $d_3 \le D_3$ such that $d_1 d_3 = p_1 p_4 p_5 p_6$. 
We start by placing the primes $p_1, p_4, p_6$ into $d_1$ and $d_3$ exactly as in the previous case, using the bounds in \cref{eq:linear-sieve-case-2-bounds}.

At this point, we have $d_1 d_3 = p_1 p_4 p_6$, and it remains to place $p_5$. Since
\[
    p_1 p_4 p_6 p_5^2 \le p_1 p_4 p_5^3 = \frac{p_1 p_2 p_3 p_4 p_5^4}{p_2 p_3 p_5} \le \frac{D_1 D_2 D_3}{D_2} = D_1 D_3,
\]
we have $p_5 \le \sqrt{D_1 D_3 / (d_1 d_3)}$, which forces $p_5 \le D_1/d_1$ or $p_5 \le D_3/d_3$. Then we are done by setting $d_i \gets d_i p_5$ for some $i \in \{1, 3\}$.
\end{proof}

We can now prove \cref{prop:fact-upper-linear-sieve}, thus completing the proof of \cref{thm:primes}.

\begin{proof}[Proof of \cref{prop:fact-upper-linear-sieve}]
Let $d \in \mD^+(D)$ have prime factorization $d = p_1 \cdots p_r$ with $p_1 \ge p_2 \ge \ldots$. We recall from \cref{eq:linear-sieve-support} that this implies
\[
    p_1 \cdots p_{2j} p_{2j+1}^3 \le D \qquad\quad \text{and} \qquad\quad 
    p_1 \cdots p_{k-1} p_k^2 \le D,
\]
for all $0 \le j < r/2$ and $1 \le k \le r$. We aim to place the primes $p_1, \ldots, p_r$ into $d_1, d_2, d_3$ such that the bounds in \cref{eq:system-to-verify} hold.

We begin by setting $d_1 = d_2 = d_3 := 1$, and perform the following iterative greedy process. At step $1 \le j \le r$, we do the following:
\begin{itemize}
    \item[$j.(i)$.] If $j \le 6$, we place $p_j$ in the corresponding factor $d_i$ from the factorization in \cref{lem:placing-first-6} (i.e., we set $d_i \gets d_i p_j$). If $j \ge 7$, we act greedily and place $p_j$ into any factor $d_i$ such that after substituting $d_i \gets d_i p_j$, we have the same system of inequalities
    \[
        d_1 \le N x^{-\delta}, \qquad\quad 
        d_2 \le x^{4/15 - 5\delta}, \qquad\quad 
        d_3 \le \frac{x^{2/5 + 2\delta}}{N},
    \]
    as in \cref{eq:linear-sieve-di-bounds}.
    We terminate unsuccessfully if this is impossible.
    \item[$j.(ii)$.] Having placed $p_j$ into some $d_i$, we check whether either of the lower bounds
    \[
        d_2 > x^{1/5 - 5\delta}, \qquad\quad \text{or} \qquad\quad 
        d_3 > \frac{x^{1/3 + 2\delta}}{N},
    \]
    holds; if so, we terminate unsuccessfully. Otherwise, we continue with step $j+1$ (or terminate successfully if $j = r$).
\end{itemize}

\underline{\textbf{Case 1:}} The process terminates successfully; this is the easier case.

We are left with a factorization $d = d_1 d_2 d_3$ satisfying
\[
    d_1 \le N x^{-\delta}, \qquad\quad d_2 \le x^{1/5-5\delta},
    \qquad\quad 
    d_3 \le \frac{x^{1/3+2\delta}}{N},
\]
which actually forces $d \le x^{8/15 - 4\delta}$ (significantly smaller than $D = x^{9/15 - 50\delta}$). Then we can verify the conditions in \cref{eq:system-to-verify}, with room to spare:
\[
    N^2 d_2 d_3^2 \le x^{2(1/3 + 2\delta) + (1/5 - 5\delta)}
    <
    x^{1-\delta},
\]
\[
    N^2 d_2^5 d_3^2 \le x^{2(1/3+2\delta) + (1-25\delta)} \le x^{2-\delta}.
\]

\underline{\textbf{Case 2:}} The process terminates unsuccessfully in substep $j.(i)$; we show that this cannot happen.

Indeed, we must have $j \ge 7$ since \cref{lem:placing-first-6} handles all $j \le 6$. We are left with a factorization $p_1 \cdots p_{j-1} = d_1 d_2 d_3$, which must satisfy
\[
    d_2 p_j > x^{4/15-5\delta},
\]
in order to terminate in substep $j.(i)$. Moreover, since we did not terminate in substep $(j-1).(ii)$, we must have
\[
    d_2 \le x^{1/5 - 5\delta}.
\]
But since $j \ge 7$, we have
\[
    \frac{x^{4/15-5\delta}}{x^{1/5-5\delta}} < \frac{d_2 p_j}{d_2} \le p_7 \le \left(p_1 \cdots p_6 p_7^3\right)^{1/9} 
    \le 
    D^{1/9} 
    <
    x^{1/15},
\]
which gives a contradiction.

\underline{\textbf{Case 3:}} The process terminates unsuccessfully in substep $j.(ii)$; this is the main case.

We are left with a factorization $p_1 \cdots p_j = d_1 d_2 d_3$ satisfying
\begin{equation} \label{eq:fact-satisfying}
    d_1 \le N x^{-\delta}, \qquad\quad 
    d_2 \le x^{4/15 - 5\delta}, \qquad\quad 
    d_3 \le \frac{x^{2/5 + 2\delta}}{N},
\end{equation}
and \emph{either} $d_2 > x^{1/5-5\delta}$ \emph{or} $d_3 > x^{1/3+2\delta}/N$ (we cannot have both, since we should have terminated in a previous substep $(ii)$ in that case; note that both of these bounds fail at the very beginning of the greedy process because $N \le x^{1/3+\delta}$, and that only one $d_i$ gets updated in each substep $(i)$).

\emph{Case 3.1:} One has $d_2 > x^{1/5 - 5\delta}$ and $d_3 \le x^{1/3+2\delta}/N$. In this case, we set $D_1 := Nx^{-\delta}$, $D_2 := d_2$, and $D_3 := x^{3/5 - 3\delta}/(Nd_2)$, which have $D_i \ge d_i$ (in light of \cref{eq:fact-satisfying}) and $D_1 D_2 D_3 \ge D$. We then run a greedy process to place the remaining primes $p_k$ with $k \ge j+1$ into either $d_1$ or $d_3$, while preserving the inequalities $d_i \le D_i$. This works because at step $k$, before placing $p_k$, we have
\[
    p_1 \cdots p_{k-1} p_k^2 \le D
    \qquad\quad 
    \Rightarrow 
    \qquad\quad
    p_k^2 \le \frac{D}{d_1 d_2 d_3} \le \frac{D_1 D_3}{d_1 d_3},
\]
so $p_k \le \max(D_1/d_1, D_3/d_3)$ (i.e., there is ``enough room'' for $p_k$ in $d_1$ or $d_3$). In the end, we have $d = d_1 d_2 d_3$ with $d_i \le D_i$, and we can verify the bounds in \cref{eq:system-to-verify} using $x^{1/5 - 5\delta} < d_2 \le x^{4/15-5\delta}$:
\[
    N^2 d_2 d_3^2 \le N^2 d_2 \left(\frac{x^{3/5-3\delta}}{Nd_2}\right)^2 = 
    \frac{x^{6/5 - 6\delta}}{d_2}
    <
    x^{1-\delta},
\]
\[
    N^2 d_2^5 d_3^2 \le N^2 d_2^5 \left(\frac{x^{3/5-3\delta}}{Nd_2}\right)^2 = d_2^3 x^{6/5-6\delta} \le x^{2-\delta}.
\]
\emph{Case 3.2:} One has $d_2 \le x^{1/5-5\delta}$ and $d_3 > x^{1/3+2\delta}/N$. Then we set $D_1 := N x^{-\delta}$, $D_2 := x^{3/5-3\delta}/(Nd_3)$, and $D_3 := d_3$, which have $D_i \ge d_i$ (in light of \cref{eq:fact-satisfying}) and $D_1 D_2 D_3 \ge D$. We run a similar greedy process to place the remaining primes $p_k$ with $k \ge j+1$ into either $d_1$ or $d_2$, while preserving the bounds $d_i \le D_i$. This works because at step $k$, before placing $p_k$, we have
\[
    p_1 \cdots p_{k-1} p_k^2 \le D
    \qquad\quad 
    \Rightarrow 
    \qquad\quad
    p_k^2 \le \frac{D}{d_1 d_2 d_3} \le \frac{D_1 D_2}{d_1 d_2},
\]
so $p_k \le \max(\frac{D_1}{d_1}, \frac{D_2}{d_2})$. In the end, we have $d = d_1 d_2 d_3$ with $d_i \le D_i$, and we can verify the bounds in \cref{eq:system-to-verify} using $x^{1/3 + 2\delta}/N < d_3 \le x^{2/5 + 2\delta}/N$:
\[
    N^2 d_2 d_3^2 \le N^2 \frac{x^{3/5-3\delta}}{Nd_3} d_3^2 = 
    N d_3\, x^{3/5 - 3\delta}
    \le
    x^{1-\delta},
\]
\[
    N^2 d_2^5 d_3^2 \le N^2 \left(\frac{x^{3/5-3\delta}}{Nd_3}\right)^5 d_3^2 = \frac{x^{3-15\delta}}{(Nd_3)^3} \le x^{2-\delta}.
\]
We have now covered all cases.
\end{proof}

\subsection{The linear sieve weights with special factors} \label{subsec:linear-sieve-special-factors}

From the work of Bombieri--Friedlander--Iwaniec \cite[Theorem 10]{bombieri1986primes} and our work in the previous subsection, we have exponents of distribution of $\tfrac{4}{7} - \eps$ and $\tfrac{3}{5} - \eps$ for the weights $\tilde\lambda_d^-$ and $\tilde\lambda_d^+$, respectively. Here we obtain a better level of distribution for $\tilde\lambda_d^\pm$, up to $\tfrac{5}{8} - \eps$, when more information about the factorization of $d$ is available. 

We adapt the computations from \cite[Section 6]{lichtman2023primes} using our \cref{prop:well-fact-convolution} instead of \cite[Proposition 5.2]{lichtman2023primes}.
More precisely, our
\cref{thm:primes-ls-weights,prop:fact-well-fact,lem:general-factorization,lem:greedy-algorithm,prop:fact-anatomy} correspond respectively to \cite[Proposition 6.6, Proposition 6.1, Lemma 6.3, Lemma 6.4, Proposition 6.5]{lichtman2023primes}; additionally, we use \cref{lem:extreme-ranges} to fix a small error in the argument from \cite[Section 6]{lichtman2023primes}. For $t \ge 0$, we let
\begin{equation} \label{eq:vartheta-t}
    \vartheta(t) := \min\left(\frac{1+t}{2}, \frac{2-3t}{2}\right) \in \left[\frac{1}{2}, \frac{5}{8}\right],
\end{equation}
which achieves its maximum (only) at $t = \tfrac{1}{4}$. For $\tfrac{1}{4} \ge t_1 \ge t_2 \ge t_3 \ge 0$ and $\delta > 0$, we also define
\begin{equation} \label{eq:vartheta-123}
\begin{aligned}
    \vartheta(t_1, t_2, t_3) :=  \max\big\{\vartheta(t_1), \vartheta(t_2), \vartheta(t_1+t_2), \vartheta(t_1+t_2+t_3), 
    w(t_1, t_2, t_3), w(t_2, t_1, t_3), \\
    \psi(\vartheta(t_1+t_3), t_1+2t_2+t_3), \psi(\vartheta(t_2+t_3), 2t_1+t_2+t_3) \big\},
\end{aligned}
\end{equation}
where $\psi(x, y) := x \one_{x \ge y+2\delta}$ and
\[
    w(t_1, t_2, t_3) = \psi\left(\min\left\{\frac{5-3t_3}{8}, 1-2t_2-2\delta\right\}, \frac{1+t_1}{2} \right).
\]

Our main result in this subsection, which will imply \cref{cor:twin-primes}, is the following.

\begin{theorem}[Primes in APs with linear sieve weights]\label{thm:primes-ls-weights}
Let $\eps = 10\delta > 0$ be sufficiently small and $A > 0$. Let $(P_1, \ldots, P_r) \in \bDw_r(D)$ with $1 \le D = x^{\vartheta_0}$ and $P_i = x^{t_i}$, where $t_i$ are part of the sequence $(\eps^2(1+\eps^9)^j)_{j \ge 1}$. Then provided that $\vartheta_0 \le \vartheta(t_1) - \eps$, for any choice of the sign $\pm$, we have
\[
    \sum_{\substack{b = p_1\cdots p_r \\ D_i < p_i \le D_i^{1+\eps^9}}} \sum_{\substack{d = bc \le D \\ c \mid P(p_r) \\ (d, a) = 1}}
    \tilde\lambda^\pm(d) \left( \pi(x; d, a) - \frac{\pi(x)}{\varphi(d)} \right) \ll_{a,A,\eps} \frac{x}{(\log x)^A}.
\]
Moreover, if $t_1 \le \tfrac{1}{4}$ and $r \ge 3$, then the same holds provided that $\vartheta_0 \le \vartheta(t_1, t_2, t_3) - \eps$.
\end{theorem}

Much like \cref{thm:primes}.$(ii)$ stems from the factorization result in \cref{prop:fact-upper-linear-sieve},
\cref{thm:primes-ls-weights} depends on the factorization result below.

\begin{proposition}[Factorization in the well-factorable support] \label{prop:fact-well-fact}
Let $0 < \delta < 10^{-5}$, $1 \le D = x^{\vartheta_0}$, $x^{2\delta} \le N \le x^{1/3+\delta}$, and $d \in \mDw(D)$. Write $d = p_1\cdots p_r$ where $p_1 \ge \cdots \ge p_r$ are primes with $p_i = x^{t_i}$. Then there exists a factorization $d = d_1d_2d_3$ into positive integers obeying \cref{eq:system-to-verify}, provided that $\vartheta_0 \le \vartheta(t_1) - 2\delta$,
as in \cref{eq:vartheta-t}.

Moreover, if $t_1 \le \tfrac{1}{4}$ and $r \ge 3$, then it suffices that
$\vartheta_0 \le \vartheta(t_1, t_2, t_3) - 2\delta$,
as in \cref{eq:vartheta-123}.
\end{proposition}

\begin{proof}[Proof of \cref{thm:primes-ls-weights} assuming \cref{prop:fact-well-fact}]
This is almost identical to the proof of \cite[Proposition 5.4]{lichtman2025modification}, using \cref{prop:well-fact-convolution} instead of \cite[Theorem 2.5]{lichtman2025modification}, and \cref{prop:fact-well-fact} instead of \cite[Proposition 3.3]{lichtman2025modification}.
\end{proof}

To prove \cref{prop:fact-well-fact}, we will need a few lemmas.

\begin{lemma}[Two-factor greedy algorithm] \label{lem:extreme-ranges}
Let $0 < \delta < 10^{-5}$, $\vartheta \in [\tfrac{1}{2}, \tfrac{5}{8}]$, $x^{2\delta} \le N \le x^{1/3+\delta}$, $1 \le D \le x^{\vartheta-2\delta}$, and suppose that
\[
    x^{\frac{5\vartheta-2}{3}-\delta} \le N.
\]
Then, any $d \in \mDw(D)$ has a factorization $d = d_1d_2d_3$ into positive integers satisfying \cref{eq:system-to-verify}.
\end{lemma}

\begin{proof}
Let $(D_1, D_2) := (N x^{-\delta}, x^{\vartheta-\delta}/N)$, so that $D \le D_1 D_2$ and $D_i \ge 1$. Let $d_1 = d_2 = 1$ and $d = p_1\cdots p_r$, where $p_1 \ge \cdots \ge p_r$ are primes. We will run a greedy algorithm to append each of these primes to one of $d_1$ or $d_2$, while preserving the bounds $d_1 \le D_1$, $d_2 \le D_2$. At step $j \ge 1$, the definition of $\mDw(D)$ implies
\[
    p_j^2 = \frac{p_1\cdots p_{j-1} p_j^2}{p_1\cdots p_{j-1}} \le \frac{D}{p_1 \cdots p_{j-1}} \le \frac{D_1 D_2}{d_1 d_2},
\]
so $p_j \le \max(D_1/d_1, D_2/d_2)$, and we can append $p_j$ to one of $d_1$ and $d_2$. In the end, we take $d_3 = 1$, and obtain $d = d_1d_2d_3$ with $d_1 \le D_1$, $d_2 \le D_2$. To verify the system \cref{eq:system-to-verify}, we write
\[
    d_1 \le D_1 = \frac{N}{x^\delta},
\]
\[
    N^2 d_2 d_3^2 \le N^2 D_2 = N x^{\vartheta-\delta} \le x^{1/3+\delta+2/3-2\delta} = x^{1-\delta},
\]
and, using the hypothesis in the form $D_2 = x^{\vartheta-\delta}/N \le x^{(2-2\vartheta)/3}$,
\[
    N^2 d_2^5 d_3^2 \le (ND_2)^2 D_2^3 \le x^{2(\vartheta-\delta)} x^{2-2\vartheta} \le x^{2-\delta},
\]
as required.
\end{proof}

\begin{lemma}[General factorization criterion] \label{lem:general-factorization}
Let $0 < \delta < 10^{-5}$, $\vartheta \in [\tfrac{1}{2}, \tfrac{5}{8}]$, $x^{2\delta} \le N \le x^{1/3+\delta}$, $1 \le D \le x^{\vartheta-2\delta}$, and  
\begin{equation} \label{eq:u-v-choice}
    v := 2\vartheta - 1 \quad<\quad \frac{1}{4} \quad<\quad u := \frac{2-2\vartheta}{3}.
\end{equation}
Suppose $d \in \mDw(D)$ has a factorization $d = d_1d_2d_3$ into positive integers satisfying
\[
    d_1 \le \frac{N}{x^\delta}, \qquad 
    d_2 \in [x^v, x^u], \qquad 
    d_3 \le \max\left(1, \frac{x^{\vartheta-\delta}}{d_2N}\right).
\]
Then, $d$ has a (potentially different) factorization obeying the system \cref{eq:system-to-verify}.
\end{lemma}

\begin{proof}
First, if $\max(1, \frac{x^{\vartheta-\delta}}{d_2N}) = 1$, then $x^{\vartheta-\delta} \le d_2 N \le x^u N$, and we can apply \cref{lem:extreme-ranges}. Otherwise, we may assume $d_2d_3 \le x^{\vartheta-\delta}/N$, and we will use the given factorization $d = d_1d_2d_3$.

The first bound in \cref{eq:system-to-verify} is reiterated in the hypothesis here. For the second and third bounds, note that
\[
    N^2 d_2 d_3^2 = \frac{(Nd_2 d_3)^2}{d_2} \le \frac{x^{2\vartheta-\delta}}{x^v} = x^{2\vartheta - \delta - 2\vartheta+1} = x^{1-\delta},
\]
\[
    N^2 d_2^5 d_3^2 = (Nd_2d_3)^2 d_2^3 \le x^{2\vartheta-\delta} x^{3u} = x^{2\vartheta - \delta + 2-2\vartheta} = x^{2-\delta}.
\]
\end{proof}

\begin{remark}
When $\vartheta = \tfrac{3}{5}$, we have $v = \tfrac{1}{5}$ and $u = \tfrac{4}{15}$, which gave a relevant interval for the construction of $d_2$ in \cref{subsec:upper-bound-linear-sieve}.
\end{remark}

\begin{lemma}[Three-factor greedy algorithm] \label{lem:greedy-algorithm}
Let $\delta, \vartheta, u, v, N$ be as in \cref{lem:general-factorization}, $r \ge 3$, $1 \le D \le x^{\vartheta-2\delta}$, and $d \in \mDw(D)$. Write $d = p_1\cdots p_r$ where $p_1 \ge \cdots \ge p_r$ are primes. Suppose that $p_3 \le x^{u-v}$. Also, assume that ($p_1 \le x^v$, $p_2^2 \le x^{1-\vartheta-2\delta}$) or ($p_2 \le x^v$, $p_1^2 \le x^{1-\vartheta-2\delta}$). Then there exists a factorization $d = d_1d_2d_3$ into positive integers satisfying \cref{eq:system-to-verify}.
\end{lemma}

\begin{proof}
Let $(D_1, D_2, D_3) := (Nx^{-\delta}, x^v, x^{1-\vartheta-\delta}/N)$, so that $D \le D_1 D_2 D_3$ and $D_i \ge 1$. Note that any tuple $(d_1, d_2, d_3)$ with $d_i \le D_i$ satisfies \cref{eq:system-to-verify}, since
\[
    N^2 D_2 D_3^2 = x^{v + 2-2\vartheta-2\delta} = x^{1-2\delta},
\]
and using $\vartheta \le \tfrac{5}{8}$,
\[
    N^2 D_2^5 D_3^2 = x^{5v + 2-2\vartheta-2\delta} = x^{8\vartheta-3-2\delta} \le x^{2-2\delta}.
\]
By assumption, we have ($p_1 \le D_2$ and $p_2^2 \le D_1D_3$) or ($p_2 \le D_2$ and $p_1^2 \le D_1D_3$), so for some choice $\{d_1, d_2, d_3\} = \{1, p_1, p_2\}$ we must have $d_i \le D_i$ for all $i$. We keep this choice, and run a greedy algorithm to append the primes $p_j$, for $j \ge 3$, to one of $d_1, d_2, d_3$ (i.e., $d_i \gets d_ip_j$), while preserving the bounds $d_i \le D_i$. If this algorithm terminates after appending all primes $p_3, \ldots, p_r$, then we obtain a factorization $d = d_1d_2d_3$ which satisfies $d_i \le D_i$, and thus also \cref{eq:system-to-verify}.

Otherwise, there must be some index $3 \le j \le r$ such that the prime $p_j$ cannot be appended to any $d_i$, where $d_1d_2d_3 = p_1\cdots p_{j-1}$; thus we have $d_ip_j > D_i$ for all $i$. By our assumption, we thus have
\[
    x^v = D_2 < d_2p_j \le D_2p_3 \le x^v x^{u-v} = x^u,
\]
so $d_2' := d_2p_j \in [x^v, x^u]$, and
\[
    D_1 < d_1 p_j = \frac{d_1d_2d_3p_j^2}{d_2d_3p_j} \le \frac{D}{d_2d_3p_j} \le \frac{D_1 D_2 D_3}{d_2'd_3},
\]
so $D_3' := D_2 D_3/d_2' \ge d_3$. By the definition of $\mDw(D)$ we have that for each $k > j$,
\[
    p_k^2 \le \frac{D}{p_1\cdots p_{k-1}} \le \frac{D_1D_2D_3}{d_1d_2d_3 p_j \cdots p_{k-1}} 
    =
    \frac{D_1D_3'}{d_1d_3p_{j+1} \cdots p_{k-1}}.
\]
Using this bound, we can greedily construct a factorization $d_1' d_3' = d_1d_3 p_{j+1} \cdots p_r$ (starting from $d_1' = d_1$, $d_3' = d_3$ and appending each $p_k$ at a time) such that $d_1' \le D_1$ and $d_3' \le D_3'$. Therefore, we have $d_1'd_2'd_3' = d$ and
\[
    d_1' \le D_1 = \frac{N}{x^{\delta}}, 
    \qquad 
    d_2' = d_2p_j \in [x^v, x^u],
    \qquad 
    d_2'd_3' \le d_2'D_3' = D_2D_3 = \frac{x^{\vartheta-\delta}}{N}.
\]
By \cref{lem:general-factorization}, we conclude that $d_1', d_2', d_3'$ satisfy \cref{eq:system-to-verify}.
\end{proof}

\begin{proposition}[Factorization depending on the anatomy] \label{prop:fact-anatomy}
Let $\delta, \vartheta, u, v, N$ be as in \cref{lem:general-factorization}, $1 \le D \le x^{\vartheta-2\delta}$, and $d \in \mDw(D)$. Write $d = p_1\cdots p_r$ where $p_1 \ge \cdots \ge p_r$ are primes. Assume that $p_1 \le x^u$, and that one of the following holds (statements involving $p_j$ implicitly assume $r \ge j$):
\begin{itemize}
    \item[$(i).$] $d_2 \in [x^v, x^u]$ for some $d_2 \in \{p_1, p_2, p_1p_2, p_1p_2p_3\}$;
    \item[$(ii).$] $d_2 := p_1p_3 \in [x^v, x^u]$ and $p_2^2 \le x^{\vartheta-2\delta}/d_2$;
    \item[$(iii).$] $d_2 := p_2p_3 \in [x^v, x^u]$ and $p_1^2 \le x^{\vartheta-2\delta}/d_2$;
    \item[$(iv).$] $p_3 \le x^{u-v}$, and ($p_1 \le x^v$, $p_2^2 \le x^{1-\vartheta-2\delta}$) or ($p_2 \le x^v$, $p_1^2 \le x^{1-\vartheta-2\delta}$).
\end{itemize}
Then there exists a factorization $d = d_1d_2d_3$ into positive integers satisfying \cref{eq:system-to-verify}.
\end{proposition}
\begin{proof}
Assuming $(iv)$, the conclusion follows immediately from \cref{lem:greedy-algorithm}. So let us assume that one of $(i)$, $(ii)$, $(iii)$ holds. Let $D_1 := N x^{-\delta}$, $d_2$ be the corresponding value from $(i)$, $(ii)$, or $(iii)$ (say, the first of these that holds), and 
\[
    D_3 := \max\left(1, \frac{x^{\vartheta-\delta}}{Nd_2}\right).
\]
Note that $D \le D_1d_2D_3$ and $D_1, D_3 \ge 1$. If we can find a factorization $d = d_1d_2d_3$ with $d_1 \le D_1$ and $d_3 \le D_3$, then \cref{lem:general-factorization} will complete the proof.

Suppose for a start that $d_2 = p_1 \cdots p_i \in [x^v, x^u]$ for some $i \in \{1, 2, 3\}$ as in $(i)$. For each $j \in \{i+1, \cdots, r\}$, by the definition of $\mDw$ we have 
\[
    p_j^2 \le \frac{D}{p_1 \cdots p_{j-1}} \le \frac{D_1 D_3}{p_{i+1}\cdots p_{j-1}}.
\]
Using this bound, we can greedily construct a factorization $d_1 d_3 = p_{i+1} \cdots p_r$ (starting from $d_1 = 1$, $d_3 = 1$ and appending each $p_j$ at a time) such that $d_1 \le D_1$ and $d_3 \le D_3$, so we are done. 

Otherwise, we have $p_1, p_1p_2, p_1p_2p_3 \not\in [x^v, x^u]$; in particular, $p_1 \le x^u$ and $p_1 \not \in [x^v, x^u]$ imply $p_1 < x^v$, so $p_2 < x^v$ as well; thus $(i)$ cannot hold.
\begin{itemize}
    \item If $(ii)$ holds, so $d_2 := p_1p_3 \in [x^v, x^u]$ and $p_2^2 \le x^{\vartheta-2\delta}/d_2 \le D_1 D_3$, then we have $p_2 \le D_1$ or $p_2 \le D_3$. 
    \item If $(iii)$ holds, so $d_2 := p_2p_3 \in [x^v, x^u]$ and $p_1^2 \le x^{\vartheta-2\delta}/d_2 \le D_1 D_3$, then we have $p_1 \le D_1$ or $p_1 \le D_3$.
\end{itemize}
In either case, we can factor $p_1p_2p_3 = d_1d_2d_3$ where $d_1 \le D_1$ and $d_3 \le D_3$. Then for each $j \in \{4, \ldots, r\}$ (if any), we have
\[
    p_j^2 \le \frac{D}{p_1 \cdots p_{j-1}} \le \frac{D_1d_2D_3}{p_1 \cdots p_{j-1}} \le \frac{(D_1/d_1) (D_3/d_3)}{p_4 \cdots p_{j-1}},
\]
so we can greedily append $p_j$ to one of $d_1$ and $d_3$ until we have $d = d_1d_2d_3$ with $d_1 \le D_1$, $d_3 \le D_3$.
\end{proof}

\begin{proof}[Proof of \cref{prop:fact-well-fact}]
Let $d = p_1\cdots p_r$ where $p_1 \ge \cdots \ge p_r$ are primes with $p_i = x^{t_i}$. We want to show that $d$ has a factorization obeying \cref{eq:system-to-verify}, under one of the following assumptions:
\begin{itemize}
    \item[$(a)$.] $d \in \mDw(x^{\vartheta(t_1)-2\delta})$ where $\vartheta(t_1)$ is as in \cref{eq:vartheta-t}, or
    \item[$(b)$.] $t_1 \le \tfrac{1}{4}$, $r \ge 3$, and $d \in \mDw(x^{\vartheta(t_1,t_2,t_3)-2\delta})$, where $\vartheta(t_1,t_2,t_3)$ is as in \cref{eq:vartheta-123}.
\end{itemize}
Applying \cref{prop:fact-anatomy} for some $\vartheta \in [\tfrac{1}{2}, \tfrac{5}{8}]$ and letting $u = u(\vartheta)$, $v = v(\vartheta)$ be as in \cref{eq:u-v-choice}, we deduce that $d$ has a factorization obeying \cref{eq:system-to-verify} provided that $d \in \mDw(x^{\vartheta-2\delta})$, $t_1 \le u$, and that one of the following holds:
\begin{itemize}
    \item[$(i)$.] $t \in [v, u]$ for some $t \in \{t_1, t_2, t_1 + t_2, t_1 + t_2 + t_3\}$;
    \item[$(ii)$.] $t_1 + t_3 \in [v, u]$ and $t_1 + 2t_2 + t_3 \le \vartheta - 2\delta$;
    \item[$(iii)$.] $t_2 + t_3 \in [v, u]$ and $2t_1 + t_2 + t_3 \le \vartheta - 2\delta$;
    \item[$(iv)$.] $t_3 \le u - v$ and $t_1 \le v$, $2t_2 \le 1 - \vartheta - 2\delta$;
    \item[$(v)$.] $t_3 \le u - v$ and $t_2 \le v$, $2t_1 \le 1 - \vartheta - 2\delta$.
\end{itemize}
Note that from \cref{eq:u-v-choice}, \cref{eq:vartheta-t} and a short computation, we have the equivalence
\[
    t \in [v, u] = \left[2\vartheta - 1, \frac{2-2\vartheta}{3} \right] \qquad \iff \qquad \vartheta \le \vartheta(t) = \min\left(\frac{1+t}{2}, \frac{2-3t}{2}\right).
\]
Now suppose assumption $(a)$ holds. Then we can use $\vartheta = \vartheta(t_1)$, which implies $t_1 \in [v, u]$. So $d \in \mDw(x^{\vartheta-2\delta})$, $t_1 \le u$, and $(i)$ holds for $t = t_1$, and thus $d$ has a factorization as required.

Next, suppose assumption $(b)$ holds. Then we can use $\vartheta = \vartheta(t_1, t_2, t_3)$, which also lies in $[\tfrac{1}{2}, \tfrac{5}{8}]$. Moreover, we have $t_1 \le \tfrac{1}{4} \le \tfrac{2-2\vartheta}{3} = u$. So $d \in \mDw(x^{\vartheta-2\delta})$, $t_1 \le u$, and it suffices to verify one of conditions $(i)$-$(v)$ above; we split into cases based on the maximum from \cref{eq:vartheta-123}:
\begin{itemize}[leftmargin=*]
    \item If $\vartheta \in \{\vartheta(t_1), \vartheta(t_2), \vartheta(t_1+t_2), \vartheta(t_1+t_2+t_3)\}$, then $(i)$ holds;
    \item If $\vartheta = \psi(\vartheta(t_1+t_3), t_1+2t_2+t_3)$, so $\vartheta = \vartheta(t_1+t_3)$ and $t_1 + 2t_2 + t_3 + 2\delta \le \vartheta(t_1+t_3)$, then $(ii)$ holds;
    \item If $\vartheta = \psi(\vartheta(t_2 + t_3), 2t_1+t_2+t_3)$, so $\vartheta = \vartheta(t_2+t_3)$ and $2t_1 + t_2 + t_3 + 2\delta \le \vartheta(t_2+t_3)$, then $(iii)$ holds;
    \item If $\vartheta = w(t_1,t_2,t_3)$, so $\vartheta = \min\left\{ \tfrac{5-3t_3}{8}, 1-2t_2-2\delta \right\}$ and $\tfrac{1+t_1}{2} + 2\delta \le \vartheta$, then $(iv)$ holds (noting that $u - v = \frac{5-8\vartheta}{3}$);
    \item If $\vartheta = w(t_2,t_1,t_3)$, so $\vartheta = \min\left\{ \tfrac{5-3t_3}{8}, 1-2t_1-2\delta \right \}$ and $\tfrac{1+t_2}{2} + 2\delta \le \vartheta$, then $(v)$ holds.
\end{itemize}
This completes our proof.
\end{proof}


\begin{proof}[Proof of \cref{cor:twin-primes}] 
We very closely follow the sieve computations in \cite[Sections 7.1 and 7.2]{lichtman2023primes}, using our \cref{thm:primes-ls-weights} instead of \cite[Proposition 6.6]{lichtman2023primes}. By comparing the exponents $\vartheta(t)$, $\vartheta(t_1, t_2, t_3)$ from \cref{eq:vartheta-t,eq:vartheta-123} with \cite[(6.2) and (6.4), with $\alpha = 0$]{lichtman2023primes}, this simply amounts to taking $\theta = 0$ rather than $\theta = 7/32$, and correcting the typo $w(t_1, t_3, t_2) \to w(t_2, t_1, t_3)$ in \cite[(6.4)]{lichtman2023primes}. Adapting the Mathematica file `PrimeAPTwinTheta.nb' from \cite{lichtman2023primes} with these quick changes, we obtain adjusted values for the sieve integrals on \cite[p.\,30]{lichtman2023primes} as below (to be compared with the table on \cite[p.\,32]{lichtman2023primes}).
\begin{center}
\setlength\extrarowheight{-2pt}
\begin{tabular}{cc|cc}
$n$ & $G_n$ & $n$ & $G_n$
\\ \hline 
1 & 38.8989 & 5 & 1.84027 \\
2 & -5.88606 & 6 & 0.628688 \\ 
3 & -4.13106 & 7 & 0.420003 \\ 
4 & -5.20164 & 8 & 0.913626 
\end{tabular}
\end{center}
This results in an improvement of \cite[(7.13)]{lichtman2023primes} to
\[
    \{p \le x : p, p+2 \text{ are prime}\} \le 3.20254\ \Pi_2(x),
\]
as we claimed.
Note that we have omitted various parameter optimizations for simplicity.
\end{proof} 

\section{Smooth numbers with arbitrary weights} \label{sec:smooth} 

Here we prove \cref{thm:smooth}, building on the arguments of Drappeau \cite{drappeau2015theoremes}. As in \cref{sec:primes}, we will work in reverse compared to the outline in \cref{sec:informal-overview}, gradually building up to a triple convolution estimate in \cref{prop:triple-convo}.

We start with a bound for multilinear forms of incomplete Kloosterman sums as in \cref{eq:incomplete-kloosterman-overview}, which follows from \cref{prop:large-sieve-additive,prop:kloosterman-incomplete}, and plays a similar role to \cref{lem:consequence-Kloosterman}.

\begin{lemma} \label{lem:consequence-incomplete} Let $\eps > 0$, $1 \ll N, T, H, K, L \ll x$ with $TH \ll N$, $a, d \in \Z \setminus \{0\}$ with $1 \le |a| \le x^\eps$, $1 \le d \le x^{2\eps}$, $\Phi_i(t)$ be smooth functions supported in $t \asymp 1$ with $\Phi_j^{(j)} \ll_j 1$, and 
\[
    \phi(h_1, h_2) := \Phi_1\left(\frac{h_1}{H}\right) \Phi_2\left(\frac{h_2}{H}\right) e(h_1\alpha_1 + h_2\alpha_2),
\]
where $\alpha_i \in \R/\Z$ have $\min_i T_H(\alpha_i) \ll x^{2\eps}$ (recall \cref{eq:tn}). Then for any smooth function $\Phi(x_1, x_2, z)$ supported in $x_i, z \asymp 1$, satisfying $\partial_{x_1}^{j_1} \partial_{x_2}^{j_2} \partial_z^{\ell} \Phi(x_1, x_2, z) \ll_{j_1,j_2,\ell,\eps} 1$, one has
\begin{equation} \label{eq:consequence-incomplete}
\begin{aligned}
    \sum_{n, n' \sim N}
    \left \vert
    \sum_{\substack{1 \le |t| \le T \\ (t, nn') = 1 \\ t \mid n - n'}} 
    \sum_{\substack{h, h' \\ e = at(n'h - nh') \neq 0}} \phi(h, h')
    \sum_{\substack{k, \ell \\ (k, dnn'\ell) = 1}}
    \Phi\left(\frac{\ell}{L}, \frac{k}{K}\right)
    e\left(\frac{e\bar{dnn'\ell}}{k}\right) \right \vert
    \\
    \ll_{\eps} x^{6\eps} THN 
    \left(L^2 TH N^3 + \left(1 + \frac{K^2}{N^3 TH^2}\right)^{\theta_{\max}} K \left(K + LN^2\right)N^2 \right)^{1/2}.
\end{aligned}
\end{equation}
\end{lemma}

\begin{proof}
Let $\mK$ denote the sum in question; we begin by splitting
\begin{equation} \label{eq:K-split-smooth}
    \mK = \mK(n = n') + \mK(n \neq n'),
\end{equation}
where after a rescaling of the $e$ variable,
\[
    \mK(n = n') := 
    \sum_{n \sim N}
    \left \vert
    \sum_{\substack{1 \le |t| \le T \\ (t, n) = 1}}
    \sum_{\substack{h, h' \\ e = at(h - h') \neq 0}} \phi(h, h')
    \sum_{\substack{k, \ell \\ (k, dn\ell) = 1}}
    \Phi\left(\frac{\ell}{L}, \frac{k}{K}\right)
    e\left(\frac{e\bar{dn\ell}}{k}\right) \right \vert.
\]
The dominant contribution will come from $\mK(n \neq n')$, but let us first bound the simpler sum $\mK(n = n')$. Setting $e \gets |e|$, putting $e$ and $q = dn$ in dyadic ranges and denoting
\[
    a_{e,q} := \one_{d \mid q} \sum_{\substack{1 \le |t| \le T \\ (t, q/d) = 1 \\ h,h' \in \Z \\ \pm at(h-h') = e}} \phi(h,h'),
\]
we get
\begin{equation} \label{eq:K-diag-split}
    \mK(n = n') \ll x^{o(1)} \sup_{\substack{E \ll |a|TH \\ Q \asymp dN}}
    \mK_1(E, Q),
\end{equation}
where
\[
    \mK_1 = 
    \sum_{q \sim Q}
    \left \vert
    \sum_{e \sim E} a_{e,q}
    \sum_{\substack{k, \ell \\ (k, q\ell) = 1}}
    \Phi\left(\frac{\ell}{L}, \frac{k}{K}\right)
    e\left(\frac{\pm e\bar{q\ell}}{k}\right) \right \vert.
\]
We recall that by \cref{prop:large-sieve-general}, the tuple $(q, E, 1, (a_{e,q})_{e \sim E}, \|(a_{e,q})_{e \sim E}\|_2, 1)$ satisfies \cref{ass:large-sieve}. So by \cref{prop:kloosterman-incomplete} with $S = 1$ (which uses none of our new large sieve technology in this instance), we have
\[
    \mK_1 \ll x^{o(1)} \|(a_{e,q})_{e \sim E, q \sim Q} \|_2 \left(L^2 E Q + \left(1 + \frac{K^2}{Q^2}\right)^{\theta_{\max}} K (K + LQ)(Q + E) \right)^{1/2}.
\]
Recalling that $\phi(h, h')$ is supported on $h, h' \asymp H$, we can bound $a_{e,q} \ll x^{o(1)} H$ by the divisor bound, and thus $\|a_{e,q}\|_2 \ll x^{o(1)} \sqrt{EQ} H$. The resulting bound for $\mK_1$ is non-decreasing in $E, Q$, so we can plug this into \cref{eq:K-diag-split} to bound
\[
\begin{aligned}
    \mK(n = n') &\ll_{\eps} x^{6\eps} H \sqrt{THN} \left(L^2 TH N + \left(1 + \frac{K^2}{N^2}\right)^{\theta_{\max}} K (K + LN)(N + TH) \right)^{1/2}
    \\
    &\ll 
    x^{6\eps} THN \left(\frac{L^2 H N}{T} + \frac{1}{T^2N}\left(1 + \frac{K^2}{N^2}\right)^{\theta_{\max}} K (K + LN)N^2 \right)^{1/2},
\end{aligned}
\]
where in the second line we multiplied and divided by $T$, then used the assumption $TH \ll N$. Since $\theta_{\max} \le 1/3$, we have 
\[
    \frac{1}{T^2N}\left(1 + \frac{K^2}{N^2}\right)^{\theta_{\max}}
    \ll 
    \left(1 + \frac{K^2}{N^5 T}\right)^{\theta_{\max}}
    \ll 
    \left(1 + \frac{K^2}{N^3 TH^2}\right)^{\theta_{\max}},
\]
so the contribution of $\mK(n = n')$ is acceptable in \cref{eq:consequence-incomplete}.

To bound $\mK(n \neq n')$, we let $n_0 := (n, n')$, substitute $(n, n', e) \gets (n_0 n, n_0 n', n_0 e)$, and use the triangle inequality in $t$ to obtain
\[
    \mK(n \neq n') \ll 
    \sum_{\substack{n_0 \le 2N \\ n, n' \sim N/n_0 \\ (n, n') = 1}}
    \sum_{\substack{1 \le |t| \le T \\ t \mid n - n' \neq 0}}
    \left\vert
    \sum_{\substack{h, h'\\ e = at(n'h - nh') \neq 0}} \phi(h, h')
    \sum_{\substack{k, \ell \\ (k, dn_0nn'\ell) = 1}}
    \Phi\left(\frac{\ell}{L}, \frac{k}{K}\right)
    e\left(\frac{e\bar{dn_0nn'\ell}}{k}\right) \right \vert.
\]
We then put $n_0$, $e \gets |e|$, and $q = dn_0 n n'$ in dyadic ranges, and use the divisor bound to write
\begin{equation} \label{eq:K-off-diag-split}
    \mK(n \neq n') \ll 
    x^{o(1)} \sup_{\substack{N_0 \ll N \\ E \ll |a| THN/N_0 \\ Q \asymp d N^2/N_0}} \mK_2(N_0, E, Q),
\end{equation}
where
\[
\begin{aligned}
    \mK_2 &:=
    \sum_{q \sim Q} \max_{\substack{n_0 \le 2N \\ n, n' \sim N/n_0 \\ (n, n') = 1 \\ 1 \le |t| \le T}}
    \left \vert
    \sum_{e \sim E} 
    \sum_{\substack{h, h' \\ at(n'h - nh') = \pm e}} \phi(h, h')
    \sum_{\substack{k, \ell \\ (k, q\ell) = 1}}
    \Phi_0\left(\frac{\ell}{L}\right)
    \Phi_0\left(\frac{k}{K}\right)  
    e\left(\frac{\pm e \bar{q\ell}}{k}\right) \right \vert
    \\
    &=
    \sum_{q \sim Q} 
    \left \vert
    \sum_{e \sim E} 
    a_{e,q}
    \sum_{\substack{k, \ell \\ (k, q\ell) = 1}}
    \Phi_0\left(\frac{\ell}{L}\right)
    \Phi_0\left(\frac{k}{K}\right)  
    e\left(\frac{\pm e \bar{q\ell}}{k}\right) \right \vert.
\end{aligned}
\]
Above, we denoted
\[
    a_{e,q} := \sum_{\substack{h, h' \in \Z \\ \pm at(q)(n'(q)h - n(q)h') = e}} \phi(h, h')
\]
if the maximum on the first line is attained at some $n(q), n'(q), t(q)$; if the maximum is empty, we let $a_{e,q} = 0$. Then by \cref{prop:large-sieve-additive}, we know that $(q, E, x, (a_{e,q})_{e \sim E}, A_q, Y)$ satisfies \cref{ass:large-sieve}, where 
\[
    Y := \frac{E H}{|a|(H+N/N_0)(N/N_0) \min_i T_H(\alpha_i)},
    \quad
    A_q := \left(\sum_{e \sim E} |a_{e,q}|^2\right)^{1/2} + \sqrt{TE} \sqrt{\frac{H N_0}{N} + \frac{H^2 N_0^2}{N^2}}.
\]
Since $\min_i T_H(\alpha_i) \ll x^{2\eps}$,
we further have
\[
    Y \gg_\eps x^{-2\eps}\frac{E H N_0}{|a|(H+N)N}.
\]
From \cref{prop:kloosterman-incomplete}, we conclude that 
\[
    \mK_2 \ll_{\eps,a} x^{2\eps} \|A_q\|_2 \left(L^2 E Q + \left(1 + \frac{K^2}{Q^2 \frac{EHN_0}{(H+N)N}}\right)^{\theta_{\max}} K (K + LQ)(Q + E) \right)^{1/2}.
\]
Now by the same computation as in \cref{eq:primes-l2-bound} (incorporating a sum over $1 \le |t| \le T$, $t \mid e$), we have
\[
    \|A_q\|_2^2 \ll_\eps x^{2\eps} T E(HN + H^2 N_0),
\]
so that (using $|a| \le x^{\eps}$)
\[
    \mK_2 \ll_{\eps} x^{4\eps} \sqrt{T E(HN + H^2 N_0)} \left(L^2 E Q + \left(1 + \frac{K^2 (H+N)N}{Q^2 EHN_0}\right)^{\theta_{\max}} K (K + LQ)(Q + E) \right)^{1/2}.
\]
Since this right-hand side is non-decreasing in $E$ (due to $\theta_{\max} \le 1$), we may use the bounds $E \ll |a|THN/N_0$, $Q \asymp dN^2/N_0$ from \cref{eq:K-off-diag-split}, and $d \le x^{2\eps}$ to obtain 
\[
\begin{aligned}
    &\mK_2 \ll_{\eps} x^{5\eps} \sqrt{T \frac{THN}{N_0}(HN + H^2 N_0)} 
    \\
    &\times \left(L^2 \frac{THN}{N_0} \frac{N^2}{N_0} + \left(1 + \frac{K^2 (H+N)N}{\left(\frac{N^2}{N_0}\right)^2 \frac{THN}{N_0} HN_0}\right)^{\theta_{\max}} K \left(K + L\frac{N^2}{N_0}\right)\left(\frac{N^2}{N_0} + \frac{THN}{N_0}\right) \right)^{1/2}.
\end{aligned}
\]
Since $\theta_{\max} \le 1/2$, this bound is seen to be non-increasing in the $N_0 \gg 1$ parameter; plugging this into \cref{eq:K-off-diag-split} and using the assumption $TH \ll N$, we conclude that
\[
\begin{aligned}
    \mK(n \neq n') \ll_{\eps} x^{6\eps} THN 
    \left(L^2 TH N^3 + \left(1 + \frac{K^2}{N^3 TH^2}\right)^{\theta_{\max}} K \left(K + LN^2\right)N^2 \right)^{1/2},
\end{aligned}
\]
which gives the right-hand side of \cref{eq:consequence-incomplete}.
\end{proof}

We now deduce a power-saving bound for an exponential sum as in \cref{eq:sketch-kl-fractions} (before passing to the complementary divisor), which improves the first set of conditions in \cite[Proposition 1]{drappeau2015theoremes}.

\begin{lemma}[Exponential sum bound for convolutions] \label{lem:expo-bound-convo}
Let $\eps > 0$ be small enough, $a \in \Z \setminus \{0\}$, $v, d_1, d_2 \in \Z_+$, $\theta := 7/32$, and $1 \ll M, K, N, L, H, R \ll x$ satisfy 
\begin{equation} \label{eq:expo-convo-ranges}
\begin{gathered}    
    |avd_1d_2| \ll x^\eps, \qquad NL \ll x^\eps K, \qquad R \ll K \ll \min(x^{-3\eps} MN, LN^2),
    \qquad 
    H \ll x^\eps \frac{R}{M},
    \\
    K \ll x^{-25\eps} \sqrt{M N R}, \qquad\quad 
    K^{3+\theta} N^{2-3\theta} \ll x^{-200\eps} M^{2-2\theta} R^{2+\theta} L.
\end{gathered}
\end{equation}
Let $(u_k)_{K < k \le 4K}, (\beta_n)_{n \sim N}, (\lambda_\ell)_{\ell \sim L}$ be complex sequences such that $|u_k| \le \tau(k)$, $|\beta_n| \le 1$, $|\lambda_\ell| \le 1$, and
\[
    (k, vd_1d_2) > 1 \ \Rightarrow \ u_k = 0,
    \qquad\qquad 
    (n\ell, vd_1) > 1 \ \Rightarrow \ \beta_n \lambda_\ell = 0.
\]
Then for any smooth functions $\Phi(t)$, $\Psi(t)$ supported in $t \asymp 1$ with $\Phi^{(j)}, \Psi^{(j)}, \ll_j 1$, and any $\omega \in \R/\Z$ with $T_H(\omega) \ll x^\eps$, one has
\begin{equation} \label{eq:expo-bound-convo}
\begin{aligned}
    \sum_{\substack{r \sim R \\ (r, avd_1 d_2) = 1}}
    \frac{M}{r} \Psi\left(\frac{r}{R}\right)
    \sum_{\substack{k,n,\ell \\ d_1 k \equiv d_2 n\ell \pmod{r} \\ (d_1 k, d_2 n \ell) = 1}}
    u_k \beta_n \lambda_\ell 
    \sum_{h \in \Z} e(h\omega)\, \Phi\left(\frac{h}{H}\right) 
    e\left(\frac{-ha\bar{vd_1d_2k}}{r}\right)
    \\
    \ll_{\eps} 
    x^{-10\eps} \frac{KMNL}{R}.
\end{aligned}
\end{equation}
\end{lemma}

\begin{remark}
As is common for exponential sum estimates with a variable $h \sim H$ coming from Poisson summation, \cref{lem:expo-bound-convo} needs to win a factor of $H$ (times an extra $x^\eps$) over the trivial bound; the same was true for \cref{lem:expo-bound-well-fact}.
\end{remark}

\begin{proof}[Proof of \cref{lem:expo-bound-convo}]
We closely follow the proof of \cite[Proposition 1]{drappeau2015theoremes}. We denote the exponential sum considered in \cref{eq:expo-bound-convo} by $\mR$; it is essentially identical to the sum in \cite[Section 3.5]{drappeau2015theoremes}, except that $h$ lies in a smooth dyadic range. As in \cite[Section 3.5]{drappeau2015theoremes}, we denote
\[
    \nu := vd_1 d_2 \ll x^\eps
    \qquad\quad 
    \text{and} 
    \qquad\quad 
    T := \frac{\max(d_1 K, d_2 NL)}{R} \ll x^{2\eps} \frac{K}{R}.
\] 
Following through the computations in \cite[p.\,844--846]{drappeau2015theoremes} with minor changes, we obtain
\begin{equation} \label{eq:R-split}
    \mR \ll_\eps x^{5\eps} K M^{-1} +
    \max_{\substack{\sigma \mid a \\ w \pmod{v}}} \left( x^{5\eps} M R^{-1} (KLT)^{1/2} \mB^{1/2} \right),
\end{equation}
where
\[
\begin{aligned}
    \mB := 
    \sum_{n, n' \sim N}
    \left \vert
    \sum_{\substack{1 \le |t| \le T \\ (t, nn') = 1 \\ t \mid n - n'}}
    \sum_{\ell}
    \Phi_0\left(\frac{\ell}{L}\right) 
    \sum_{(k, \nu d_2 nn'\ell) = 1}
    \Phi_0\left(\frac{k}{K}\right) 
    \sum_{h, h' \in \Z} \phi(h, h')\, 
    e\left(a t (n'h - nh') \frac{\bar{\nu d_2 nn'\ell}}{k}\right) \right \vert,
\end{aligned}
\]
with 
\[
    \phi(h, h') := \Phi\left(\frac{h}{H}\right) \bar{\Phi\left(\frac{h'}{H}\right)} e\left((h-h')\omega' \right),
    \qquad\quad 
    \omega' := \omega + \frac{a \bar{w}}{v}.
\]
This corresponds to the sum on top of \cite[p.\,847]{drappeau2015theoremes}; note that we broke up the coefficients $\beta(n, h)$ in \cite[p.\,846]{drappeau2015theoremes} and ignored the phases in $n, n'$ via absolute values. We note at this point that by \cref{eq:tn},
\[
\begin{aligned}
    T_H(\omega') &\le \min_{t \in \Z_+} \left(tv + H\|tv\omega'\|\right)
    \\
    &= \min_{t \in \Z_+} \left(tv + H\|tv\omega\|\right)
    \\
    &\le \min_{t \in \Z_+} v\left(t + H\|t\omega\|\right)
    =
    v T_H(\omega) \ll x^{2\eps}.
\end{aligned}
\]
Letting $e := at(n'h - nh')$, the contribution of $e = 0$ is bounded by
\[
    \mB(e = 0) \ll_\eps x^{\eps} K L N H T,
\]
just as in \cite[(3.24)]{drappeau2015theoremes}. Since by \cref{eq:expo-convo-ranges},
\[
    TH \ll x^{3\eps} \frac{K}{R} \frac{R}{M} \ll x^{3\eps} \frac{K}{M} \ll N,
\] 
\cref{lem:consequence-incomplete} applies directly to the contribution of $e \neq 0$, giving
\[
    \mB(e \neq 0) \ll_{\eps} x^{6\eps} THN 
    \left(L^2 TH N^3 + \left(1 + \frac{K^2}{N^3 TH^2}\right)^\theta K \left(K + LN^2\right)N^2 \right)^{1/2}.
\]
Plugging these bounds and $K \le LN^2$ (from \cref{eq:expo-convo-ranges}) into \cref{eq:R-split}, we obtain
\[
\begin{aligned}
    \mR &\ll_{\eps} x^{5\eps} K M^{-1} +
    x^{10\eps} M R^{-1} \sqrt{KLT} 
    \\
    &\times \left(\sqrt{KLNHT} + \sqrt{THN} 
    \left(L^2 TH N^3 + \left(1 + \frac{K^2}{N^3 TH^2}\right)^\theta K L N^4 \right)^{1/4}\right).
\end{aligned}
\]
Combining $TH \ll N$ with $L \le NL \ll x^\eps K$ (from \cref{eq:expo-convo-ranges}), we see that
\[
    L^2 T H N^3 \ll L^2 N^4 \ll x^\eps KLN^4,
\]
so 
\[
    \mR \ll_{\eps} x^{5\eps} K M^{-1} +
    x^{11\eps} M R^{-1} \sqrt{KLT} \left(\sqrt{KLNHT} + \sqrt{THN} 
    \left(1 + \frac{K^2}{N^3 TH^2}\right)^{\theta/4} \left(K L N^4 \right)^{1/4} \right).
\]
Since this bound is non-decreasing in $H$ and $T$, we can plug in $H \le x^\eps R/M$ (from \cref{eq:expo-convo-ranges}) and $T \le x^{2\eps} K/R$ to obtain
\[
\begin{aligned}
    \mR 
    &\ll_{\eps} x^{5\eps} K M^{-1} +
    x^{15\eps} \frac{M}{R} \sqrt{\frac{K^2L}{R}} \left(K\sqrt{\frac{LN}{M}} + \sqrt{\frac{KN}{M}} 
    \left(1 + \frac{K M^2}{N^3 R}\right)^{\theta/4} \left(K L N^4 \right)^{1/4}\right)
    \\
    &=
    x^{5\eps} K M^{-1} + 
    x^{15\eps} \frac{K^2 L \sqrt{MN}}{R^{3/2}}
    +
    x^{15\eps} \frac{K^{7/4} L^{3/4} M^{1/2} N^{3/2}}{R^{3/2}} \left(1 + \frac{K M^2}{N^3 R}\right)^{\theta/4}.
\end{aligned}
\]
This is acceptable in \cref{eq:expo-bound-convo} (i.e., $\ll_{\eps} x^{-10\eps} KMNL/R$) provided that
\[
\begin{gathered}
    R \ll x^{-15\eps} M^2 N L,
    \qquad\qquad 
    K \ll x^{-25\eps} \sqrt{M N R},
    \\
    K^{3} N^{2} \ll x^{-100\eps} M^2 R^2 L,
    \qquad\qquad
    K^{3+\theta} N^{2-3\theta} \ll x^{-100\eps} M^{2-2\theta} R^{2+\theta} L.
\end{gathered}
\]
The first of these conditions follows easily from $R \ll K \ll x^{-25\eps} \sqrt{MNR}$, while the second and fourth conditions are part of \cref{eq:expo-convo-ranges}. It remains to verify the third condition which can be deduced from \cref{eq:expo-convo-ranges} as follows:
\[
\begin{aligned}
    K^3 N^2 
    =
    \left(K^{3-3\theta} N^{2-2\theta}\right)^{\frac{1}{1-\theta}}
    &\ll 
    \left(\left(\frac{K}{R}\right)^{3\theta} \left(\frac{x^\eps K}{NL}\right)^{\theta} K^{3-3\theta} N^{2-2\theta}\right)^{\frac{1}{1-\theta}}
    \\
    &= 
    \left( \frac{x^{\theta\eps}}{R^{3\theta} L^\theta} K^{3+\theta} N^{2-3\theta}\right)^{\frac{1}{1-\theta}}
    \\
    &\ll 
    \left( \frac{1}{R^{3\theta} L^\theta} x^{(\theta-200)\eps} M^{2-2\theta} R^{2+\theta} L\right)^{\frac{1}{1-\theta}}
    \\
    &\ll 
    x^{-100\eps}
    \left(M^{2-2\theta} R^{2-2\theta} L^{1-\theta} \right)^{\frac{1}{1-\theta}}
    =
    x^{-100\eps}
    M^2 R^2 L.
\end{aligned}
\]
This completes our proof.
\end{proof}

We can now deduce an estimate on the equidistribution in arithmetic progressions of convolutions of three sequences, corresponding to \cref{eq:conv-estimate-smooth} and improving \cite[Th\'eor\`eme 3]{drappeau2015theoremes}. For $r \in \Z_+$ and $k \pmod{r}$, we recall Drappeau's notation 
\begin{equation} \label{eq:small-cond-chars}
    \omega_\eps(k; r) := \sum_{\substack{\chi \text{ primitive} \\ \cond(\chi) \le x^\eps \\ \cond(\chi) \mid r}} \chi(k)
\end{equation}
from \cite[(3.1)]{drappeau2015theoremes}. Separating all the Dirichlet characters of conductors $\le x^\eps$ was crucial to obtaining power-saving convolution estimates in \cite{drappeau2015theoremes}. In a certain sense, $\tfrac{\one_{(k, r) = 1}}{\varphi(r)} \omega_\eps(k; r)$ gives a better approximation to the function $\one_{k \equiv 1 \pmod{r}}$ than the crude $\tfrac{\one_{(k, r) = 1}}{\varphi(r)}$; indeed, $\tfrac{\one_{(k, r) = 1}}{\varphi(r)} \omega_\eps(k; r)$ interpolates between the latter two quantities as $\eps$ varies in $[0, \infty)$.

The relevant constraints on the ranges of the convolved sequences are gathered in \cref{eq:triple-convo-ranges}. We note that the conditions on the top row of \cref{eq:triple-convo-ranges} also appear in \cite[Th\'eor\`eme 3]{drappeau2015theoremes}.

\begin{proposition}[Triple convolution estimate] \label{prop:triple-convo}
For any small enough $\eps > 0$, there exists $\delta > 0$ such that the following holds. Let $M, N, L \gg 1$, $x := MNL$, $a_1, a_2 \in \Z \setminus \{0\}$ satisfy $|a_1a_2| \le x^\delta$, $(a_1, a_2) = 1$, and $(\alpha_m)_{m \sim M}, (\beta_n)_{n \sim N}, (\gamma_\ell)_{\ell\sim L}$ be $1$-bounded complex sequences. Suppose that with $\theta := 7/32$, one has
\begin{equation} \label{eq:triple-convo-ranges}
\begin{gathered}
    x^\eps \le N, \qquad 
    NL \le x^{2/3-5\eps}, \qquad 
    L \le x^{-\eps} M, \qquad
    M \le R \le x^{-\eps} NL, \qquad 
    N^2 L^3 \le x^{1-\eps} R,
    \\
    N^{7-4\theta} L^{4-\theta} \le x^{2-2\theta-\eps} R^{2+\theta}.
\end{gathered}
\end{equation}
Then one has
\[
    \sum_{\substack{r \sim R \\ (r, a_1 a_2) = 1}} 
    \left\vert \sum_{\substack{m \sim M \\ n \sim N \\ \ell \sim L}} \alpha_m  \beta_n 
    \gamma_\ell 
    \left(\one_{mn\ell \equiv a_1 \bar{a_2} \pmod{r}} - \frac{\one_{(mn\ell, r) = 1}}{\varphi(r)} \omega_\eps(mn\ell \bar{a_1} a_2; r) \right)
    \right\vert
    \ll_{\eps,a_1,a_2} x^{1-\delta}.
\]
\end{proposition}

\begin{remark}
The inequalities $R \le x^{-o(1)} NL$, $N^2 L^3 \le x^{1-o(1)} R$, and $N^{7-4\theta} L^{4-\theta} \le x^{2-2\theta-o(1)} R^{2+\theta}$ imply $R \le x^{5/8 - o(1)}$, which corresponds to the level from \cref{thm:smooth}. We note that for $R = x^{-o(1)} NL$, the last inequality in \cref{eq:triple-convo-ranges} is equivalent to $N^7 L^4 \le x^{-o(1)} R^2$, which explains why our final exponent of distribution does not depend on the $\theta$ parameter.
\end{remark}

\begin{proof}[Proof of \cref{prop:triple-convo}]
We closely follow the proof of \cite[Th\'ero\`eme 3]{drappeau2015theoremes}, which applies Cauchy--Schwarz in $r, m$ (and inserts a smooth majorant $f(m) = \Phi(m/M)$) to obtain three dispersion sums \cite[Section 3.1]{drappeau2015theoremes}. We change nothing in the treatment of the second and third dispersion sums from \cite[Sections 3.2, 3.3]{drappeau2015theoremes}, noting that the conditions on the top row of \cref{eq:triple-convo-ranges} are sufficient here.

We also begin treating the first dispersion sum similarly as in \cite[Section 3.4]{drappeau2015theoremes}, with the technical change that we Poisson complete via \cref{lem:truncated-poisson} rather than \cite[Lemme 2]{drappeau2015theoremes}. Instead of \cite[(3.12)]{drappeau2015theoremes}, we thus obtain
\[
    \mS_1 = \hat{f}(0)X_1 + \sum_{\substack{v, d_1,e_1,e_2 \le x^\eta \\ d_1 \mid v^\infty, e_1 \mid a_2^\infty, e_2 \mid a_2}} \int_{0.1}^{10} \sum_{\substack{H_j = 2^j \\ 1 \le H_j \le H}} R_{j,u}(v; d_1, e_1, e_2)\, \frac{du}{u} + O_\eta(x^{1-\eta/4} K R^{-1}),
\]
where $X_1$ is the main term from \cite[(3.12)]{drappeau2015theoremes}, $H = x^\eta R M^{-1}$, $\Psi_j$ are as in \cref{lem:truncated-poisson}, and
\[
    R_{j,u}(v; d_1, e_1, e_2) := \sum_{\substack{r \sim R \\ (r, a_1a_2v) = 1}} \frac{M}{r} \tilde\Phi\left(\frac{ur}{R}\right) \sum_{(k_1, k_2) \in \mK} u_{k_1} \bar{u_{k_2}} \sum_{h \in \Z} e\left(h\omega\right) \Psi_j\left(\frac{|h|}{H_j}\right) e\left(\frac{-ha_1\bar{a_2k_1}}{r}\right),
\]
where $u_k := \sum_{n\ell = k} \beta_n \gamma_\ell$, $\mK$ is as in \cite[p.\,841]{drappeau2015theoremes}, and
\[
    \omega := \frac{uM}{R} \ll H^{-1} x^\eta
    \qquad\quad 
    \Rightarrow 
    \qquad\quad 
    T_H(\omega) \ll x^\eta.
\]
We then develop and bound $R_{j,u}$ as in \cite[p.\,843]{drappeau2015theoremes}, with the only major change that we use our \cref{lem:expo-bound-convo} instead of \cite[Proposition 1]{drappeau2015theoremes}. To apply \cref{lem:expo-bound-convo} (with $\eta > 0$ in place of $\eps$), we need to verify the conditions in \cref{eq:expo-convo-ranges}; thus instead of the third-to-last display on \cite[p.\,843]{drappeau2015theoremes}, we require that $|a_1 a_2| \le x^{\eta/10}$ and
\[
\begin{gathered}    
    R \ll x^{-100\eta} NL, \qquad\quad 
    L \ll x^{-100\eta} M, \qquad\quad 
    1 \ll x^{-100\eta} N,
    \\
    \sqrt{N} L \ll x^{-200\eta} \sqrt{M R}, \qquad\quad 
    N^{5-2\theta} L^{2+\theta} \ll x^{-300\eta} M^{2-2\theta} R^{2+\theta}.
\end{gathered}
\]
Here, we implicitly used that in Drappeau's computations near \cite[bottom of p.\,843]{drappeau2015theoremes}, one has $vd_1d_2 \ll x^{5\eta}$, $H = x^\eta R M^{-1}$, and $x^{10\eta} NL \ll K \ll x^{-10\eta} NL$. Since $MNL \asymp x$, these conditions follow from \cref{eq:triple-convo-ranges} provided $\eta$ is chosen sufficiently small in terms of $\eps$.
\end{proof}

Finally, we prove a direct generalization of \cref{thm:smooth}, in a form analogous to \cite[Th\'eor\`eme 1]{drappeau2015theoremes}. We recall the notation specific to smooth numbers,
\[
    u := \frac{\log x}{\log y},
    \qquad\quad 
    H(u) := \exp\left(\frac{u}{(\log (u+1))^2}\right),
\]
from \cite{drappeau2015theoremes}, as well as the definitions of $\Psi_q(x, y)$ and $\Psi(x, y; a, q)$ from \cref{eq:smooth-notations}.

\begin{theorem}[Smooth numbers in APs to large moduli, refined] \label{thm:smooth-refined}
For any $\eps > 0$, there exist $\delta, C > 0$ such that the following holds.
Let $x \ge 2$ and $a_1, a_2 \in \Z \setminus \{0\}$ satisfy $(a_1, a_2) = 1$ and $|a_1a_2| \le x^\delta$. Then for any $y \in [(\log x)^C, x^{1/C}]$ and $A \ge 0$, one has
\[
    \sum_{\substack{q \le x^{5/8-\eps} \\ (q, a_1 a_2) = 1}}
    \left\vert 
    \Psi(x, y; a_1 \bar{a_2}, q) - 
    \frac{\Psi_q(x, y)}{\varphi(q)} \right\vert 
    \ll_{\eps,A} 
    \Psi(x, y)\left( H(u)^{-\delta} (\log x)^{-A} + y^{-\delta} \right).
\]
The implicit constant is effective if $A < 1$.
\end{theorem}

\begin{proof}
We assume without loss of generality that $\eps > 0$ is small enough, and choose $\eta$ to be a small multiple of $\eps$.
As in \cite[p.\,855--856]{drappeau2015theoremes}, Harper's result \cite[Lemme 5]{drappeau2015theoremes} (see also \cite{harper2012bombieri}) handles the contribution of Dirichlet characters with conductors $\le x^\eta$, so it suffices to prove the bound
\begin{equation} \label{eq:smooth-power-saving}
    \sum_{\substack{q \le x^{5/8-\eps} \\ (q, a_1 a_2) = 1}}
    \left\vert 
    \sum_{\substack{n \le x \\ P^+(n) \le y}} \left(\one_{n \equiv a_1 \bar{a_2} \pmod{q}} - 
    \frac{\one_{(n, q) = 1}}{\varphi(q)} \omega_\eta(n; a_1\bar{a_2})\right) \right\vert 
    \ll_\eps x^{1-\delta/2},
\end{equation}
for some $\delta = \delta(\eps) > 0$. Such a power-saving is enough up to a final rescaling of $\delta$, due to the bound $x^{1-\delta/2} \ll_\delta \Psi(x, y) y^{-\delta/4}$ for sufficiently large $C$ (see \cite[p.\,856]{drappeau2015theoremes}).

The proof of \cref{eq:smooth-power-saving} is completely analogous to that of \cite[Proposition 2]{drappeau2015theoremes}, except that we use our triple convolution estimate from \cref{prop:triple-convo} instead of \cite[Th\'eor\`eme 3]{drappeau2015theoremes}. 
The key point is that the indicator function of smooth numbers can be approximated by convolutions of three sequences with pre-specified ranges, due to their flexible factorization. Specifically, we rescale $\eps \gets 100\eps$, take $C = \eps^{-1}$ so that $y \le x^{1/C} \le x^\eps$, and put $q \gets r$ in dyadic ranges $r \sim R$. Then instead of the parameters on the bottom of \cite[p.\,852]{drappeau2015theoremes}, we pick
\[
    M_0 := \frac{x^{1-10\eps}}{R},
    \qquad\quad 
    N_0 := \frac{R^2}{x^{1-40\eps}},
    \qquad\quad 
    L_0 := \frac{x^{1-30\eps}}{R},
\]
in the range $x^{(1/2) - (\eps/10)} \le R \le x^{5/8-100\eps}$ (smaller values of $R$ are covered by previous results \cite{drappeau2015theoremes}). Any resulting values of $M, N, L$ with
\[
    M_0 \le M \le y \frac{M_0}{2},
    \qquad\quad 
    L_0 \le L \le y \frac{L_0}{2},
    \qquad\quad 
    y^{-2} N_0 \le N \le N_0
\]
are seen to satisfy the conditions in \cref{eq:triple-convo-ranges}. In particular, for the last two conditions in \cref{eq:triple-convo-ranges}, we note that 
\[
    \frac{x R}{N_0^2 L_0^3}
    =
    x^{10\eps}
    \qquad\quad 
    \text{and}
    \qquad\quad
    \frac{x^{2-2\theta} R^{2+\theta}}{N_0^{7-4\theta} L_0^{4-\theta}}
    =
    \frac{x^{5(1-\theta)-(160-130\theta)\eps}}{R^{8(1-\theta)}}
    \ge
    x^{100\eps},
\]
since $R \le x^{5/8-100\eps}$; this gives enough $x^{o(1)}$ room when replacing $M_0, N_0, L_0$ by $M, N, L$, since $y \le x^\eps$. Following through the combinatorial decompositions and separations of variables in \cite[p.\,852--854]{drappeau2015theoremes}, we can apply \cref{prop:triple-convo} for the sequences $(\alpha_m)^{(j)}, (\beta_m)^{(j)}, (\lambda_\ell)^{(j)}$ on \cite[p.\,854]{drappeau2015theoremes}, which recovers the desired bound.
\end{proof}

\section{Smooth numbers with weights on smooth moduli} \label{sec:smooth-weights-smooth}

Here we quickly prove a variant (and in fact, a generalization) of \cref{thm:smooth} when the sum over $q$ is restricted to smooth moduli, which improves the first exponent of distribution in \cite[Th\'eor\`eme 2.1]{de2020niveau} from $\tfrac{3}{5} -\eps$ to $\tfrac{5}{8} -\eps$. Recall again the notation from \cref{eq:smooth-notations}.

\begin{theorem}[Smooth numbers in APs to smooth moduli] \label{thm:smooth-mod-smooth}
For any $\eps, A > 0$ and $k \ge 1$, there exist $\delta, C > 0$ such that the following holds. Let $x \ge 2$ and $a_1, a_2 \in \Z\setminus \{0\}$ satisfy $(a_1, a_2) = 1$ and $|a_1a_2| \le x^\delta$. Then for any $y_1 \in [(\log x)^C, x^{1/C}]$, $y_2 \in [(\log x)^C, x]$, $Q \le x^{5/8 - \eps}$, and $q_0 \in \Z_+$ with $q_0 \le x^\delta$, $(q_0, a_1a_2) = 1$, $P^+(q_0) \le y_2$, one has
\begin{equation} \label{eq:smooth-mod-smooth}
    \sum_{\substack{q \sim Q \\ P^+(q) \le y_2 \\ (q, a_1a_2) = 1}}
    \tau_k(q)
    \left\vert 
    \Psi(x, y_1; a_1\bar{a_2}, q_0q) - 
    \frac{\Psi_{q_0q}(x, y_1)}{\varphi(q_0q)} \right\vert 
    \ll_{\eps,A,k}
    \frac{\Psi(x, y_1)}{(\log x)^A} \frac{\Psi(Q, y_2)}{\varphi(q_0)Q} e^{O_k(u_2)},
\end{equation}
where $u_2 := (\log x)/\log y_2$.
\end{theorem}

\begin{proof}
Again, we assume without loss of generality that $\eps > 0$ is small enough, and we will pick $\eta, \delta$ to be small enough in terms of $\eps, A, k$. It suffices to prove our claim with $\delta$ replaced by $\delta/10$.

Let $\mS$ denote the left-hand side of \cref{eq:smooth-mod-smooth}. Recalling the notation in \cref{eq:small-cond-chars}, we separate the contribution of Dirichlet characters of small conductors by writing
\[
    |\mS| \le \mS_{\text{small}} + \mS_{\text{large}},
\]
where
\[
\begin{aligned}
    \mS_{\text{large}} &:= \sum_{\substack{q \sim Q \\ P^+(q) \le y_2 \\ (q, a_1a_2) = 1}}
    \tau_k(q)
    \left\vert 
    \sum_{\substack{n \le x \\ P^+(n) \le y_1}} \left(\one_{n \equiv a_1\bar{a_2} \pmod{q_0q}} - \frac{\one_{(n, q_0q) = 1}}{\varphi(q_0q)} \omega_\eta(n; q_0q) \right)
    \right\vert,
    \\
    \mS_{\text{small}} &:= \sum_{\substack{q \sim Q \\ P^+(q) \le y_2 \\ (q, a_1a_2) = 1}}
    \frac{\tau_k(q)}{\varphi(q_0q)}
    \left\vert 
    \sum_{\substack{n \le x \\ P^+(n) \le y_1}} \sum_{\substack{\chi \pmod{q_0 q} \\ 1 < \cond(\chi) \le x^\eta}} \chi(n)
    \right\vert.
\end{aligned}
\]
For $\mS_{\text{small}}$, we use the triangle inequality for the sum over $\chi$, and then proceed identically as in \cite[after (2.3)]{de2020niveau}; this gives the desired bound when $\delta$ is sufficiently small and $C$ is sufficiently large. 

For $\mS_{\text{large}}$, we drop the smoothness condition on $q$, use the pointwise divisor bound $\tau_k(q) \ll_k q^{o(1)}$, group $q_0q$ into a new variable, and drop its divisibility constraint by $q_0$. Combined with \cref{eq:smooth-power-saving} (which followed from \cref{prop:triple-convo}), this gives
\[
    \mS_{\text{large}} \ll_{\eps,k} x^{1-\delta/3},
\]
provided $\delta$ is sufficiently small and $C$ is sufficiently large in terms of $\eps$. This is acceptable once $C$ is chosen to be large enough in terms of $\delta, A$, due to the bounds $x^{1-\delta/10} \ll_\delta \Psi(x, y_1) y_1^{-\delta/20}$, $Q^{1-\delta/10} \ll_\delta \Psi(Q, y_2)$, and $y_1 \ge (\log x)^C$, $q_0 \le x^{\delta/10}$.
\end{proof}


\begin{corollary}[Smooth values of factorable quadratic polynomials] \label{cor:factorable-quadratic}
For any $\eps > 0$, there exist $C, \delta > 0$ such that the following holds. Let $x \ge 2$ and $a, b, c, d \in \Z$ satisfy $(a, c) = 1$, $ad - bc \neq 0$, and $|a|, |b|, |c|, |d| \le x^\delta$. Then for any $(\log x)^C \le y_1 \le y_2 \le x$ with $y_2 \le y_1^C$, one has
\[
    \#\{n \le x : P^+(an+b) \le y_1,\ P^+(cn+d) \le y_2\} \ll_\eps \Psi(x, y_1) \varrho(u_2)^{5/8-\eps},
\]
where $u_2 := (\log x)/\log y_2$.
\end{corollary}

\begin{proof}
This is identical to the proof of \cite[Th\'eor\`eme 4.1]{de2020niveau}, using  \cref{thm:smooth-mod-smooth} instead of \cite[Th\'eor\`eme 2.1]{de2020niveau}. When applying \cref{thm:smooth-mod-smooth}, $q$ will be a divisor of $cn+d$ coming from an upper-bound sieve \cite[Proposition 3.1]{de2020niveau}, while $q_0 = a$, $a_1 = -(ad-bc)$, and $a_2 = c$; note that 
\[
    q \mid cn+d \qquad\quad \iff \qquad\quad an+b \equiv a_1\bar{a_2} \pmod{q_0q},
\]
and $P^+(an+b) \le y_1$, $P^+(q) \le y_2$.
\end{proof}

\begin{proof}[Proof of \cref{cor:consec-smooth}]
Take $(a, b, c, d) = (1, 0, 1, 1)$, $y_1 = y_2$ in \cref{cor:factorable-quadratic}, and use $\Psi(x, y_1) = x \varrho(u) e^{O(u)}$ where $u := (\log x)/\log y$  (see \cite[(1.7)]{de2020niveau} and \cite[(2.6) and (2.7)]{hildebrand1986number}).
\end{proof}

\let\oldaddcontentsline\addcontentsline
\renewcommand{\addcontentsline}[3]{}
\bibliographystyle{plain}
\bibliography{main}
\let\addcontentsline\oldaddcontentsline

\end{document}